\documentclass[10pt,a4paper]{article}
\usepackage[utf8]{inputenc}
\usepackage[american]{babel}
\usepackage[all]{xy}
\usepackage{enumitem}
\usepackage[dvipsnames]{xcolor}
\usepackage{graphicx}
\usepackage{stmaryrd}
\usepackage{amsfonts,amstext,amssymb,amsmath,amsthm,pspicture,bigints}
\usepackage{fullpage}
\usepackage{moreverb}
\usepackage{pifont}
\usepackage{pst-all}
\usepackage[left=2cm,right=2cm,top=2cm,bottom=2cm]{geometry}
\usepackage{esint}
\usepackage{fourier-orns}
\usepackage{mathrsfs}
\usepackage{array}
\newcommand{\T}{\mathbb{T}}
\usepackage{hyperref}
\usepackage{authblk}
\newcolumntype{C}[1]{>{\centering\arraybackslash}b{#1}}
\newcolumntype{R}[1]{>{\raggedleft\arraybackslash}b{#1}}
\newcolumntype{L}[1]{>{\raggedright\arraybackslash}b{#1}}
\newcolumntype{M}[1]{>{\centering}m{#1}}
\newtheorem{theo}{Theorem}[section]

\newtheorem{lem}{Lemma}[section]
\newtheorem{prop}{Proposition}[section]

\newtheorem{remark}{Remark}[section]

\usepackage{bbm}
\usepackage{todonotes}
\numberwithin{equation}{section}

\usepackage{subfigure}

\usepackage{pgfplots}
\usepackage{amsfonts,amsmath}
\usepackage{tikz-3dplot}
\pgfplotsset{compat=newest}
\usetikzlibrary{calc}
\usetikzlibrary{quotes,angles}
\usepackage{pgfplots}
\pgfplotsset{compat=1.12}
\usetikzlibrary{math}

\AtBeginDocument{%
	\def\MR#1{}
}

\date{}
\title{Dynamics of vortex cap solutions on the rotating unit sphere}

\author{Claudia Garc\'{i}a, Zineb Hassainia and Emeric Roulley}

\begin{document}

	\maketitle
	\begin{abstract}
		In this work, we analytically study the existence of periodic vortex cap solutions for the homogeneous and incompressible Euler equations on the rotating unit $2$-sphere, which was numerically conjectured in \cite{DP92,DP93,KSS18, KS21}. Such solutions are piecewise constant vorticity distributions, subject to the Gauss constraint and rotating uniformly around the vertical axis. The proof is based on the bifurcation from zonal solutions given by spherical caps. For the one--interface case, the bifurcation eigenvalues correspond to Burbea's frequencies obtained in the planar case but shifted by the rotation speed of the sphere. The two--interfaces case (also called band type or strip type solutions) is more delicate. Though,  for any fixed large enough symmetry, and under some non-degeneracy conditions to avoid spectral collisions, we achieve the existence of at most two branches of bifurcation. 
		
	\end{abstract}
	\tableofcontents
	
	\section{Introduction}
	The purpose of this paper is to give an analytical proof of  the emergence of uniformly rotating (around the $z$-axis) vortex caps with $\mathbf{m}$-fold symmetries ($\mathbf{m}\in\mathbb{N}^*$) for the incompressible Euler equations set on the rotating unit sphere $\mathbb{S}^2$ defined by
	$$\mathbb{S}^2\triangleq\Big\{(x,y,z)\in\mathbb{R}^3\quad\textnormal{s.t.}\quad x^2+y^2+z^2=1\Big\}.$$
	In particular, we shall implement bifurcation techniques in order to find non-trivial vortex cap solutions close to the trivial stationary flat (spherical) ones, which was numerically conjectured in \cite{DP92, DP93, KS21, KSS18}. In this introduction we present the model of interest, discuss some historical background, derive the contour dynamics equations, expose our results and give the organization of this work.
	\subsection{Euler equations on the rotating unit sphere}
	This study deals with the homogeneous incompressible Euler equations on the two dimensional unit sphere in rotation around the vertical axis. Such a model, sometimes called \textit{barotropic}, is commonly used in geophysical fluid dynamics for meteorological predictions or to study the motion of planets' atmosphere. In order to present the model, we shall first recall some basic notions in differential calculus/geometry. The set $\mathbb{S}^2$ is endowed with a smooth manifold structure described by the following two charts
	\begin{align*}
		&\begin{array}[t]{rcl}
			C_1:(0,\pi)\times(0,2\pi) & \rightarrow & \mathbb{R}^3\\
			(\theta,\varphi) & \mapsto & \big(\sin(\theta)\cos(\varphi)\,,\,\sin(\theta)\sin(\varphi)\,,\,\cos(\theta)\big),
		\end{array}\\
		&\begin{array}[t]{rcl}
			C_2:(0,\pi)\times(0,2\pi) & \rightarrow & \mathbb{R}^3\\
			(\vartheta,\phi) & \mapsto & \big(-\sin(\vartheta)\cos(\phi)\,,\,-\cos(\vartheta)\,,\,-\sin(\vartheta)\sin(\phi)\big).
		\end{array}
	\end{align*}
	For our purpose, we shall mainly work in the chart $C_1$ where the variables $\theta$ and $\varphi$ are called \textit{colatitude} and \textit{longitude}, respectively. Notice that the physical literature mentioned above rather considers the latitude/longitude convention, but we found more convenient to work with the other one.
	In the colatitude/longitude chart $C_1$, the Riemannian metric of $\mathbb{S}^2$ is given by
	\begin{equation}\label{metric S2}
		\mathtt{g}_{\mathbb{S}^2}(\theta,\varphi)\triangleq d\theta^2+\sin^2(\theta)d\varphi^2.
	\end{equation}
	Therefore, denoting $N$ and $S$ the north and south poles, we have that for any $p\in\mathbb{S}^2\setminus\{N,S\},$ an orthonormal basis of the tangent space $T_{p}\mathbb{S}^2$ is given by
	$$\mathbf{e}_{\theta}\triangleq\partial_{\theta},\qquad\mathbf{e}_{\varphi}\triangleq\frac{1}{\sin(\theta)}\partial_{\varphi}.$$
	We have used the classical identification between tangent vectors and directional differentiations. In these coordinates, the Riemannian volume is given by
	$$d\sigma=\sin(\theta
	)d\theta d\varphi.$$
	Therefore, for any function $\mathtt{f}:\mathbb{S}^2\rightarrow\mathbb{R},$ we define 
	\begin{equation}\label{conv int S2}
		f(\theta,\varphi)\triangleq\mathtt{f}\big(C_1(\theta,\varphi)\big),\qquad\int_{\mathbb{S}^2}\mathtt{f}(\xi)d\sigma(\xi)\triangleq\int_{0}^{2\pi}\int_{0}^{\pi}f(\theta,\varphi)\sin(\theta)d\theta d\varphi.
	\end{equation}
	In the sequel, with a small abuse of notation, we shall denote $f$  for both $\mathtt{f}$ or $f$ with no possible confusion according to the context, since one is in the cartesian variables $\xi$ and the other one is in the spherical coordinates $(\theta,\varphi).$ The gradient of $f$ is defined as follows
	$$\nabla f(\theta,\varphi)\triangleq\partial_{\theta}f(\theta,\varphi)\mathbf{e}_{\theta}+\frac{\partial_{\varphi}f(\theta,\varphi)}{\sin(\theta)}\mathbf{e}_{\varphi}.$$
	Similarly, we define its orthogonal as
	$$\nabla^{\perp}f(\theta,\varphi)\triangleq J\nabla f(\theta,\varphi),\qquad\underset{(\mathbf{e}_{\theta},\mathbf{e}_{\varphi})}{\textnormal{Mat}}(J)=\begin{pmatrix}
		0 & 1\\
		-1 & 0
	\end{pmatrix}.$$
	The Laplace-Beltrami operator applied to $f$ is defined by
	$$\Delta f(\theta,\varphi)\triangleq\frac{1}{\sin(\theta)}\partial_{\theta}\big[\sin(\theta)\partial_{\theta}f(\theta,\varphi)\big]+\frac{1}{\sin^2(\theta)}\partial_{\varphi}^2f(\theta,\varphi).$$
	For a vector field $U(\theta,\varphi)=U_{\theta}(\theta,\varphi)\mathbf{e}_{\theta}+U_{\varphi}(\theta,\varphi)\mathbf{e}_{\varphi},$ the divergence expresses as
	$$(\nabla\cdot U)(\theta,\varphi)\triangleq\frac{1}{\sin(\theta)}\partial_{\theta}\big[\sin(\theta)U_{\theta}(\theta,\varphi)\big]+\frac{1}{\sin(\theta)}\partial_{\varphi}U_{\varphi}(\theta,\varphi).$$
	The incompressible Euler equations on the rotating sphere $\mathbb{S}^2$ with angular rotation speed $\widetilde{\gamma}$ are given by
	\begin{equation}\label{Euler eq on S2}
		(E_{\widetilde{\gamma}})\begin{cases}
			\partial_{t}\Omega(t,\theta,\varphi)+U(t,\theta,\varphi)\cdot\nabla\big(\Omega(t,\theta,\varphi)-2\widetilde{\gamma}\cos(\theta)\big)=0,\\
			U(t,\theta,\varphi)=\nabla^{\perp}\Psi(t,\theta,\varphi),\\
			\Delta\Psi(t,\theta,\varphi)=\Omega(t,\theta,\varphi).
		\end{cases}
	\end{equation}
	The reader is referred to \cite[Sec. 13.4.1]{HH13} and \cite{S17} (see also \cite{T16}) for a rather complete introduction to these equations. We mention that the term $-2\widetilde{\gamma}\,U(t,\theta,\varphi)\cdot\nabla\cos(\theta)$ corresponds to the Coriolis force coming from the rotation of the sphere. Additionally, the equations \eqref{Euler eq on S2} must be completed by the physical impermeability condition
	 $$\forall\varphi\in[0,2\pi],\quad U_{\theta}(0,\varphi)=0=U_{\theta}(\pi,\varphi).$$
	 In the sequel, we shall work with the following quantity called \textit{absolute vorticity} 
	\begin{equation}\label{def abs vort}
		\overline{\Omega}(t,\theta,\varphi)\triangleq\Omega(t,\theta,\varphi)-2\widetilde{\gamma}\cos(\theta).
	\end{equation}
	The second equation in \eqref{Euler eq on S2} states that the velocity field $U$ is divergence-free. Then, the divergence theorem implies the following so-called \textit{Gauss constraint}
	\begin{equation}\label{null avrg vort}
		\forall t\geqslant0,\quad\int_{\mathbb{S}^2}\overline{\Omega}(t,\xi)d\sigma(\xi)=\int_{\mathbb{S}^2}\Omega(t,\xi)d\sigma(\xi)=0.
	\end{equation}
	Notice that the first equality above is justified by
	$$\int_{\mathbb{S}^2}\big[\overline{\Omega}(t,\xi)-\Omega(t,\xi)\big]d\sigma(\xi)=-4\pi\widetilde{\gamma}\int_{0}^{\pi}\cos(\theta)\sin(\theta)d\theta=\big[\pi\widetilde{\gamma}\cos(2\theta)\big]_0^{\pi}=0.$$
	According to \cite{BD15}, the stream function $\Psi$ can be computed from the vorticity $\Omega$ through the following integral representation
	\begin{equation}\label{psi-G}
		\Psi(t,\xi)=\int_{\mathbb{S}^2}G(\xi,\xi')\Omega(t,\xi')d\sigma(\xi'),\qquad G(\xi,\xi')\triangleq\frac{1}{2\pi}\log\left(\frac{|\xi-\xi'|_{\mathbb{R}^3}}{2}\right),
	\end{equation}
	where
	$|\cdot|_{\mathbb{R}^3}$ is the usual Euclidean norm in $\mathbb{R}^3.$ In the colatitude/longitude coordinates, we have
	\begin{equation}\label{Green sph}
		G(\theta,\varphi,\theta',\varphi')=\frac{1}{4\pi}\log\Big(1-\cos(\theta)\cos(\theta')-\sin(\theta)\sin(\theta')\cos(\varphi-\varphi')\Big)-\frac{\log(2)}{4\pi}\cdot
	\end{equation}
	In what follows, we shall denote for simplicity
	\begin{equation}\label{def D}
		D(\theta,\theta',\varphi,\varphi')\triangleq 1-\cos(\theta)\cos(\theta')-\sin(\theta)\sin(\theta')\cos(\varphi-\varphi').
	\end{equation}
	Observe that we can write
	\begin{align}
		D(\theta,\theta',\varphi,\varphi')&=1-\cos(\theta-\theta')+\sin(\theta)\sin(\theta')\big(1-\cos(\varphi-\varphi')\big)\nonumber\\
		&=2\Big[\sin^2\left(\tfrac{\theta-\theta'}{2}\right)+\sin(\theta)\sin(\theta')\sin^2\left(\tfrac{\varphi-\varphi'}{2}\right)\Big].\label{expr D}
	\end{align}
	The above function $D$ will play an important role since it describes the singularity of the integral operator defining the stream function. Indeed, with this last expression, we recover that $$D(\theta,\theta',\varphi,\varphi')\geqslant0\qquad \textnormal{and}\qquad D(\theta,\theta',\varphi,\varphi')=0\quad\Leftrightarrow\quad\theta=\theta'\textnormal{ and }\big(\varphi=\varphi' \textnormal{ or } \theta\in\{0,\pi\}\textnormal{ or }\theta'\in\{0,\pi\}\big).$$
	Notice that once one works with the colatitude/longitude coordinates instead of the physical ones, the Euclidean norm is deformed implying that the Green kernel is anisotropic and the north and south poles are degenerating points. This will generate extra complexity later when dealing with the regularity of the stream function in the new coordinates.
	\pgfdeclarelayer{foreground}
	\pgfsetlayers{main,foreground}
	\begin{figure}[!h]
		\begin{center}
			\tdplotsetmaincoords{80}{100}
			\begin{tikzpicture}[tdplot_main_coords,scale=2.37]
				\draw[thick,->] (0,0,0) -- (3,0,0) node[anchor=north east]{$x$};
				\draw[thick,->] (0,0,0) -- (0,1.2,0) node[anchor=north west]{$y$};
				\draw[thick,->] (0,0,0) -- (0,0,1.2) node[anchor=south]{$z$};
				
				\def\h{1.1}
				\coordinate (O) at (0,0,0);
				\coordinate (z) at (0,0,\h);
				\coordinate (x) at (\h,0,0);
				\coordinate (y) at (0,\h,0);
				\coordinate (a) at (0.6,0.6,0.53);
				\coordinate (b) at (0.6,0.6,0);
				
				\shade[tdplot_screen_coords, ball color = gray!40, opacity=0.4] (0,0) circle (1);
				\node at (0.5,-0.5,0.6)[above] {$\mathbb{S}^2$};
				\draw[gray!100,dashed] (0,0) ellipse(1cm and 0.4cm);

				\begin{pgfonlayer}{foreground}
					\draw[->] (z)+(-40:0.25) arc (-40:170:0.25);
				\end{pgfonlayer}
				\node at (0,0.3,1.15) {$\widetilde{\gamma}$};
				
				\draw (0.7,0.75,0.7) coordinate (a) node[left] {$C_1(\theta,\varphi)$};
				
				\draw[black,dashed] (0,0,0)--(0.6,0.6,0.53);
				\draw[black,dashed] (0.6,0.6,0)--(0.6,0.6,0.53);
				\draw[black,dashed] (0,0,0)--(0.6,0.6,0);
				\draw (0.6,0.6,0.53) node {$\bullet$};
				\draw[->,black,thick] (0.6,0.6,0.53)--(0.8,0.75,0.4);
				\node at (0.8,0.75,0.4) [right] {$\mathbf{e}_\theta$};
				\draw[->,black,thick] (0.6,0.6,0.53)--(0.5,0.7,0.58);
				\node at (0.5,0.7,0.58) [right] {$\mathbf{e}_\varphi$};
				\node at (0.8,0.2,0) {$\varphi$};
				\node at (0,0.1,0.25) {$\theta$};
				
				\draw[->,>=stealth] (0.5,0,0) arc [start angle=0,end angle=90,x radius=0.25,y radius=0.24];
				\draw[->,>=stealth] (0,0,0.2) arc [start angle=0,end angle=45,x radius=-0.3,y radius=0.3];
			\end{tikzpicture}
			\caption{Convention colatitude/longitude for spherical coordinates.}
		\end{center}
	\end{figure}
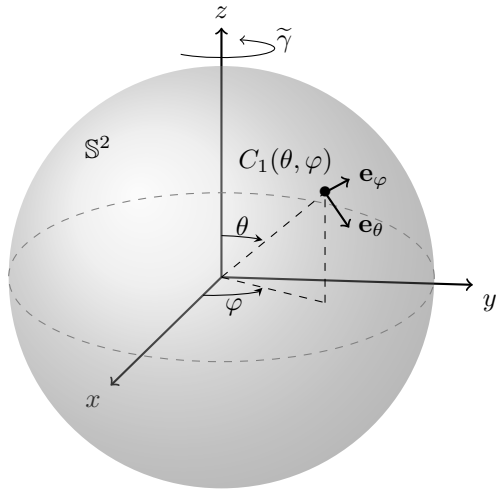
	
	\subsection{Historical discussion}
	We shall expose here some relevant results linked to our work.\\

	\textbf{Vortex patches  in the plane}

	\medskip\noindent
	Recall that the planar homogeneous incompressible Euler equations write
	\begin{equation}\label{Euler R2}
		\begin{cases}
			\partial_{t}\boldsymbol{\omega}(t,x_1,x_2)+\mathbf{v}(t,x_1,x_2)\cdot\nabla\boldsymbol{\omega}(t,x_1,x_2)=0,\\
			\mathbf{v}(t,x_1,x_2)=\nabla_{\mathbb{R}^2}^{\perp}\boldsymbol{\Psi}(t,x_1,x_2),\\
			\Delta_{\mathbb{R}^2}\boldsymbol{\Psi}(t,x_1,x_2)=\boldsymbol{\omega}(t,x_1,x_2),
		\end{cases}\qquad\nabla_{\mathbb{R}^2}^{\perp}\triangleq\begin{pmatrix}
			-\partial_{x_2}\\
			\partial_{x_1}
		\end{pmatrix},\qquad\Delta_{\mathbb{R}^2}\triangleq\partial_{x_1}^2+\partial_{x_2}^2.
	\end{equation}
	Vortex patches are weak solutions to \eqref{Euler R2} in the Yudovich class and taking the form $t\mapsto\mathbf{1}_{D_t}$ where $D_t$ is a bounded planar domain. The vorticity jump here is normalized to 1. The dynamics is described by the evolution of the boundary $\partial D_t$. More precisely, according to \cite[p. 174]{HMV13}, if $z(t,\cdot):\mathbb{T}\rightarrow\partial D_t$ is a parametrization of (one connected component of) the boundary at time $t$, then it must solve the following contour dynamics equation
	\begin{equation}\label{planar votex patch eq}
		\textnormal{Im}\Big(\partial_{t}z(t,x)\overline{\partial_x z(t,x)}\Big)=\partial_{x}\Big(\boldsymbol{\Psi}\big(t,z(t,x)\big)\Big).
	\end{equation}
	The solutions that we construct in this study are the analogous  to the V-states obtained in the planar case. These later are particular vortex patches where the dynamics is given by a uniform rotation of the initial domain around its center of mass (fixed to be the origin), namely $D_t=e^{\mathrm{i}\Omega t}D_0$ with $\Omega\in\mathbb{R}.$ As remarked by Rankine in 1858, any radial domain is a V-state rotating with any angular velocity $\Omega\in\mathbb{R}.$ Later, Kirchhoff \cite{K74} discovered non-trivial explicit  examples of V-states with elliptic shapes. Close to the unit disc, such solutions were first obtained numerically by Deem and Zabusky \cite{DZ78} and then analytically by Burbea \cite{B82} through bifurcation techniques. He constructed a countable family of local curves of V-states with $\mathbf{m}$-fold symmetries (i.e. invariance by $\tfrac{2\pi}{\mathbf{m}}$ angular rotation) bifurcating from the disc  at the angular velocities 
	\begin{equation}\label{Burbea}
		\Omega_{\mathbf{m}}=\frac{\mathbf{m}-1}{2\mathbf{m}}\cdot
	\end{equation}
	The global continuation of the Burbea's bifurcation branches was found in \cite{HMW20} through global bifurcation arguments. Some rigidity results giving necessary conditions on the angular velocity to obtain non-trivial V-states can be found in \cite{F00,H15,GSPSY-rigidity}.
	In the last decade, there have been intensive rigorous studies on the subject, giving more properties of the V-states and exploring their existence around  different
	topological structures. In particular,  the bifurcation from the annulus 
	$$A_b\triangleq\big\{z\in\mathbb{C}\quad\textnormal{s.t.}\quad b<|z|<1\big\}, \quad b\in(0,1),$$
	has been studied by Hmidi-de la Hoz-Mateu-Verdera \cite{HHMV16}. They proved that under the condition
	$$\mathtt{p}(b,\mathbf{m})\triangleq1+b^{\mathbf{m}}-\frac{1-b^2}{2}\mathbf{m}<0,$$
	there are exactly  two branches of periodic vortex patches, with two interfaces, emerging at the angular velocities 
	\begin{align}\label{Omega DC}
		\Omega_{\mathbf{m}}^{\pm}(b)\triangleq\frac{1-b^{2}}{4}\pm\frac{1}{2\mathbf{m}}\sqrt{\left(\frac{1-b^2}{2}\mathbf{m}-1\right)^2-b^{2\mathbf{m}}}.
	\end{align}
	The degenerate case $\mathtt{p}(b,\mathbf{m})=0$ has been discussed in \cite{HM16,WXZ22}.  
	Periodic multiple patches were found via a desingularization of an appropriate distribution of point vortices. The first work studying the desingularization of two point vortices using contour dynamics equations is due to Hmidi-Mateu \cite{HM17}. Other extensions, in the same spirit, have been achieved in \cite{G20,G21,HH21,HW22}. Similar results were obtained  using variational arguments in \cite{CLZ21,CLW14,CW22,T85} or via gluing techniques in \cite{DPMW20,DPMW22}.  Recently, the authors in \cite{GH22} studied the global continuation of the desingularization of the vortex pairs.
	The  existence of  a non-uniform rotating vorticity distribution has been treated in \cite{CCGS19,GHS20,GHM23}. In \cite{GSPS:2021}, G\'omez-Serrano, Park and Shi have constructed stationary configurations of multi-layered patches with finite kinetic energy using Nash-Moser techniques. Some additional references can be found in \cite{CCG16,HHHM15,HM16-1,HMV13,HMV15}.
	
	Many of  the results, mentioned above, apply not only to the  planar Euler equations but also to other active scalar equations such as the generalized surface quasi-geostrophic equation or the quasi-geostrophic shallow-water equations, see \cite{ADPM21,CQZZ21,CQZZ22,CWWZ21,CCGS16,CCGS19,CCGS20,HHH16,DHR19,G20,G21,GH22,GHM22,GHM22-1,GHS20,GCGS20,HH15,HH21, HMW20,HW22,HM16-1,HM16,HM17,HMV15,HXX22,R22,R21}.
	
	We end this discussion by very recent new perspectives concerning the existence of quasi-periodic patches using the Nash-Moser scheme together with KAM theory, see \cite{BHM22,GSIP23,HHM21,HHR23,HR22,HR21,R22}.\\	
	
	\textbf{Around stationary solutions on the rotating unit sphere}
	
	\medskip\noindent
	Looking for stationary solutions to \eqref{Euler eq on S2} is equivalent to solving the following equation on the stream function
	\begin{equation}\label{Stationary eq}
		\nabla^{\perp}\Psi(\theta,\varphi)\cdot\nabla\Big(\Delta\Psi(\theta,\varphi)-2\widetilde{\gamma}\cos(\theta)\Big)=0.
	\end{equation}
	If $\Psi(\theta,\varphi)=\Psi(\theta)$ is longitude independant, then it automatically solves \eqref{Stationary eq} because in this case, $\nabla^{\perp}\Psi(\theta)$ is colinear to $\mathbf{e}_{\varphi}$ and $\nabla\Big(\Delta\Psi(\theta)-2\widetilde{\gamma}\cos(\theta)\Big)$ is colinear to $\mathbf{e}_{\theta}.$ Such type of solutions are called \textit{zonal solutions}.
	Also, one can easily check that any solution of the semilinear elliptic problem
	\begin{equation}\label{Stationary eq bis}
		\Delta\Psi(\theta,\varphi)-2\widetilde{\gamma}\cos(\theta)=F\big(\Psi(\theta,\varphi)\big),
	\end{equation}
	with $F\in C^{1}(\mathbb{R},\mathbb{R})$, solves \eqref{Stationary eq}, but the converse is not true in general. Constantin and  Germain \cite{CG22} showed that any solution to \eqref{Stationary eq bis} with $F'>-6$ must be zonal (modulo rotation) and stable in $H^{2}(\mathbb{S}^2)$ provided the additional constraint $F'<0.$ Notice that the $-6$ corresponds to the second eigenvalues of the Laplace-Beltrami operator. The zonal Rossby-Haurwitz stream functions of degree $n\in\mathbb{N}$ are special stationary solutions in the form
	$$\Psi_{n}(\theta)=\beta Y_n^0(\theta)+\frac{2\widetilde{\gamma}}{n(n+1)-2}\cos(\theta),\qquad\beta\in\mathbb{R}^*,$$
	where $Y_{n}^0$ is the spherical harmonic. We refer the reader to \cite{AH12} for an introduction to spherical harmonics. In \cite{CG22}, the authors also discussed the local and global bifurcation of non-zonal solutions to \eqref{Stationary eq bis} from Rossby-Haurwitz waves. They also proved the stability in $H^{2}(\mathbb{S}^2)$ of Rossby-Haurwitz zonal solutions of degree $2$ as well as the unstability in $H^{2}(\mathbb{S}^2)$ of more general non-zonal Rossby-Haurwitz type solutions. Very recently, the stability of the degree $2$ Rossby-Haurwitz waves in $L^{p}(\mathbb{S}^2)$ spaces with $p\in(1,\infty)$ has been obtained by Cao-Wang-Zuo \cite{CWZ23}. Recently, Nualart \cite{N22} proved the existence of non-zonal stationary solutions Gevrey-close to the zonal Rossby-Haurwitz stream functions of degree $2$. The $L^1$-Lyapunov stability of zonal monotone vorticities belonging to $L^{p}(\mathbb{S}^2)$ for $p\in(2,\infty)$ was discussed by Caprino-Marchioro \cite{CM88}. We mention that stationary solutions to \eqref{Euler eq on S2} correspond to traveling solutions of the non-rotating case $(E_0)$. More generally, we have the following result, see \cite{CG22,N22}.
	\begin{lem}
		We consider two vorticities $\Omega,\widetilde{\Omega}$ related through
		$$\Omega(t,\theta,\varphi)=-2c\cos(\theta)+\widetilde{\Omega}(t,\theta,\varphi-ct),$$
		with $c\in\mathbb{R}$,	and associated to stream functions $\Psi,\widetilde{\Psi}$ and velocity fields $U,\widetilde{U}$:
		$$\Psi(t,\theta,\varphi)=c\cos(\theta)+\widetilde{\Psi}(t,\theta,\varphi-ct),\qquad U(t,\theta,\varphi)=c\sin(\theta)\mathbf{e}_{\varphi}+\widetilde{U}(t,\theta,\varphi-ct).$$
		Then, the following are equivalent
		\begin{enumerate}[label=(\roman*)]
			\item $(\Omega,\Psi,U)$ is a solution to $(E_{\widetilde{\gamma}}).$
			\item $(\widetilde{\Omega},\widetilde{\Psi},\widetilde{U})$ is a solution to $(E_{\widetilde{\gamma}+c}).$
		\end{enumerate}
	\end{lem}
	For our later purposes, we shall prove that a longitude independent (absolute) vorticity generates a zonal (so stationary) flow. This is given by the following lemma, whose proof is postponed to Appendix \ref{proof lemma}.
	\begin{lem}\label{lem sat vort}
		For any $\alpha\in\mathbb{R},$ we introduce the rotation of angle $\alpha$ around the $z$ axis
		$$\mathcal{R}(\alpha)\triangleq\begin{pmatrix}
			\cos(\alpha) & -\sin(\alpha) & 0\\
			\sin(\alpha) & \cos(\alpha) & 0\\
			0 & 0 & 1
		\end{pmatrix}\in SO_{3}(\mathbb{R}).$$
		Assume that $$\forall\alpha\in\mathbb{R},\quad\forall\xi\in\mathbb{S}^2,\quad\Omega\big(\mathcal{R}(\alpha)\xi\big)=\Omega(\xi),$$
		or equivalently
		$$\forall\alpha\in\mathbb{R},\quad\forall\xi\in\mathbb{S}^2,\quad\overline{\Omega}\big(\mathcal{R}(\alpha)\xi\big)=\overline{\Omega}(\xi),$$
		then
		$$\forall\alpha\in\mathbb{R},\quad\forall\xi\in\mathbb{S}^2,\quad\Psi\big(\mathcal{R}(\alpha)\xi\big)=\Psi(\xi).$$
		This means that the flow is zonal (and thus stationary).
	\end{lem}
	
	\textbf{Patch type solutions on the rotating $2$-sphere}
	
	\medskip\noindent
	Our analysis is strongly motivated by previous numerical works concerning the existence of patch type solutions for the rotating $2$-sphere. The pioneering results are due to Dritschel-Polvani \cite{DP92,DP93} where they considered the sphere at rest ($\widetilde{\gamma}=0$) and found, numerically, vortex cap solutions with one and two interfaces. They also studied the numerical nonlinear stability. Later Kim \cite{K18} described the free boundary problem for patch type solutions in the rotating $2$-sphere   $(\widetilde{\gamma}\neq 0$) by using the stereographic projection. In \cite{KS21}, Kim, Sakajo and Sohn  numerically observed the existence of vortex caps (only one interface)  and vortex bands (two interfaces) \cite{KSS18}. They also showed the linear stability of those solutions.

	\subsection{Dynamics of vortex cap solutions}
	Here, we introduce the notion of vortex cap solutions on the sphere, which are the analogous to vortex patches in the plane. Then, we derive the fundamental contour dynamics equations, introduced in this work to track the evolution of the caps interfaces. Our formulation \eqref{vorticity cap equations} is derived by following the ideas of \cite[p. 174]{HMV13}, implemented in the context of vortex patches, but adapted to fit with the non Euclidean geometry of the sphere.\\
	
	Fix $M\in\mathbb{N}\setminus\{0,1\}$ and $(\omega_{k})_{1\leqslant k\leqslant M}\in\mathbb{R}^M$ such that 
	\begin{equation}\label{diff omgk}
		\forall k\in\llbracket 1,M-1\rrbracket,\quad\omega_{k}\neq\omega_{k+1}.
	\end{equation}
	Consider a partition of the unit sphere in the form
	$$\mathbb{S}^2=\bigsqcup_{k=1}^{M}\mathscr{C}_{k}(0),$$
	where for any $k\in\llbracket 1,M-1\rrbracket$, the boundary $\Gamma_{k}(0)\triangleq\partial\mathscr{C}_{k}(0)\cap\partial\mathscr{C}_{k+1}(0)$ is diffeomorphic to a circle. Take an initial condition in the form
	\begin{equation}\label{initial vc}
		\overline{\Omega}(0,\cdot)=\sum_{k=1}^{M}\omega_{k}\mathbf{1}_{\mathscr{C}_k(0)}.
	\end{equation}
	The Gauss constraint \eqref{null avrg vort} requires the following additional condition
	\begin{equation}\label{G-Gauss}
		\sum_{k=1}^{M}\omega_{k}\sigma\big(\mathscr{C}_{k}(0)\big)=0.
	\end{equation}
	Observe that, by virtue of \eqref{Euler eq on S2}-\eqref{def abs vort}, the absolute vorticity $\overline{\Omega}$ solves the nonlinear transport equation
	$$\partial_{t}\overline{\Omega}+U\cdot\nabla\overline{\Omega}=0.$$
	Since the singularity of the Green function in \eqref{psi-G} is logarithmic, then, similarly to Yudovich’s theory \cite{Y63} in the planar case, one can expect to obtain existence and uniqueness of a global in time weak solution which is Lagrangian, namely
	$$\forall t\geqslant 0,\quad\forall\xi\in\mathbb{S}^2,\quad\overline{\Omega}(t,\xi)=\overline{\Omega}\big(0,\Phi_{t}^{-1}(\xi)\big),$$
	where
	$$\forall\xi\in\mathbb{S}^2,\quad\partial_{t}\Phi_{t}(\xi)=U\big(t,\Phi_{t}(\xi)\big),\qquad\Phi_{0}(\xi)=\xi.$$
	Applying this remark with the initial condition \eqref{initial vc} gives the following structure of the solution at any later time $t\geqslant0$
	\begin{equation}\label{Gvc sol}
		\overline{\Omega}(t,\cdot)=\sum_{k=1}^{M}\omega_{k}\mathbf{1}_{\mathscr{C}_k(t)},\qquad\textnormal{with}\qquad\forall k\in\llbracket 1,M\rrbracket,\quad\mathscr{C}_{k}(t)\triangleq\Phi_{t}\big(\mathscr{C}_k(0)\big).
	\end{equation}
	Since $U$ is solenoidal, then the flow $t\mapsto\Phi_{t}$ is measure preserving
	$$\forall k\in\llbracket 1,M\rrbracket,\quad\sigma\big(\mathscr{C}_k(t)\big)=\sigma\big(\mathscr{C}_k(0)\big).$$
	Any solution in the form \eqref{Gvc sol} satisfying the conditions \eqref{diff omgk} and \eqref{G-Gauss} is called a \textit{vortex cap solution.} From now on, we fix a $k\in\llbracket 1,M-1\rrbracket.$ Assume that the initial boundary $\Gamma_{k}(0)$ can be described as the zero level set of a certain $C^1$ regular function $\mathtt{h}_k:\mathbb{S}^2\rightarrow\mathbb{R}$, namely
	$$\Gamma_{k}(0)=\big\{\xi\in\mathbb{S}^2\quad\textnormal{s.t.}\quad \mathtt{h}_k(\xi)=0\big\}.$$
	Let us consider the following quantity
	\begin{equation}\label{Fk}
		\forall t\geqslant0,\quad\forall\xi\in\mathbb{S}^2,\quad F_k(t,\xi)\triangleq \mathtt{h}_k\big(\Phi_{t}^{-1}(\xi)\big).
	\end{equation}
	Then, by construction $F_k(t,\cdot)$ describes the boundary $\Gamma_{k}(t)\triangleq\partial\mathscr{C}_k(t)\cap\partial\mathscr{C}_{k+1}(t)$. More precisely,
	$$\Gamma_{k}(t)=\big\{\xi\in\mathbb{S}^2\quad\textnormal{s.t.}\quad F_k(t,\xi)=0\big\}.$$
	Differentiating the relation \eqref{Fk} with respect to time yields
	$$\forall t\geqslant0,\quad\forall\xi\in\mathbb{S}^2,\quad \partial_{t}F_k\big(t,\Phi_{t}(\xi)\big)+U\big(t,\Phi_{t}(\xi)\big)\cdot\nabla F\big(t,\Phi_{t}(\xi)\big)=0.$$
	Now, take a parametrization $z_k(t,\cdot):\mathbb{T}\rightarrow\Gamma_k(t).$ We have
	$$\forall t\geqslant0,\quad\forall x\in\mathbb{T},\quad F_k\big(t,z_k(t,x)\big)=0.$$
	Differentiating the previous relation with respect to time implies
	$$\forall t\geqslant0,\quad\forall x\in\mathbb{T},\quad\partial_{t}F_k\big(t,z_k(t,x)\big)+\partial_{t}z_k(t,x)\cdot\nabla F_k\big(t,z_k(t,x)\big)=0.$$
	Putting together the foregoing calculations, we deduce that
	$$\Big[\partial_{t}z_k(t,x)-U\big(t,z_k(t,x)\big)\Big]\cdot\nabla F_k\big(t,z_k(t,x)\big)=0.$$
	The above scalar products can be understood either as taken in the tangent space $T_{z_k(t,x)}\mathbb{S}^2\cong\mathbb{R}^2$ or in the classical Euclidean space $\mathbb{R}^3$, both are equivalent. Indeed, in the spherical coordinates the Euclidean metric writes
	$$\mathtt{g}_{\mathbb{R}^3}(r,\theta,\varphi)\triangleq dr^2+r^2\big(d\theta^2+\sin^2(\theta)d\varphi^2\big).$$
	Since the sphere is described by the equation $r=1,$ then the induced metric of $\mathtt{g}_{\mathbb{R}^3}$ on $\mathbb{S}^2$ is indeed $\mathtt{g}_{\mathbb{S}^2}$ as defined in \eqref{metric S2}. On $T_{z_k(t,x)}\mathbb{S}^2,$ the operator $J$ acts as a rotation of $-\tfrac{\pi}{2}.$ Consequently, since $\partial_{x}z_k(t,x)$ is tangential to $\Gamma_k(t)$ and contained in $T_{z_k(t,x)}\mathbb{S}^2$, then $J\partial_{x}z_k(t,x)$ is orthogonal to $\Gamma_k(t)$ and contained in $T_{z_k(t,x)}\mathbb{S}^2.$ In addition, since $\Gamma_k(t)$ is a level set of $F_k(t,\cdot),$ then $\nabla F_k\big(t,z_k(t,x)\big)$ is also orthogonal to $\Gamma_k(t)$ and contained in $T_{z_k(t,x)}\mathbb{S}^2.$ We deduce that $J\partial_{x}z_k(t,x)$ and $\nabla F_k\big(t,z_k(t,x)\big)$ are proportional, which leads to
	$$\forall t\geqslant0,\quad\forall x\in\mathbb{T},\quad\Big[\partial_{t}z_k(t,x)-U\big(t,z_k(t,x)\big)\Big]\cdot\big(J\partial_{x}z_k(t,x)\big)=0.$$
	But, using that $J^{T}=J^{-1}=-J$ and $\nabla^{\perp}=J\nabla$, we obtain
	\begin{align*}
		U\big(t,z_k(t,x)\big)\cdot\big(J\partial_{x}z_k(t,x)\big)&=\nabla^{\perp}\Psi\big(t,z_k(t,x)\big)\cdot\big(J\partial_{x}z_k(t,x)\big)\\
		&=-\Big(J\nabla^{\perp}\Psi\big(t,z_k(t,x)\big)\Big)\cdot\partial_{x}z_k(t,x)\\
		&=\nabla\Psi\big(t,z_k(t,x)\big)\cdot\partial_{x}z_k(t,x)\\
		&=\partial_{x}\Big(\Psi\big(t,z_k(t,x)\big)\Big).
	\end{align*}
	Hence the contour dynamics equations for the vortex cap solutions are 
	\begin{equation}\label{vorticity cap equations}
		\forall k\in\llbracket 1,M-1\rrbracket,\quad\forall t\geqslant0,\quad\forall x\in\mathbb{T},\quad\partial_{t}z_k(t,x)\cdot\big(J\partial_{x}z_k(t,x)\big)=\partial_{x}\Big(\Psi\big(t,z_k(t,x)\big)\Big),
	\end{equation}
	which are comparable to \eqref{planar votex patch eq}.	
	\subsection{Main results}
	Let us now present our main results. First observe that Lemma \ref{lem sat vort} implies that any spherical vortex cap in the form
	$$\overline{\Omega}(\theta)=\omega_1\mathbf{1}_{0<\theta<\theta_1}+\omega_2\mathbf{1}_{\theta_1\leqslant\theta<\theta_2}+\ldots+\omega_{M-1}\mathbf{1}_{\theta_{M-2}\leqslant\theta<\theta_{M-1}}+\omega_{M}\mathbf{1}_{\theta_{M-1}\leqslant\theta<\pi},$$
	with
	$$M\in\mathbb{N}\setminus\{0,1\},\qquad\theta_0\triangleq0<\theta_1<\ldots<\theta_{M-1}<\pi\triangleq\theta_M,\qquad\forall k\in\llbracket 1,M-1\rrbracket,\quad\omega_{k}\neq\omega_{k+1},$$
	and supplemented by the Gauss condition
	$$\sum_{k=1}^{M}\omega_{k}\big(\cos(\theta_{k-1})-\cos(\theta_{k})\big)=0,$$
	generates a stationary vortex cap solution to \eqref{Euler eq on S2}. In the sequel, we shall focus on the cases $M=2$ and $M=3$ (the latter will also be referred to as \textit{vortex bands} or \textit{vortex strips}). More precisely, we study the existence non-trivial periodic solutions living close to these structures. Due to the symmetry of the stationary vortex caps, we will look for rotating solutions around the vertical axis at uniform velocity $c$, that is
	$$\forall t\geqslant 0,\quad\forall\xi\in\mathbb{S}^2,\quad\overline{\Omega}(t,\xi)=\overline{\Omega}\big(0,\mathcal{R}(ct)\xi\big),$$
	and satisfy for some fixed $\mathbf{m}\in\mathbb{N}^*$ the $\mathbf{m}$-fold property
	$$\forall\xi\in\mathbb{S}^2,\quad\overline{\Omega}\big(0,\mathcal{R}(\tfrac{2\pi}{\mathbf{m}})\xi\big)=\overline{\Omega}(0,\xi).$$
	
	Our first result concerns the one-interface case $M=2$ and reads as follows.
	\begin{theo}\label{thm bif vcap1}
		Let $\widetilde{\gamma}\in\mathbb{R}$, $\mathbf{m}\in\mathbb{N}^*$ and $\theta_0\in(0,\pi).$ Consider $\omega_{N},\omega_{S}\in\mathbb{R}$ such that
		$$\frac{\omega_N+\omega_S}{\omega_{N}-\omega_{S}}=\cos(\theta_0).$$
		There exists a branch of $\mathbf{m}$-fold uniformly rotating vortex cap solutions to \eqref{Euler eq on S2} with one interface bifurcating from
		$$\overline{\Omega}_{\textnormal{\tiny{FC}}}(\theta)\triangleq\omega_{N}\mathbf{1}_{0<\theta<\theta_0}+\omega_{S}\mathbf{1}_{\theta_0\leqslant\theta<\pi},$$
		at the velocity
		$$c_{\mathbf{m}}(\widetilde{\gamma})\triangleq\widetilde{\gamma}-(\omega_{N}-\omega_{S})\frac{\mathbf{m}-1}{2\mathbf{m}}\cdot$$
	\end{theo}
	Before moving to the second result, some remarks are in order.
	\begin{remark}\label{rmk thm1}

		\begin{enumerate}
			\item The bifurcation points $c_{\mathbf{m}}(\widetilde{\gamma})$ correspond to a shift by the rotation speed $\widetilde{\gamma}$ of Burbea's frequencies \eqref{Burbea} with a vorticity jump $\llbracket\overline{\Omega}\rrbracket\triangleq \omega_{N}-\omega_S=-1$. Hence, in the rotation frame of the sphere, the solutions bifurcate at Burbea's frequencies. 
			\item The local curve is parametrized in particular with $\varepsilon\in(-\varepsilon_0,\varepsilon_0)\mapsto c^\varepsilon_{\mathbf{m}}(\widetilde{\gamma})\in\mathbb{R}$ for some small $\varepsilon_0>0.$ Denoting $\widetilde{\gamma}_{\mathbf{m}}\triangleq(\omega_N-\omega_S)\tfrac{\mathbf{m}-1}{2\mathbf{m}},$ we have
			$c^0_{\mathbf{m}}(\widetilde{\gamma}_{\mathbf{m}})=0.$ Since the dependence of $c_{\mathbf{m}}(\widetilde{\gamma})$ in $\widetilde{\gamma}$ is affine, an application of the implicit function theorem allows to construct a curve $\varepsilon\in(-\varepsilon_0,\varepsilon_0)\mapsto\widetilde{\gamma}_{\mathbf{m}}^\varepsilon\in\mathbb{R}$ such that (up to reducing the size of $\varepsilon_0$)
			$$\forall\varepsilon\in(-\varepsilon_0,\varepsilon_0),\quad c^\varepsilon_{\mathbf{m}}\big(\widetilde{\gamma}_{\mathbf{m}}^\varepsilon\big)=0.$$
			This means that we can construct a branch of non-trivial $\mathbf{m}$-fold solutions which are stationary in the geocentric frame. One could also obtain these solutions by implementing bifurcation theory with the parameter~$\widetilde\gamma$.
			\item The bifurcation analysis is performed in H\"older spaces but, similarly to the planar case, we expect our solutions to be analytic.
		\end{enumerate}
	\end{remark}
	
	
	Next, we shall  present our second result dealing with two interfaces ($M=3$). In this case, the computations in the spectral study are much more involved. For the sake of simplicity, we give here an informal statement of the second main theorem. More precise statement is postponed to Theorem \ref{thm bif vcap3}. 
	\begin{theo}\label{thm bif vcap2}
		Let $\widetilde{\gamma}\in\mathbb{R}$ and $0<\theta_1<\theta_2<\pi.$ Fix $\omega_N,\omega_C,\omega_S\in\mathbb{R}$ such that
		\begin{equation}\label{Gauss thm2}
			\omega_{N}+\omega_{S}=(\omega_N-\omega_C)\cos(\theta_1)+(\omega_C-\omega_{S})\cos(\theta_2).
		\end{equation}
		Then, there exists $N(\theta_1,\theta_2)\triangleq N(\theta_1,\theta_2,\omega_N,\omega_S,\omega_C)\in\mathbb{N}^*$  such that for any $\mathbf{m}\in\mathbb{N}^*$ with $\mathbf{m}\geqslant N(\theta_1,\theta_2)$, there exists at least one branch of $\mathbf{m}$-fold uniformly rotating vortex strips for \eqref{Euler eq on S2} bifurcating from
		\begin{equation}\label{omegaFC2}
			\overline{\Omega}_{\textnormal{\tiny{FC2}}}(\theta)\triangleq\omega_{N}\mathbf{1}_{0<\theta<\theta_1}+\omega_C\mathbf{1}_{\theta_1\leqslant\theta<\theta_2}+\omega_{S}\mathbf{1}_{\theta_2\leqslant\theta<\pi},
		\end{equation}
		at some value of the velocity $c$.			
	\end{theo}

	\begin{remark} Let us make the following remarks.
		\begin{enumerate}
			\item Notice that the proof in this case reveals some non-degeneracy conditions, required to avoid spectral collisions 
			when applying bifurcation techniques. This situation does not appear in the planar case \cite{HHMV16}, where the bifurcation eigavalues are well-separated.
			\item As we shall see in Lemma \ref{lem:spectrum} the bifurcation points are also explicit, in this case, and, a priori, are independent from the planar case \eqref{Omega DC}, even if they globally have a similar structure. Moreover, similarly to Remark \ref{rmk thm1}-2, we can construct at most two curves of non-trivial $\mathbf{m}$-fold solutions which are stationary in the geocentric frame.
			\item The cases with more interfaces ($M\geqslant 4$) seem much more involved to study.
		\end{enumerate}
	\end{remark}
	
	\subsection{Organization of the paper}
	This work is organized as follows. In Section \ref{sec-one-interf} we provide the proof of Theorem \ref{thm bif vcap1} showing the analytical existence of non-trivial vortex caps with one interface bifurcating from the trivial one. First, we characterize the existence of such solutions with the non-trivial roots of a nonlinear and nonlocal functional. Later, we give the spectral properties of such functional and conlude the proof of the first main result. In Section \ref{sec-two-interf} we analyze the two-interfaces problem (Theorem \ref{thm bif vcap2}), which is more delicate since one has to study a coupled nonlinear system. The spectral study is more complex in this case, and we can show the existence of non-trivial vortex strips with large $\mathbf{m}$-fold symmetries under non-degeneracy conditions. Finally, the Appendix gives the proof of some technical lemmas and states the Crandall-Rabinowitz theorem.

	\subsection*{Acknowledgments} The authors would like to thank David G. Dritschel for suggesting the problem. The third author also would like to thank Vincent Duchêne and Didier Smets who also arised the question during his PhD defense. The work of Claudia Garc\'ia has been supported by the MINECO--Feder (Spain) research grant number RTI2018--098850--B--I00, the Junta de Andaluc\'ia (Spain) Project
	FQM 954, the Severo Ochoa Programme for Centres of Excellence in R\&D(CEX2019-000904-S) and by the PID2021-124195NB-C32. The work of Zineb Hassainia has been supported by Tamkeen under the NYU Abu Dhabi Research Institute grant of the center SITE. The work of Emeric Roulley has been supported by PRIN 2020XB3EFL, "Hamiltonain and Dispersive PDEs".
	
	\section{The one--interface case}\label{sec-one-interf}
	This section is devoted to the proof of Theorem \ref{thm bif vcap1} dealing with the case of one interface ($M=2$). We first discuss the stationary flat cap solution. Then, by choosing a suitable ansatz, we can rewrite the vortex cap equation \eqref{vorticity cap equations} which leads to reformulate our problem in looking for the zeros of a nonlinear and nonlocal functional, see \eqref{scrF}. Finally, we implement bifurcation arguments in order to show the existence of non-trivial roots of such functional. Hence, the proof relies on checking all the hypothesis of Crandall-Rabinowitz Theorem, see Appendix \ref{sec-CR}.
	\subsection{Equation of interest}\label{sec eq1}
	First, we shall discuss some properties of the spherical stationary solutions with one interface.
	\begin{lem}\label{lem flat cap}
		Let $\theta_0\in(0,\pi).$ For any $\omega_{N},\omega_{S}\in\mathbb{R}$ such that
		\begin{equation}\label{constraint omega NS}
			\frac{\omega_{N}+\omega_{S}}{\omega_{N}-\omega_{S}}=\cos(\theta_0),
		\end{equation}
		the following function describing the flat cap (FC) $$\overline{\Omega}_{\textnormal{\tiny{FC}}}(\theta)\triangleq\omega_{N}\mathbf{1}_{0<\theta<\theta_0}+\omega_{S}\mathbf{1}_{\theta_0\leqslant\theta<\pi},$$
		is a stationary solution to Euler equations.
		In addition,
		\begin{equation}\label{dttPsi tt0}
			\partial_{\theta}\Psi_{\textnormal{\tiny{FC}}}(\theta_0)=\left(\frac{\omega_{N}-\omega_{S}}{2}-\widetilde{\gamma}\right)\sin(\theta_0).
		\end{equation}
	\end{lem}
	\begin{proof}
		$\blacktriangleright$ Observe that $$\forall\alpha\in\mathbb{R},\quad\forall\xi\in\mathbb{S}^2,\quad\overline{\Omega}_{\textnormal{\tiny{FC}}}\big(\mathcal{R}(\alpha)\xi\big)=\overline{\Omega}_{\textnormal{\tiny{FC}}}(\xi).$$
		Hence, Lemma \ref{lem sat vort} applies and proves that this is a stationary solution.\\
		$\blacktriangleright$ Notice that the constraint \eqref{constraint omega NS} is required since \eqref{null avrg vort} and \eqref{conv int S2} imply
		\begin{align*}
			0=\int_{\mathbb{S}^2}\Omega_{\textnormal{\tiny{FC}}}(t,\xi)d\sigma(\xi)&=\int_{0}^{2\pi}\int_{0}^{\pi}\Omega_{\textnormal{\tiny{FC}}}(t,\theta,\varphi)\sin(\theta)d\theta d\varphi\\
			&=2\pi\left(\omega_{N}\int_{0}^{\theta_0}\sin(\theta)d\theta+\omega_{S}\int_{\theta_0}^{\pi}\sin(\theta)d\theta\right)\\
			&=2\pi\Big[\omega_{N}\big(1-\cos(\theta_0)\big)+\omega_{S}\big(1+\cos(\theta_0)\big)\Big].
		\end{align*}
		$\blacktriangleright$ The potential velocity solves the elliptic equation
		$$\Delta\Psi_{\textnormal{\tiny{FC}}}=\overline{\Omega}_{\textnormal{\tiny{FC}}}+2\widetilde{\gamma}\cos(\theta),\qquad\textnormal{i.e.}\qquad\partial_{\theta}\big[\sin(\theta)\partial_{\theta}\Psi_{\textnormal{\tiny{FC}}}(\theta)\big]=\sin(\theta)\Big(\omega_{N}\mathbf{1}_{0<\theta<\theta_0}+\omega_{S}\mathbf{1}_{\theta_0\leqslant\theta<\pi}\Big)+\widetilde{\gamma}\sin(2\theta).$$
		Integrating the previous relation gives
		$$\partial_{\theta}\Psi_{\textnormal{\tiny{FC}}}(\theta)=\begin{cases}
			\frac{\omega_{N}}{\sin(\theta)}\big(1-\cos(\theta)\big)-\frac{\widetilde{\gamma}\cos(2\theta)}{2\sin(\theta)}+\frac{c}{\sin(\theta)}, & \textnormal{if }\theta\in(0,\theta_0),\\
			\frac{\omega_{N}}{\sin(\theta)}\big(1-\cos(\theta_0)\big)+\frac{\omega_S}{\sin(\theta)}\big(\cos(\theta_0)-\cos(\theta)\big)-\frac{\widetilde{\gamma}\cos(2\theta)}{2\sin(\theta)}+\frac{c}{\sin(\theta)}, & \textnormal{if }\theta\in[\theta_0,\pi).
		\end{cases}$$
		Since the flow is zonal, there is no velocity at the pole.  Which implies that
		$$\lim_{\theta\to0^+}\partial_{\theta}\Psi_{\textnormal{\tiny{FC}}}(\theta)=0.$$
		As a consequence, we must take the constant of integration $c$ as follows
		$$c\triangleq\frac{\widetilde{\gamma}}{2}\cdot$$
		Finally, using \eqref{constraint omega NS}, we can write
		$$\partial_{\theta}\Psi_{\textnormal{\tiny{FC}}}(\theta)=\begin{cases}
			\frac{\omega_{N}}{\sin(\theta)}\big(1-\cos(\theta)\big)-\widetilde{\gamma}\sin(\theta), & \textnormal{if }\theta\in(0,\theta_0),\\
			-\frac{\omega_S}{\sin(\theta)}\big(1+\cos(\theta)\big)-\widetilde{\gamma}\sin(\theta), & \textnormal{if }\theta\in[\theta_0,\pi).
		\end{cases}$$
		At $\theta=\theta_0,$ using again \eqref{constraint omega NS}, we find
		$$\partial_{\theta}\Psi_{\textnormal{\tiny{FC}}}(\theta_0)=\frac{1}{2\sin(\theta_0)}\Big[\omega_{N}\big(1-\cos(\theta_0)\big)-\omega_{S}\big(1+\cos(\theta_0)\big)\Big]-\widetilde{\gamma}\sin(\theta_0)=\left(\frac{\omega_{N}-\omega_{S}}{2}-\widetilde{\gamma}\right)\sin(\theta_0).$$
		The proof of Lemma \ref{lem flat cap} is now complete.
	\end{proof}
	
	\begin{figure}[!h]
		\begin{center}
			\tdplotsetmaincoords{80}{100}
			\begin{tikzpicture}[tdplot_main_coords,scale=3]
				\draw[thick,->] (0,0,0) -- (3,0,0) node[anchor=north east]{$x$};
				\draw[thick,->] (0,0,0) -- (0,1.2,0) node[anchor=north west]{$y$};
				\draw[thick,->] (0,0,0) -- (0,0,1.2) node[anchor=south]{$z$};
				
				\def\h{1.1}
				\coordinate (O) at (0,0,0);
				\coordinate (z) at (0,0,\h);
				\coordinate (x) at (\h,0,0);
				\coordinate (y) at (0,\h,0);
				\coordinate (a) at (0.6,0.6,0.53);
				\coordinate (b) at (0.6,0.6,0);

				\draw[gray!100,dashed] (0,0,0.6) ellipse(0.785cm and 0.15cm);
				\draw[gray!100,dashed] (0,0,0)--(0.55,0.45,0.55);

				\begin{pgfonlayer}{foreground}
					\draw[->] (z)+(-40:0.25) arc (-40:170:0.25);
				\end{pgfonlayer}
				\node at (0,0.3,1.15) {$\widetilde{\gamma}$};
				\node at (0,0.2,0.9) {$\omega_N(f)$};
				\node at (1,0.5,0) {$\omega_S$};
				\node at (0,0.1,0.3) {$\textcolor{gray!100}{\theta_0}$};
				
				\draw[->,>=stealth,gray!100] (0,0,0.2) arc [start angle=0,end angle=45,x radius=-0.3,y radius=0.22];
				
				\tikzmath{function absci(\t) {return sin(54+5*cos(6*\t r))*cos(\t r);};}
				\tikzmath{function ord(\t) {return sin(54+5*cos(6*\t r))*sin(\t r);};}
				\tikzmath{function cote(\t) {return cos(54+5*cos(6*\t r));};}
				
				\draw[red,thick] plot[domain=0:6.3,smooth,variable=\t,samples=500] ({absci(\t)},{ord(\t)},{cote(\t)});

				\shade[tdplot_screen_coords, ball color = gray!40, opacity=0.4] (0,0) circle (1);
				
			\end{tikzpicture}
			\qquad
			\tdplotsetmaincoords{0}{0}
			\begin{tikzpicture}[tdplot_main_coords,scale=3]
				\draw[thick,->] (0.035,0,0) -- (1.2,0,0) node[anchor=north east]{$x$};
				\draw[thick,->] (0,0.035,0) -- (0,1.2,0) node[anchor=north west]{$y$};
				\draw[thick] (0,0,0) -- (0,0,1.2) node[anchor=east]{$z$};
				
				\def\h{1.1}
				\coordinate (O) at (0,0,0);
				\coordinate (z) at (0,0,\h);
				\coordinate (x) at (\h,0,0);
				\coordinate (y) at (0,\h,0);
				\coordinate (a) at (0.6,0.6,0.53);
				\coordinate (b) at (0.6,0.6,0);

				\draw[gray!100,dashed] (0,0,0) ellipse(0.8cm and 0.8cm);
				\node at (0,0,0) {$\odot$};

				\begin{pgfonlayer}{foreground}
					\draw[->] (z)+(-40:0.25) arc (-40:170:0.25);
				\end{pgfonlayer}
				\node at (-0.2,0.25,0) {$\widetilde{\gamma}$};

				\tikzmath{function absci(\t) {return sin(54+5*cos(6*\t r))*cos(\t r);};}
				\tikzmath{function ord(\t) {return sin(54+5*cos(6*\t r))*sin(\t r);};}
				\tikzmath{function cote(\t) {return cos(54+5*cos(6*\t r));};}
				
				\draw[red,thick] plot[domain=0:6.3,smooth,variable=\t,samples=500] ({absci(\t)},{ord(\t)},{cote(\t)});

				\shade[tdplot_screen_coords, ball color = gray!40, opacity=0.4] (0,0) circle (1);
			\end{tikzpicture}
			\caption{Representation of one interface (in red) vortex cap solutions with $6$-fold symmetry.}
		\end{center}
	\end{figure}
	
	Now, fix $\theta_0\in(0,\pi),$ $\omega_N,\omega_S\in\mathbb{R}$ satisfying \eqref{constraint omega NS} and let us consider a vortex cap solution close to $\overline{\Omega}_{\textnormal{\tiny{FC}}}$ in the form
	$$\overline{\Omega}(t,\theta,\varphi)=\omega_N(f)\mathbf{1}_{0<\theta<\theta_0+f(t,\varphi)}+\omega_{S}\mathbf{1}_{\theta_0+f(t,\varphi)\leqslant\theta<\pi},\qquad|f(t,\varphi)|\ll1.$$
	The quantity $\omega_N(f)$ is defined implicitely so that the Gauss constraint is satisfied,namely
	\begin{align*}
		0&=\int_{\mathbb{S}^2}\overline{\Omega}(t,\xi)d\sigma(\xi)\\
		&=\int_{0}^{2\pi}\left(\omega_N(f)\int_{0}^{\theta_0+f(t,\varphi)}\sin(\theta)d\theta+\omega_S\int_{\theta_0+f(t,\varphi)}^{\pi}\sin(\theta)d\theta\right)d\varphi\\
		&=\int_{0}^{2\pi}\omega_N(f)\big(1-\cos(\theta_0+f(t,\varphi)\big)+\omega_S\big(1+\cos(\theta_0+f(t,\varphi)\big)d\varphi.
	\end{align*}
	To make the argument rigorous, we define, for any reasonable Banach algebra functional space $X$ whose elements have zero space-average, the functional
	$$\mathcal{G}:\begin{array}[t]{rcl}
		X\times\mathbb{R} & \rightarrow & \mathbb{R}\\
		(f,\omega) & \mapsto & \int_{0}^{2\pi}\omega\big(1-\cos(\theta_0+f(t,\varphi)\big)+\omega_S\big(1+\cos(\theta_0+f(t,\varphi)\big)d\varphi. 
	\end{array}$$
	The precise function spaces will be given in the next section. The application $\mathcal{G}$ is clearly smooth and, using that $\theta_0\in(0,\pi)$,
	$$\mathcal{G}(0,\omega_N)=0,\qquad(\partial_{\omega}\mathcal{G})(0,\omega_N)=2\pi\big(1-\cos(\theta_0)\big)\neq0.$$
	Therefore, one can apply the classical implicit function theorem to get the existence of a smooth functional $\omega_N(\cdot):B_{X}(r_0)\rightarrow\mathbb{R}$ such that
	\begin{equation}\label{omgN at 0}
		\forall f\in B_{X}(r_0),\quad\mathcal{G}\big(f,\omega_N(f)\big)=0,\qquad\omega_N(0)=\omega_N.
	\end{equation}
	We have denoted $B_{X}(r_0)$ the ball in $X$ of center $0$ and small enough radius $r_0>0$. Differentiating with respect to $f$ in any direction $h\in X$ (which has zero average), we get
	\begin{align*}
		d_f\omega_N(0)[h]&=-\frac{(\partial_{f}\mathcal{G})(0,\omega_N)[h]}{(\partial_{\omega}\mathcal{G})(0,\omega_N)}\\
		&=\frac{\sin(\theta_0)(\omega_N-\omega_S)}{2\pi\big(1-\cos(\theta_0)\big)}\int_{0}^{2\pi}h(t,\varphi)d\varphi=0.
	\end{align*}
	Therefore,
	\begin{equation}\label{diff omgN}
		d_{f}\omega_N(0)\equiv0.
	\end{equation}
	Let us mention that one could also fix $\omega_N$ and make $\omega_S$ vary.
	The evolution is given by the dynamics of the interface which can be described (in the colatitude/longitude chart) through the following parametrization
	$$z(t,\varphi)=C_{1}\big(\theta_0+f(t,\varphi),\varphi\big)=\begin{pmatrix}
		\sin\big(\theta_0+f(t,\varphi)\big)\cos(\varphi)\\
		\sin\big(\theta_0+f(t,\varphi)\big)\sin(\varphi)\\
		\cos\big(\theta_0+f(t,\varphi)\big)
	\end{pmatrix}.$$
	Differentiating in time amounts to
	$$\partial_{t}z(t,\varphi)=\partial_{t}f(t,\varphi)\begin{pmatrix}
		\cos\big(\theta_0+f(t,\varphi)\big)\cos(\varphi)\\
		\cos\big(\theta_0+f(t,\varphi)\big)\sin(\varphi)\\
		-\sin\big(\theta_0+f(t,\varphi)\big)
	\end{pmatrix},$$
	and the derivation in the longitude variable gives
	$$\partial_{\varphi}z(t,\varphi)=\partial_{\varphi}f(t,\varphi)\begin{pmatrix}
		\cos\big(\theta_0+f(t,\varphi)\big)\cos(\varphi)\\
		\cos\big(\theta_0+f(t,\varphi)\big)\sin(\varphi)\\
		-\sin\big(\theta_0+f(t,\varphi)\big)
	\end{pmatrix}+\begin{pmatrix}
		-\sin\big(\theta_0+f(t,\varphi)\big)\sin(\varphi)\\
		\sin\big(\theta_0+f(t,\varphi)\big)\cos(\varphi)\\
		0
	\end{pmatrix}.$$
	The vector $J\partial_{\varphi}z(t,\varphi)$ can be obtained using the cross product
	\begin{align*}
		J\partial_{\varphi}z(t,\varphi)&=\partial_{\varphi}z(t,\varphi)\times z(t,\varphi)\\
		&=\partial_{\varphi}f(t,\varphi)\begin{pmatrix}
			\sin(\varphi)\\
			-\cos(\varphi)\\
			0
		\end{pmatrix}+\begin{pmatrix}
			\cos\big(\theta_0+f(t,\varphi)\big)\sin\big(\theta_0+f(t,\varphi)\big)\cos(\varphi)\\
			\cos\big(\theta_0+f(t,\varphi)\big)\sin\big(\theta_0+f(t,\varphi)\big)\sin(\varphi)\\
			-\sin^2\big(\theta_0+f(t,\varphi)\big)
		\end{pmatrix}.
	\end{align*}
	Therefore,
	$$\partial_{t}z(t,\varphi)\cdot\big(J\partial_{\varphi}z(t,\varphi)\big)=\sin\big(\theta_0+f(t,\varphi)\big)\partial_{t}f(t,\varphi).$$
	Our ansatz corresponds to a vorticity in the form
	$$\Omega(t,\theta,\varphi)=\omega_{N}(f)\mathbf{1}_{0<\theta<\theta_0+f(t,\varphi)}+\omega_{S}\mathbf{1}_{\theta_0+f(t,\varphi)\leqslant\theta<\pi}+2\widetilde{\gamma}\cos(\theta).$$
	Consequently, according to \eqref{psi-G}, \eqref{Green sph}, \eqref{def D} and \eqref{null avrg vort}, we have
	\begin{align*}
		\Psi\big(t,z(t,\varphi)\big)&=\frac{1}{4\pi}\int_{0}^{2\pi}\int_{0}^{\pi}\log\Big(D\big(\theta_0+f(t,\varphi),\theta',\varphi,\varphi'\big)\Big)\Omega(t,\theta',\varphi')\sin(\theta')d\theta'd\varphi'\\
		&=\frac{\omega_{N}(f)}{4\pi}\int_{0}^{2\pi}\int_{0}^{\theta_0+f(t,\varphi')}\log\Big(D\big(\theta_0+f(t,\varphi),\theta',\varphi,\varphi'\big)\Big)\sin(\theta')d\theta'd\varphi'\\
		&\quad+\frac{\omega_{S}}{4\pi}\int_{0}^{2\pi}\int_{\theta_0+f(t,\varphi')}^{\pi}\log\Big(D\big(\theta_0+f(t,\varphi),\theta',\varphi,\varphi'\big)\Big)\sin(\theta')d\theta'd\varphi'\\
		&\quad+\frac{\widetilde{\gamma}}{4\pi}\int_{0}^{2\pi}\int_{0}^{\pi}\log\Big(D\big(\theta_0+f(t,\varphi),\theta',\varphi,\varphi'\big)\Big)\sin(2\theta')d\theta' d\varphi'.
	\end{align*}
	Remark that the unperturbed stream function can be written as follows
	\begin{align*}
		\Psi_{\textnormal{\tiny{FC}}}(\theta)&=\frac{1}{4\pi}\int_{0}^{2\pi}\int_{0}^{\pi}\log\Big(D\big(\theta,\theta',0,\varphi'\big)\Big)\Omega_{\textnormal{\tiny{FC}}}(\theta')\sin(\theta')d\theta'd\varphi'\\
		&=\frac{\omega_{N}}{4\pi}\int_{0}^{2\pi}\int_{0}^{\theta_0}\log\Big(D\big(\theta,\theta',0,\varphi'\big)\Big)\sin(\theta')d\theta'd\varphi'\\
		&\quad+\frac{\omega_{S}}{4\pi}\int_{0}^{2\pi}\int_{\theta_0}^{\pi}\log\Big(D\big(\theta,\theta',0,\varphi'\big)\Big)\sin(\theta')d\theta'd\varphi'\\
		&\quad+\frac{\widetilde{\gamma}}{4\pi}\int_{0}^{2\pi}\int_{0}^{\pi}\log\Big(D\big(\theta,\theta',0,\varphi'\big)\Big)\sin(2\theta')d\theta' d\varphi'.
	\end{align*}
	Thus, making appeal to Chasles' relation, we can write
	\begin{align*}
		\Psi\big(t,z(t,\varphi)\big)&=\Psi_{\textnormal{\tiny{FC}}}\big(\theta_0+f(t,\varphi)\big)+\Psi_{p}^{[1]}\{f\}\big(\theta_0+f(t,\varphi),\varphi\big)+\Psi_{p}^{[2]}\{f\}\big(\theta_0+f(t,\varphi),\varphi\big)\\
		&\triangleq\Psi\{f\}\big(\theta_0+f(t,\varphi),\varphi\big),
	\end{align*}
	where
	\begin{align*}
		\Psi_{p}^{[1]}\{f\}(\theta,\varphi)&\triangleq \frac{\omega_{N}-\omega_{S}}{4\pi}\int_{0}^{2\pi}\int_{\theta_0}^{\theta_0+f(t,\varphi')}\log\Big(D\big(\theta,\theta',\varphi,\varphi'\big)\Big)\sin(\theta')d\theta'd\varphi',\\
		\Psi_{p}^{[2]}\{f\}(\theta,\varphi)&\triangleq \frac{\omega_{N}(f)-\omega_{N}}{4\pi}\int_{0}^{2\pi}\int_{0}^{\theta_0+f(t,\varphi')}\log\Big(D\big(\theta,\theta',\varphi,\varphi'\big)\Big)\sin(\theta')d\theta'd\varphi'.
	\end{align*}
	Therefore, the vortex cap equation \eqref{vorticity cap equations} becomes
	\begin{equation}\label{vorticity cap eq general ansatz}
		\partial_{t}f(t,\varphi)=\frac{\partial_{\varphi}\Big(\Psi\{f\}\big(\theta_0+f(t,\varphi),\varphi\big)\Big)}{\sin\big(\theta_0+f(t,\varphi)\big)}\cdot
	\end{equation}
	Looking for traveling solutions at speed $c\in\mathbb{R}$ leads to consider
	$$f(t,\varphi)=f(\varphi-ct).$$
	Inserting this into \eqref{vorticity cap eq general ansatz} gives
	\begin{equation}\label{scrF}
		\mathscr{F}(c,f)(\varphi)\triangleq c\,\partial_{\varphi}f(\varphi)+\frac{\partial_{\varphi}\Big(\Psi\{f\}\big(\theta_0+f(\varphi),\varphi\big)\Big)}{\sin\big(\theta_0+f(\varphi)\big)}=0.
	\end{equation}
	Observe, using in particular \eqref{omgN at 0}, that 
	$$\forall c\in\mathbb{R},\quad\mathscr{F}(c,0)=0.$$
	That corresponds to a trivial line of roots of $\mathscr{F}$, corresponding to the flat cap stationary solution associated with the angle $\theta_0$. In order to look for non-trivial roots, we shall use bifurcation arguments in terms of the Crandall-Rabinowitz Theorem. For that goal, we shall study the regularity of $\mathscr{F}$ and the spectral properties of its linearized operator.

	\subsection{Bifurcation study}
	In this section, we shall check the hypothesis of the Crandall-Rabinowitz Theorem for the functional $\mathscr{F}$ introduced in \eqref{scrF}. Firstly, let us introduce the function spaces, in terms of Hölder regularity, that will be used in the bifurcation argument. Fix $\alpha\in(0,1)$, then the Hölder space $C^{\alpha}(\mathbb{T})$ consists in $2\pi$-periodic functions $f:\mathbb{T}\rightarrow\mathbb{R}$ such that the following norm is finite:
	$$\|f\|_{C^{\alpha}(\mathbb{T})}\triangleq\|f\|_{L^{\infty}(\mathbb{T})}+\sup_{(\varphi,\varphi')\in\mathbb{T}^2\atop\varphi\neq\varphi'}\frac{|f(\varphi)-f(\varphi')|}{|\varphi-\varphi'|^\alpha}\cdot$$
	The subspace $C^{1+\alpha}(\mathbb{T})$ of regular functions is associated with the following norm
	$$\|f\|_{C^{1+\alpha}(\mathbb{T})}\triangleq\|f\|_{L^{\infty}(\mathbb{T})}+\|\partial_{\varphi}f\|_{C^{\alpha}(\mathbb{T})}.$$
	Define also the following subspaces taking into account parity and symmetries
	\begin{align*}
		X_{\mathbf{m}}^{1+\alpha}&\triangleq\left\lbrace f\in C^{1+\alpha}(\mathbb{T})\quad\textnormal{s.t.}\quad \forall\varphi\in\mathbb{T},\,f(\varphi)=\sum_{n=1}^{\infty}f_{n}\cos(\mathbf{m}n\varphi),\quad f_n\in\mathbb{R}\right\rbrace,\\
		Y_{\mathbf{m}}^{\alpha}&\triangleq\left\lbrace g\in C^{\alpha}(\mathbb{T})\quad\textnormal{s.t.}\quad \forall\varphi\in\mathbb{T},\,g(\varphi)=\sum_{n=1}^{\infty}g_{n}\sin(\mathbf{m}n\varphi),\quad g_n\in\mathbb{R}\right\rbrace,\\
		B_{r,\mathbf{m}}^{1+\alpha}&\triangleq\Big\{f\in X_{\mathbf{m}}^{1+\alpha}\quad\textnormal{s.t.}\quad\|f\|_{C^{1+\alpha}(\mathbb{T})}<r\Big\},\qquad r>0.
	\end{align*}
	The next proposition gathers the regularity properties for the functional $\mathscr{F}$ and gives the structure of its linearized operator at the flat cap.
	\begin{prop}\label{propreg1}
		Let $\alpha\in(0,1)$ and $\mathbf{m}\in\mathbb{N}^*.$ There exists $r>0$ such that
		\begin{enumerate}[label=(\roman*)]
			\item The function $\mathscr{F}:\mathbb{R}\times B_{r,\mathbf{m}}^{1+\alpha}\rightarrow Y_{\mathbf{m}}^{\alpha}$ is well-defined and of class $C^1.$
			\item The partial derivative $\partial_{c}d_{f}\mathscr{F}:\mathbb{R}\times B_{r,\mathbf{m}}^{1+\alpha}\rightarrow\mathcal{L}(X_{\mathbf{m}}^{1+\alpha},Y_{\mathbf{m}}^{\alpha})$ exists and is continuous.
			\item At the equilibrium $f=0,$ the linearized operator admits the following Fourier representation
			\begin{equation}\label{split dfF2}
				d_{f}\mathscr{F}(c,0)\left[\sum_{n=1}^{\infty}h_{n}\cos(\mathbf{m}n\varphi)\right]=\sum_{n=1}^{\infty}\mathbf{m}n\left[-c-(\omega_{N}-\omega_{S})\frac{\mathbf{m}n-1}{2\mathbf{m}n}+\widetilde{\gamma}\right]h_{n}\sin(\mathbf{m}n\varphi).
			\end{equation}
			In addition, if $c\neq\tfrac{\omega_N-\omega_S}{2}-\widetilde{\gamma},$ then the operator $d_f\mathscr{F}(c,0):X_{\mathbf{m}}^{1+\alpha}\rightarrow Y_{\mathbf{m}}^{\alpha}$ is of Fredholm type with index zero.
		\end{enumerate}
	\end{prop}
	\begin{proof}

		\medskip \noindent
		\textbf{(i)}
		First, notice that the oddness and $\mathbf{m}$-fold properties follow from the evenness and $\mathbf{m}$-fold properties of $f$ and changes of variables in the non-local part. Now, we need to check that $\mathscr{F}(c,f)$ belongs to $C^{\alpha}(\mathbb{T})$ provided that $f\in C^{1+\alpha}(\mathbb{T})$. Then, let us write $\mathscr{F}$ as
		$$\mathscr{F}(c,f)(\varphi)=c f'(\varphi)+ f'(\varphi)\frac{(\partial_\theta \Psi\{f\})\big(\theta_0+f(\varphi),\varphi\big)}{\sin\big(\theta_0+f(\varphi)\big)}+\frac{(\partial_\varphi \Psi\{f\})\big(\theta_0+f(\varphi),\varphi\big)}{\sin\big(\theta_0+f(\varphi)\big)}\cdot$$ 
		Notice that since $\theta_0\notin\{0,\pi\}$ and $\|f\|_{C^{1+\alpha}(\mathbb{T})}<r$, hence by considering $r$ small enough we find
		\begin{equation}\label{lower bound sin th0f}
			\inf_{\varphi\in\mathbb{T}}\big|\sin\big(\theta_0+f(\varphi)\big)\big|\geqslant\delta_0>0,\qquad\delta_{0}\triangleq\inf_{x\in[\theta_0-r,\theta_0+r]}|\sin(x)|\in(0,1).
		\end{equation}
		Thus, in order to check that $\mathscr{F}$ is well-defined, it is enough to prove that
		\begin{equation}\label{condition-welldef}
			(\partial_\theta \Psi\{f\})\big(\theta_0+f(\cdot),\cdot\big)\in C^{\alpha}(\T)\quad\textnormal{and}\quad (\partial_\varphi \Psi\{f\})\big(\theta_0+f(\cdot),\cdot\big)\in C^{\alpha}(\T).
		\end{equation}
		Note that proving \eqref{condition-welldef} uses the same techniques as the one used to show \eqref{condition-c1}. Thus, we shall skip the details and only  check that $f\mapsto (d_f \mathscr{F})$ is continuous using the expression in \eqref{scrF}. Indeed, we can compute the Gateaux derivative of $\mathscr{F}$ and obtain
		\begin{align*}
			d_{f}\mathscr{F}(c,f)[h](\varphi)&=c\,h'(\varphi)-h(\varphi)\frac{\cos\big(\theta_0+f(\varphi)\big)}{\sin^{2}\big(\theta_0+f(\varphi)\big)}\partial_{\varphi}\Big(\Psi\{f\}\big(\theta_0+f(\varphi),\varphi\big)\Big)\\
			&\quad+\frac{1}{\sin\big(\theta_0+f(\varphi)\big)}\partial_{\varphi}\Big(h(\varphi)(\partial_{\theta}\Psi\{f\})\big(\theta_0+f(\varphi),\varphi\big)\Big)\\
			&\quad+\frac{1}{\sin\big(\theta_0+f(\varphi)\big)}\partial_{\varphi}\Big(\big(d_{f}\Psi\{f\}[h]\big)\big(\theta_0+f(\varphi),\varphi\big)\Big),
		\end{align*}
		with
		\begin{align*}
			\big(d_{f}\Psi\{f\}[h]\big)\big(\theta_0+f(\varphi),\varphi\big)&=\big(d_{f}\Psi_p^{[1]}\{f\}[h]\big)\big(\theta_0+f(\varphi),\varphi\big)+\big(d_{f}\Psi_p^{[2]}\{f\}[h]\big)\big(\theta_0+f(\varphi),\varphi\big),\\
			\big(d_{f}\Psi_p^{[1]}\{f\}[h]\big)\big(\theta_0+f(\varphi),\varphi\big)&=\frac{\omega_{N}-\omega_{S}}{4\pi}\int_{0}^{2\pi}h(\varphi')\log\Big( D\big(\theta_0+f(\varphi),\theta_0+f(\varphi'),\varphi,\varphi'\big)\Big)\sin\big(\theta_0+f(\varphi')\big)d\varphi',\\
			\big(d_{f}\Psi_p^{[2]}\{f\}[h]\big)\big(\theta_0+f(\varphi),\varphi\big)&=\frac{d_{f}\omega_N(f)[h]}{4\pi}\int_{0}^{2\pi}\int_{0}^{\theta_0}\log\Big(D\big(\theta_0+f(\varphi),\theta',\varphi,\varphi'\big)\Big)\sin(\theta')d\theta'd\varphi'\\
			&\quad+\frac{\omega_N(f)-\omega_N}{4\pi}\int_{0}^{2\pi}\log\Big(D\big(\theta_0+f(\varphi),\theta_0+f(\varphi'),\varphi,\varphi'\big)\Big)\sin\big(\theta_0+f(\varphi')\big)d\varphi'
		\end{align*}
		and
		\begin{align}
			(\partial_{\theta}\Psi\{f\})(\theta,\varphi)=&\frac{\omega_{N}(f)}{4\pi}\int_{0}^{2\pi}\int_{0}^{\theta_0+f(\varphi')}\frac{\partial_\theta D(\theta,\theta',\varphi,\varphi')\sin(\theta')}{D(\theta,\theta',\varphi,\varphi')}d\theta'd\varphi'\nonumber\\
			&+\frac{\omega_{S}}{4\pi}\int_{0}^{2\pi}\int_{\theta_0+f(\varphi')}^{\pi}\frac{\partial_\theta D(\theta,\theta',\varphi,\varphi')\sin(\theta')}{D(\theta,\theta',\varphi,\varphi')}d\theta'd\varphi'\nonumber\\
			&+\frac{\widetilde{\gamma}}{4\pi}\int_{0}^{2\pi}\int_{0}^{\pi}\frac{\partial_\theta D(\theta,\theta',\varphi,\varphi')\sin(2\theta')}{D(\theta,\theta',\varphi,\varphi')}d\theta'd\varphi'.\label{dthPsi}
		\end{align}
		Notice that $d_f \mathscr{F}(c,f)$ is continuous in $f$ provided that the functions
		\begin{equation}\label{condition-c1}
			f\mapsto \partial_\varphi \Big((\partial_\theta \Psi)\{f\}\big(\theta_0+f(\varphi), \varphi\big)\Big) \quad\textnormal{and}\quad f\mapsto \partial_\varphi \Big(d_f \Psi\{f\}[h]\big(\theta_0+f(\varphi),\varphi\big)\Big)
		\end{equation}
		are continuous.  Let us start with the first condition in \eqref{condition-c1} and, since the analysis is similar, we shall only give details for just one of the terms. Note that from \eqref{dthPsi}, we can write
		\begin{align*}
			(\partial_\theta \Psi\{f\})\big(\theta_0+f(\varphi),\varphi\big)=&\frac{\omega_{N}(f)}{4\pi}\int_{0}^{2\pi}\int_{0}^{\theta_0+f(\varphi')}\frac{\partial_\theta D\big(\theta_0+f(\varphi),\theta',\varphi,\varphi'\big)\sin(\theta')}{D\big(\theta_0+f(\varphi),\theta',\varphi,\varphi'\big)}d\theta'd\varphi'\\
			&+\frac{\omega_{S}}{4\pi}\int_{0}^{2\pi}\int_{\theta_0+f(\varphi')}^{\pi}\frac{\partial_\theta D\big(\theta_0+f(\varphi),\theta',\varphi,\varphi'\big)\sin(\theta')}{D\big(\theta_0+f(\varphi),\theta',\varphi,\varphi'\big)}d\theta'd\varphi'\\
			&+\frac{\widetilde{\gamma}}{4\pi}\int_{0}^{2\pi}\int_{0}^{\pi}\frac{\partial_\theta D\big(\theta_0+f(\varphi),\theta',\varphi,\varphi')\sin(2\theta'\big)}{D\big(\theta_0+f(\varphi\big),\theta',\varphi,\varphi')}d\theta'd\varphi'\\
			\triangleq & J_1\{f\}(\varphi)+J_2\{f\}(\varphi)+J_3\{f\}(\varphi).
		\end{align*}
		We focus on the first term $J_1$. Note that to check that the first condition in \eqref{condition-c1} is satisfied, we need to compute $\partial_\varphi J_1\{f\}$. However, let us simplify $J_1$ before differentiating. Observe that taking the derivative of \eqref{expr D} leads to
		\begin{align}
			\partial_{\theta}D(\theta,\theta',\varphi,\varphi')&=2\Big[\sin\left(\tfrac{\theta-\theta'}{2}\right)\cos\left(\tfrac{\theta-\theta'}{2}\right)+\cos(\theta)\sin(\theta')\sin^{2}\left(\tfrac{\varphi-\varphi'}{2}\right)\Big],\label{dDth}\\
			\partial_{\theta'}D(\theta,\theta',\varphi,\varphi')&=2\Big[-\sin\left(\tfrac{\theta-\theta'}{2}\right)\cos\left(\tfrac{\theta-\theta'}{2}\right)+\sin(\theta)\cos(\theta')\sin^{2}\left(\tfrac{\varphi-\varphi'}{2}\right)\Big].\label{dDthp}
		\end{align}
		Hence,
		\begin{align*}
			(\partial_{\theta}D+\partial_{\theta'}D)(\theta,\theta',\varphi,\varphi')&=2\big[\cos(\theta)\sin(\theta')+\sin(\theta)\cos(\theta')\big]\sin^2\left(\tfrac{\varphi-\varphi'}{2}\right)\\
			&=2\sin(\theta+\theta')\sin^2\left(\tfrac{\varphi-\varphi'}{2}\right).
		\end{align*}
		Thus, adding and substracting $\partial_{\theta'}D$ appropriately and integrating by parts, we find
		\begin{align*}
			J_1\{f\}(\varphi)=&\frac{\omega_{N}(f)}{4\pi}\int_{0}^{2\pi}\int_{0}^{\theta_0+f(\varphi')}\frac{\sin^2\left(\frac{\varphi-\varphi'}{2}\right)\sin\big(\theta'+\theta_0+f(\varphi)\big)}{D(\theta_0+f(\varphi),\theta',\varphi,\varphi')}\sin(\theta')d\theta'd\varphi'\\
			&-\frac{\omega_{N}(f)}{4\pi}\int_{0}^{2\pi}\int_{0}^{\theta_0+f(\varphi')}\frac{\partial_{\theta'} D\big(\theta_0+f(\varphi),\theta',\varphi,\varphi'\big)}{D\big(\theta_0+f(\varphi),\theta',\varphi,\varphi'\big)}\sin(\theta')d\theta'd\varphi'\\
			=&\frac{\omega_{N}(f)}{2\pi}\int_{0}^{2\pi}\int_{0}^{\theta_0+f(\varphi')}\frac{\sin^2\left(\frac{\varphi-\varphi'}{2}\right)\sin\big(\theta'+\theta_0+f(\varphi)\big)}{D\big(\theta_0+f(\varphi),\theta',\varphi,\varphi'\big)}\sin(\theta')d\theta'd\varphi'\\
			&-\frac{\omega_{N}(f)}{4\pi}\int_{0}^{2\pi}\log\Big(D\big(\theta_0+f(\varphi),\theta_0+f(\varphi'),\varphi,\varphi'\big)\Big)\sin\big(\theta_0+f(\varphi')\big)d\varphi'\\
			&+\frac{\omega_{N}(f)}{4\pi}\int_{0}^{2\pi}\int_{0}^{\theta_0+f(\varphi')}\log\Big( D\big(\theta_0+f(\varphi),\theta',\varphi,\varphi'\big)\Big)\cos(\theta')d\theta'd\varphi'\\
			\triangleq & J_{1,1}\{f\}(\varphi)+J_{1,2}\{f\}(\varphi)+J_{1,3}\{f\}(\varphi).
		\end{align*}
		Let us work with $J_{1,1}\{f\}$, which is the most singular term. Now, we differentiate in $\varphi$ and obtain
		\begin{align*}
			\partial_\varphi& J_{1,1}\{f\}(\varphi)=\frac{\omega_{N}(f)}{2\pi}\int_{0}^{2\pi}\int_{0}^{\theta_0+f(\varphi')}\frac{f'(\varphi)\cos\big(\theta'+\theta_0+f(\varphi)\big)\sin^{2}\left(\frac{\varphi-\varphi'}{2}\right)}{D(\theta_0+f(\varphi),\theta',\varphi,\varphi')}\sin(\theta')d\theta'd\varphi'\\
			&+\frac{\omega_{N}(f)}{4\pi}\int_{0}^{2\pi}\int_{0}^{\theta_0+f(\varphi')}\frac{\sin\big(\theta'+\theta_0+f(\varphi)\big)\sin(\varphi-\varphi')}{D(\theta_0+f(\varphi),\theta',\varphi,\varphi')}\sin(\theta')d\theta'd\varphi'\\
			&-\frac{\omega_{N}(f)}{2\pi}\int_{0}^{2\pi}\int_{0}^{\theta_0+f(\varphi')}\frac{\sin\big(\theta'+\theta_0+f(\varphi)\big)\sin^{2}\left(\frac{\varphi-\varphi'}{2}\right)}{D^2\big(\theta_0+f(\varphi),\theta',\varphi,\varphi'\big)}\partial_{\varphi} \Big(D\big(\theta_0+f(\varphi),\theta',\varphi,\varphi'\big)\Big)\sin(\theta')d\theta'd\varphi'\\
			\triangleq & J_{1,1,1}\{f\}(\varphi)+J_{1,1,2}\{f\}(\varphi)+J_{1,1,3}\{f\}(\varphi).
		\end{align*}
		Notice that the most singular integral is $J_{1,1,3}$.  Thus, we only deal with that term. Remark that
		\begin{align*}
			\partial_{\varphi} \Big(D\big(\theta_0+f(\varphi),\theta',\varphi,\varphi'\big)\Big)&=f'(\varphi)\partial_{\theta}D\big(\theta_0+f(\varphi),\theta',\varphi,\varphi'\big)+\partial_{\varphi}D\big(\theta_0+f(\varphi),\theta',\varphi,\varphi'\big)\\
			&=2f'(\varphi)\Big[\tfrac{1}{2}\sin\big(\theta_0+f(\varphi)-\theta'\big)+\cos\big(\theta_0+f(\varphi)\big)\sin(\theta')\sin^2\left(\tfrac{\varphi-\varphi'}{2}\right)\Big]\\
			&\quad+\sin\big(\theta_0+f(\varphi)\big)\sin(\theta')\sin(\varphi-\varphi').
		\end{align*}
		We can make the change of variables $\theta'=t(\theta_0+f(\varphi'))$ to simplify the integral obtaining
		$$J_{1,1,3}\{f\}(\varphi)=-\frac{1}{2\pi}\int_{0}^{2\pi}\int_{0}^{1}\mathbb{K}\{f\}(t,\varphi,\varphi')dtd\varphi',$$
		where
		\begin{align*}
			\mathbb{K}\{f\}(t,\varphi,\varphi')&\triangleq\omega_{N}(f)\frac{\sin\big((1+t)\theta_0+f(\varphi)+tf(\varphi')\big)\sin^2\left(\frac{\varphi-\varphi'}{2}\right)}{D^2\big(\theta_0+f(\varphi),t(\theta_0+f(\varphi')),\varphi,\varphi'\big)}\sin\big(t(\theta_0+f(\varphi'))\big)\big(\theta_0+f(\varphi')\big)\\
			&\times\Big\{2f'(\varphi)\Big[\tfrac{1}{2}\sin\big((1-t)\theta_0+f(\varphi)-tf(\varphi')\big)+\cos\big(\theta_0+f(\varphi)\big)\sin\big(t(\theta_0+f(\varphi'))\big)\sin^2\left(\tfrac{\varphi-\varphi'}{2}\right)\Big]\\
			&\quad+\sin\big(\theta_0+f(\varphi)\big)\sin\big(t(\theta_0+f(\varphi))\big)\sin(\varphi-\varphi')\Big\}.
		\end{align*}
		Our goal is to check that $f\mapsto J_{1,1,3}\{f\}$ is continuous. For this aim, we take $f_1, f_2\in B_{r,m}^{1+\alpha}$ and we estimate the difference at those points
		\begin{align*}
			J_{1,1,3}\{f_2\}(\varphi)-J_{1,1,3}\{f_1\}(\varphi).
		\end{align*}
		In order to simplify, let us illustrate one of the terms since the order of singularity is the same at every term. For that, define
		$$\tilde{J}_{1,1,3}\{f_1,f_2\}(\varphi)\triangleq\int_{0}^{2\pi}\int_{0}^{1}\mathbb{K}\{f_1,f_2\}(t,\varphi,\varphi')dtd\varphi',$$
		where
		
		\begin{align*}
			\mathbb{K}\{f_1,f_2\}(t,\varphi,\varphi')&\triangleq\frac{\sin\big((1+t)\theta_0+f_2(\varphi)+tf_2(\varphi')\big)\sin^2\left(\frac{\varphi-\varphi'}{2}\right)}{D^2\big(\theta_0+f_2(\varphi),t(\theta_0+f_2(\varphi')),\varphi,\varphi'\big)}\sin^2\big(t(\theta_0+f_2(\varphi'))\big)\big(\theta_0+f_2(\varphi')\big)\sin(\varphi-\varphi')\\
			&\quad\times\sin\big(\theta_0+f_2(\varphi)\big)\Big[\omega_{N}(f_2)-\omega_N(f_1)\Big].
		\end{align*}
		To estimate the previous term in the H\"older space $C^\alpha(\mathbb{T})$, we use Proposition \ref{prop-potentialtheory} with kernel
		$$K(\varphi,\varphi')\triangleq\int_{0}^{1}\mathbb{K}\{f_1,f_2\}(t,\varphi,\varphi')dt.$$
		Remark that, using the smoothness of $\omega_N(\cdot)$, the mean value theorem implies
		$$\Big|\omega_{N}(f_2)-\omega_N(f_1)\Big|\leqslant C\|f_1-f_2\|_{L^{\infty}(\mathbb{T})}.$$ 
		Now, we choose 
		$$\begin{cases}
			r<\tfrac{1}{2}\min(\theta_0,|\theta_0-\pi|), & \textnormal{if }\theta_0\neq\tfrac{\pi}{2},\\
			r<\tfrac{\pi}{4}, & \textnormal{if } \theta_0=\tfrac{\pi}{2}
		\end{cases}$$
		and denote 
		$$\mathtt{m}_{\theta_0}(r)\triangleq\begin{cases}
			\theta_0-r, & \textnormal{if }\theta_0\leqslant\tfrac{\pi}{2},\\
			\theta_0+r, & \textnormal{if }\theta_0>\tfrac{\pi}{2},
		\end{cases}\qquad \mathtt{M}_{\theta_0}(r)\triangleq\begin{cases}
			\theta_0-r, & \textnormal{if }\theta_0>\tfrac{\pi}{2},\\
			1, & \textnormal{if }\theta_0=\tfrac{\pi}{2},\\
			\theta_0+r, & \textnormal{if }\theta_0<\tfrac{\pi}{2}\cdot
		\end{cases}$$
		Notice that $\mathtt{m}_{\theta_0}(r)\in(0,\pi)$ and $\mathtt{M}_{\theta_0}(r)\in(0,\pi)$, thus, a convexity argument ensures that
		$$\exists C_1,C_2>0,\quad\forall t\in[0,1],\quad C_1 t\leqslant\sin\big(t\,\mathtt{m}_{\theta_0}(r)\big)\leqslant C_2 t,\qquad C_1 t\leqslant\sin\big(t\,\mathtt{M}_{\theta_0}(r)\big)\leqslant C_2 t.$$
		A direct estimation gives
		\begin{align*}
			\big|K(\varphi,\varphi')\big|&\leqslant C\|f_1-f_2\|_{L^{\infty}(\mathbb{T})}\int_{0}^{1}\frac{|\sin(\varphi-\varphi')|\sin^2\left(\frac{\varphi-\varphi'}{2}\right)}{D^2\big(\theta_0+f_2(\varphi),t(\theta_0+f_2(\varphi')),\varphi,\varphi'\big)}\sin^2\big(t\,\mathtt{M}_{\theta_0}(r)\big)dt\\
			&\leqslant C\|f_1-f_2\|_{L^{\infty}(\mathbb{T})}\int_{0}^{1}\frac{|\sin(\varphi-\varphi')|\sin^2\left(\frac{\varphi-\varphi'}{2}\right)t^2}{D^2\big(\theta_0+f_2(\varphi),t(\theta_0+f_2(\varphi')),\varphi,\varphi'\big)}dt.
		\end{align*}
		We need to control the denominator. For any $t\in[0,1]$, we have
		\begin{equation}\label{D square}
				\begin{aligned}
					&D^\frac12\big(\theta_0+f_2(\varphi),t(\theta_0+f_2(\varphi')),\varphi,\varphi'\big)\\
					&=\Big\{2\sin^2\left(\tfrac{(1-t)\theta_0+f_2(\varphi)-tf_2(\varphi')}{2}\right)+2\sin\big(\theta_0+f_2(\varphi)\big)\sin\big(t(\theta_0+f_2(\varphi'))\big)\sin^2\left(\tfrac{\varphi-\varphi'}{2}\right)\Big\}^\frac12\\
					&\geqslant \left\{\sqrt{2}\left|\sin\left(\tfrac{(1-t)\theta_0+f_2(\varphi)-tf_2(\varphi')}{2}\right)\right|+C\sqrt{t}\left|\sin\left(\tfrac{\varphi-\varphi'}{2}\right)\right|\right\}.
				\end{aligned}
			\end{equation}
	Let us now estimate the first term. For $r$ small enough, we have
		$$0<(1-t)\theta_0-2r\leqslant(1-t)\theta_0+f_2(\varphi)-tf_2(\varphi')\leqslant\theta_0+2\|f_2\|_{L^{\infty}(\mathbb{T})}\leqslant\theta_0+2r<\pi.$$
		 Using the concavity of the function $\sin$ on $(0,\tfrac{\pi}{2}),$ we infer
		\begin{align*}
			\sin\left(\tfrac{(1-t)\theta_0+f_2(\varphi)-tf_2(\varphi')}{2}\right)			&\geqslant C\Big((1-t)\big(\theta_0+f_2(\varphi)\big)+t\big(f_2(\varphi)-f_2(\varphi')\big)\Big)\\
			&\geqslant C_2 (1-t)-t\|f_2\|_{\textnormal{\tiny{Lip}}}\left|\sin\left(\tfrac{\varphi-\varphi'}{2}\right)\right|.
		\end{align*}
	Inserting this into \eqref{D square} we arrive to 
	\begin{align*}
					D^\frac12\big(\theta_0+f_2(\varphi),t(\theta_0+f_2(\varphi')),\varphi,\varphi'\big)
					&\geqslant C\left\{(1-t)+(\sqrt{t}-\|f_2\|_{\textnormal{Lip}} t)\left|\sin\left(\tfrac{\varphi-\varphi'}{2}\right)\right|\right\}\\
					&\geqslant C\left\{(1-t)+\sqrt{t}\left|\sin\left(\tfrac{\varphi-\varphi'}{2}\right)\right|\right\},
				\end{align*}
	where in the last inequality we are using that $t\leqslant1$ and that $\|f_2\|_{\textnormal{Lip}}$ is small enough. Finally, we find 
	$$D\big(\theta_0+f_2(\varphi),t(\theta_0+f_2(\varphi')),\varphi,\varphi'\big)\geqslant C\Big[(1-t)^2+t\sin^2\left(\tfrac{\varphi-\varphi'}{2}\right)\Big],$$
	implying
		$$D^2\big(\theta_0+f_2(\varphi),t(\theta_0+f_2(\varphi')),\varphi,\varphi'\big)\geqslant C\Big[(1-t)^2+t\sin^2\left(\tfrac{\varphi-\varphi'}{2}\right)\Big]^2.$$
		Putting together the foregoing calculations yields
		$$\big|K(\varphi,\varphi')\big|\leqslant C\|f_1-f_2\|_{L^{\infty}(\mathbb{T})}\int_{0}^{1}\frac{|\sin(\varphi-\varphi')|\sin^2\left(\frac{\varphi-\varphi'}{2}\right)t^2}{\Big[(1-t)^2+t\sin^2\left(\frac{\varphi-\varphi'}{2}\right)\Big]^2}dt.$$
		Now, observe that
		$$\frac{t^2\sin^2\left(\frac{\varphi-\varphi'}{2}\right)}{(1-t)^2+t\sin^2\left(\frac{\varphi-\varphi'}{2}\right)}\leqslant Ct$$
		and then
		$$\big|K(\varphi,\varphi')\big|\leqslant C\|f_1-f_2\|_{L^{\infty}(\mathbb{T})}\int_{0}^{1}\frac{t|\sin(\varphi-\varphi')|}{(1-t)^2+t\sin^2\left(\frac{\varphi-\varphi'}{2}\right)}dt.$$
		Next, use that for any $t\in[0,1],$
		\begin{align*}
			\frac{t|\sin(\varphi-\varphi')|}{(1-t)^2+t\sin^2\left(\frac{\varphi-\varphi'}{2}\right)}&\leqslant\frac{t|\sin(\varphi-\varphi')|}{\Big[(1-t)^2+t\sin^2\left(\frac{\varphi-\varphi'}{2}\right)\Big]^{\frac{1}{2}}\sqrt{t}\left|\sin\left(\tfrac{\varphi-\varphi'}{2}\right)\right|}\\
			&\leqslant\Big[(1-t)^2+t\sin^2\left(\frac{\varphi-\varphi'}{2}\right)\Big]^{-\frac{1}{2}}\\
			&\leqslant C\Big[|1-t|+\sqrt{t}\left|\sin\left(\tfrac{\varphi-\varphi'}{2}\right)\right|\Big]^{-1}.
		\end{align*}
		The last inequality follows from the following classical estimate
		$$\forall (a,b)\in(\mathbb{R}_+)^2,\quad\sqrt{a^2+b^2}\geqslant\tfrac{1}{\sqrt{2}}(a+b).$$
		This implies in turn
		\begin{align*}
			|K(\varphi,\varphi')|&\leqslant C\|f_1-f_2\|_{L^\infty(\mathbb{T})} \int_0^{1} \Big[|1-t|+\sqrt{t}\left|\sin\left(\tfrac{\varphi-\varphi'}{2}\right)\right|\Big]^{-1}dt\\
			&\leqslant C\|f_1-f_2\|_{L^\infty(\mathbb{T})} \int_0^{1}\left|\sin\left(\tfrac{\varphi-\varphi'}{2}\right)\right|^{-(1-\alpha)}t^{-\frac{1-\alpha}{2}}|1-t|^{-\alpha}dt\\
			&\leqslant C\|f_1-f_2\|_{L^\infty(\mathbb{T})} \left|\sin\left(\tfrac{\varphi-\varphi'}{2}\right)\right|^{-(1-\alpha)},
		\end{align*}
		where we have used the classical interpolation estimate
		$$\forall\alpha\in(0,1),\quad\forall(a,b)\in(\mathbb{R}_+)^2,\quad(a+b)^{-1}\leqslant a^{-\alpha}b^{-(1-\alpha)}.$$
		The above computations allow to conclude that the hypothesis \eqref{prop-potentialtheory-h0} is checked for the kernel $K$. Similarly, we can check that \eqref{prop-potentialtheory-h2} is satisfied and hence Proposition \ref{prop-potentialtheory} can be applied obtaining the continuity in $f$. Let us continue with the second condition in \eqref{condition-c1}. We only treat the case of $\Psi_p^{[1]},$ the other one being similar. We can write 
		\begin{align*}
			\partial_\varphi \Big(d_f \Psi_p^{[1]}\{f\}[h]\big(\theta_0+f(t,\varphi),\varphi\big)\Big)&=\frac{\omega_N-\omega_S}{4\pi}\int_0^{2\pi}\frac{\partial_\varphi \big[D\big(\theta_0+f(\varphi),\theta_0+f(\varphi')\big)\big]}{D\big(\theta_0+f(\varphi),\theta_0+f(\varphi')\big)}\sin\big(\theta_0+f(\varphi')\big)h(\varphi')d\varphi'.
		\end{align*}
		By adding and subtracting $\partial_{\varphi'}\big[D\big(\theta_0+f(\varphi),\theta_0+f(\varphi')\big)\big]$ appropriately and integrating by parts, we find
		\begin{align*}
			&\partial_\varphi \Big(d_f \Psi_p^{[1]}\{f\}[h]\big(\theta_0+f(\varphi),\varphi\big)\Big)\\
			&=\frac{\omega_N-\omega_S}{4\pi}\int_0^{2\pi}\frac{\partial_\varphi \big[D\big(\theta_0+f(\varphi),\theta_0+f(\varphi'),\varphi,\varphi'\big)\big]+\partial_{\varphi'} \big[D\big(\theta_0+f(\varphi),\theta_0+f(\varphi'),\varphi,\varphi'\big)\big]}{D\big(\theta_0+f(\varphi),\theta_0+f(\varphi'),\varphi,\varphi'\big)}\sin\big(\theta_0+f(\varphi')\big)h(\varphi')d\varphi'\\
			&\quad-\frac{\omega_N-\omega_S}{4\pi}\int_0^{2\pi}\frac{\partial_{\varphi'} \big[D\big(\theta_0+f(\varphi),\theta_0+f(\varphi'),\varphi,\varphi'\big)\big]}{D\big(\theta_0+f(\varphi),\theta_0+f(\varphi'),\varphi,\varphi'\big)}\sin\big(\theta_0+f(\varphi')\big)h(\varphi')d\varphi'\\
			&=\frac{\omega_N-\omega_S}{4\pi}\int_0^{2\pi}\frac{\partial_\varphi \big[D\big(\theta_0+f(\varphi),\theta_0+f(\varphi'),\varphi,\varphi'\big)\big]+\partial_{\varphi'} \big[D\big(\theta_0+f(\varphi),\theta_0+f(\varphi'),\varphi,\varphi'\big)\big]}{D\big(\theta_0+f(\varphi),\theta_0+f(\varphi'),\varphi,\varphi'\big)}\sin\big(\theta_0+f(\varphi')\big)h(\varphi')d\varphi'\\
			&\quad+\frac{\omega_N-\omega_S}{4\pi}\int_0^{2\pi}\log\Big( D\big(\theta_0+f(\varphi),\theta_0+f(\varphi'),\varphi,\varphi'\big)\Big)\partial_{\varphi'}\Big(\sin\big(\theta_0+f(\varphi')\big)h(\varphi')\Big)d\varphi'.
		\end{align*}
		Using the definition of $D$ in \eqref{def D}, we infer
		\begin{align*}
			\partial_{\varphi}D(\theta,\theta',\varphi,\varphi')&=\sin(\theta)\sin(\theta')\sin(\varphi-\varphi')=-\partial_{\varphi'}D(\theta,\theta',\varphi,\varphi').
		\end{align*}
		Combined with \eqref{dDth}-\eqref{dDthp}, we get
		\begin{align*}
			&\partial_{\varphi}\big[D\big(\theta_0+f(\varphi),\theta_0+f(\varphi'),\varphi,\varphi'\big)\big]+\partial_{\varphi'}\big[D\big(\theta_0+f(\varphi),\theta_0+f(\varphi'),\varphi,\varphi'\big)\big]\\
			&=f'(\varphi)\partial_{\theta}D\big(\theta_0+f(\varphi),\theta_0+f(\varphi'),\varphi,\varphi'\big)+f'(\varphi')\partial_{\theta'}D\big(\theta_0+f(\varphi),\theta_0+f(\varphi'),\varphi,\varphi'\big)\\
			&=\sin\big(f(\varphi)-f(\varphi')\big)\big(f'(\varphi)-f'(\varphi')\big)+2f'(\varphi)\cos\big(\theta_0+f(\varphi)\big)\sin\big(\theta_0+f(\varphi')\big)\sin^2\left(\tfrac{\varphi-\varphi'}{2}\right)\\
			&\quad+2f'(\varphi')\cos\big(\theta_0+f(\varphi')\big)\sin\big(\theta_0+f(\varphi)\big)\sin^2\left(\tfrac{\varphi-\varphi'}{2}\right).
		\end{align*}
		As a consequence, we obtain
		\begin{align*}
			&\partial_\varphi \Big( d_f \Psi_p^{[1]}\{f\}[h]\big(\theta_0+f(t,\varphi),\varphi\big)\Big)\\
			=&\frac{\omega_N-\omega_S}{4\pi}\int_0^{2\pi}\frac{\sin\big(f(\varphi)-f(\varphi')\big)\big(f'(\varphi)-f'(\varphi')\big)}{D\big(\theta_0+f(\varphi),\theta_0+f(\varphi'),\varphi,\varphi'\big)}\sin\big(\theta_0+f(\varphi')\big)h(\varphi')d\varphi'\\
			&+\frac{\omega_N-\omega_S}{2\pi}\int_0^{2\pi}\frac{\cos\big(\theta_0+f(\varphi)\big)f'(\varphi)\sin\big(\theta_0+f(\varphi')\big)\sin^2\left(\frac{\varphi-\varphi'}{2}\right)}{D\big(\theta_0+f(\varphi),\theta_0+f(\varphi'),\varphi,\varphi'\big)}\sin\big(\theta_0+f(\varphi')\big)h(\varphi')d\varphi'\\
			&+\frac{\omega_N-\omega_S}{2\pi}\int_0^{2\pi}\frac{\cos(\theta_0+f(\varphi'))f'(\varphi')\sin(\theta_0+f(\varphi))\sin^2\left(\frac{\varphi-\varphi'}{2}\right)}{D\big(\theta_0+f(\varphi),\theta_0+f(\varphi'),\varphi,\varphi'\big)}\sin\big(\theta_0+f(\varphi')\big)h(\varphi')d\varphi'\\
			&+\frac{\omega_N-\omega_S}{4\pi}\int_0^{2\pi}\log\Big( D\big(\theta_0+f(\varphi),\theta_0+f(\varphi'),\varphi,\varphi'\big)\Big)\partial_{\varphi'}\left(\sin\big(\theta_0+f(\varphi')\big)h(\varphi')\right)d\varphi'\\
			\triangleq &\frac{\omega_N-\omega_S}{4\pi}\big[ I_1\{f\}h(\varphi)+I_2\{f\}h(\varphi)+I_3\{f\}h(\varphi)+I_4\{f\}h(\varphi)\big].
		\end{align*}
		In the following, we show that $f\mapsto I_i\{f\}$ is continuous by showing that it has a modulus of continuity. Let us just give the details for $I_1$ and the others follow similarly. For that, take $f_1,f_2\in B_{r,m}^{1+\alpha}$, and  estimate
		\begin{align*}
			I_1\{f_1\}h(\varphi)-I_1\{f_2\}h(\varphi)=&
			\int_0^{2\pi}K_1\{f_1,f_2\}(\varphi,\varphi')\big(f_1'(\varphi)-f_1'(\varphi')\big)h(\varphi')d\varphi'\\
			&+\int_0^{2\pi}K_2\{f_1,f_2\}(\varphi,\varphi')\big((f_1-f_2)'(\varphi)-(f_1-f_2)'(\varphi')\big)h(\varphi')d\varphi'\\
			&+\int_0^{2\pi}K_3\{f_1,f_2\}(\varphi,\varphi')\big(f_2'(\varphi)-f_2'(\varphi')\big)h(\varphi')d\varphi',
		\end{align*}
		where
		\begin{align*}
			K_1\{f_1,f_2\}(\varphi,\varphi')&\triangleq\frac{\sin\big(f_1(\varphi)-f_1(\varphi')\big)-\sin\big(f_2(\varphi)-f_2(\varphi')\big)}{D\big(\theta_0+f_1(\varphi),\theta_0+f_1(\varphi'),\varphi,\varphi'\big)}\sin\big(\theta_0+f_1(\varphi')\big),\\
			K_2\{f_1,f_2\}(\varphi,\varphi')&\triangleq\frac{\sin\big(f_2(\varphi)-f_2(\varphi')\big)}{D\big(\theta_0+f_1(\varphi),\theta_0+f_1(\varphi'),\varphi,\varphi'\big)}\sin\big(\theta_0+f_1(\varphi')\big),\\
			K_3\{f_1,f_2\}(\varphi,\varphi')&\triangleq\frac{\sin\big(f_2(\varphi)-f_2(\varphi')\big)}{D\big(\theta_0+f_1(\varphi),\theta_0+f_1(\varphi'),\varphi,\varphi'\big)}\Big[\sin\big(\theta_0+f_1(\varphi')\big)-\sin\big(\theta_0+f_2(\varphi')\big)\Big]\\
			&\quad+\frac{\sin\big(f_2(\varphi)-f_2(\varphi')\big)}{D\big(\theta_0+f_1(\varphi),\theta_0+f_1(\varphi'),\varphi,\varphi'\big)D\big(\theta_0+f_2(\varphi),\theta_0+f_2(\varphi'),\varphi,\varphi'\big)}\\
			&\quad\times \Big[D\big(\theta_0+f_2(\varphi),\theta_0+f_2(\varphi'),\varphi,\varphi'\big)-D\big(\theta_0+f_1(\varphi),\theta_0+f_1(\varphi'),\varphi,\varphi'\big)\Big]\sin\big(\theta_0+f_2(\varphi')\big).
		\end{align*}
		Since the kernel of the integral operator has a non differentiable term, our purpose is to use Proposition \ref{prop-potentialtheory-2}. For that, let us estimate each kernel $K_i$. First note that using \eqref{def D} and \eqref{lower bound sin th0f}, we get
		\begin{align}\label{D-estim}
			D\big(\theta_0+f(\varphi),\theta_0+f(\varphi'),\varphi,\varphi'\big)&\geqslant 2\sin\big(\theta_0+f(\varphi)\big)\sin\big(\theta_0+f(\varphi')\big)\sin\left(\tfrac{\varphi-\varphi'}{2}\right)\nonumber\\
			&\geqslant 2\delta_0^2 \sin^2\left(\tfrac{\varphi-\varphi'}{2}\right).
		\end{align}
		Using \eqref{D-estim}, it is easy to check that for $K_1$ we get 
		\begin{align*}
			|K_1\{f_1,f_2\}(\varphi,\varphi')|&\leqslant C\big|(f_1-f_2)(\varphi)-(f_1-f_2)(\varphi')\big|\left|\sin\left(\tfrac{\varphi-\varphi'}{2}\right)\right|^{-2}\\
			&\leqslant C\|f_1-f_2\|_{C^{1+\alpha}(\mathbb{T})}\left|\sin\left(\tfrac{\varphi-\varphi'}{2}\right)\right|^{-1}
		\end{align*}
		and
		\begin{align*}
			|\partial_\varphi K_1\{f_1,f_2\}(\varphi,\varphi')|\leqslant C\|f_1-f_2\|_{C^{1+\alpha}(\mathbb{T})}\left|\sin\left(\tfrac{\varphi-\varphi'}{2}\right)\right|^{-2}.
		\end{align*}
		Similarly for $K_2$ we obtain 
		\begin{align*}
			|K_2\{f_1,f_2\}(\varphi,\varphi')|\leqslant C\left|\sin\left(\tfrac{\varphi-\varphi'}{2}\right)\right|^{-1},\\
			|\partial_\varphi K_2\{f_1,f_2\}(\varphi,\varphi')|\leqslant C\left|\sin\left(\tfrac{\varphi-\varphi'}{2}\right)\right|^{-2}.
		\end{align*}
		Finally, we shall work with the last kernel $K_3$. Note that
		\begin{align*}
			&D\big(\theta_0+f_2(\varphi),\theta_0+f_2(\varphi'),\varphi,\varphi'\big)-D\big(\theta_0+f_1(\varphi),\theta_0+f_1(\varphi'),\varphi,\varphi'\big)\\
			&=2\sin^2\left(\tfrac{f_2(\varphi)-f_2(\varphi')}{2}\right)-2\sin^2\left(\tfrac{f_1(\varphi)-f_1(\varphi')}{2}\right)\\
			&\quad+2\sin^2\left(\tfrac{\varphi-\varphi'}{2}\right)\Big\{\sin\big(\theta_0+f_2(\varphi)\big)\sin\big(\theta_0+f_2(\varphi')\big)-\sin\big(\theta_0+f_1(\varphi)\big)\sin\big(\theta_0+f_1(\varphi')\big)\Big\}\\
			&\leqslant 2\left|\sin\left(\tfrac{f_2(\varphi)-f_2(\varphi')}{2}\right)\right|\left|\sin\left(\tfrac{f_2(\varphi)-f_2(\varphi')}{2}\right)-\sin\left(\tfrac{f_1(\varphi)-f_1(\varphi')}{2}\right)\right|\\
			&\quad+2\left|\sin\left(\tfrac{f_1(\varphi)-f_1(\varphi')}{2}\right)\right|\left|\sin\left(\tfrac{f_2(\varphi)-f_2(\varphi')}{2}\right)-\sin\left(\tfrac{f_1(\varphi)-f_1(\varphi')}{2}\right)\right|\\
			&\quad+2\sin^2\left(\tfrac{\varphi-\varphi'}{2}\right)\left|\sin\big(\theta_0+f_2(\varphi)\big)\right|\left|\sin\big(\theta_0+f_2(\varphi')\big)-\sin\big(\theta_0+f_1(\varphi')\big)\right|\\
			&\quad+2\sin^2\left(\tfrac{\varphi-\varphi'}{2}\right)\left|\sin\big(\theta_0+f_1(\varphi')\big)\right|\left|\sin\big(\theta_0+f_2(\varphi)\big)-\sin\big(\theta_0+f_1(\varphi)\big)\right|\\
			&\leqslant C\|f_2-f_1\|_{C^{1+\alpha}(\mathbb{T})}\sin^2\left(\tfrac{\varphi-\varphi'}{2}\right).
		\end{align*}
		To get the last estimate, we have used the $1$-Lipschitz property of the function $\sin$ together with \eqref{convex sin} and
		\begin{align*}
			\forall k\in\{1,2\},\quad\left|\sin\left(\tfrac{f_k(\varphi)-f_k(\varphi')}{2}\right)\right|&\leqslant|f_k(\varphi)-f_k(\varphi')|\\
			&\leqslant\|f_k\|_{C^{1+\alpha}(\mathbb{T})}|\varphi-\varphi'|\\
			&\leqslant Cr\left|\sin\left(\tfrac{\varphi-\varphi'}{2}\right)\right|.
		\end{align*}
		Hence
		$$|K_3\{f_1,f_2\}(\varphi,\varphi')|\leqslant C\|f_2-f_1\|_{C^{1+\alpha}(\mathbb{T})}\left|\sin\left(\tfrac{\varphi-\varphi'}{2}\right)\right|^{-1}.$$
		Then, differentiating we find
		$$|\partial_\varphi K_3\{f_1,f_2\}(\varphi,\varphi')|\leqslant C\|f_2-f_1\|_{C^{1+\alpha}(\mathbb{T})}\left|\sin\left(\tfrac{\varphi-\varphi'}{2}\right)\right|^{-2}.$$
		Hence Proposition \ref{prop-potentialtheory-2} implies
		\begin{align*}
			\left\|\big(I_1\{f_1\}-I_2\{f_2\}\big)[h]\right\|_{C^\alpha(\mathbb{T})}\leqslant C\|f_1-f_2\|_{C^{1+\alpha}(\mathbb{T})}\|h\|_{L^\infty(\mathbb{T})},
		\end{align*}
		concluding that $f\mapsto I_1\{f\}$ is continuous.

		\medskip \noindent
		\textbf{(ii)} It follows from

		\begin{equation}\label{dcF}
			\partial_{c}d_f\mathscr{F}(c,f)[h]=\partial_{\varphi}h.
		\end{equation}

		\medskip \noindent
		\textbf{(iii)} We assume now that $f=0.$ First notice that \eqref{omgN at 0}-\eqref{diff omgN} imply $d_f\Psi_{p}^{[2]}(0)=0.$ In addition, one readily has $\Psi_{p}^{[1]}\{0\}=0.$ Hence, we have
		$$d_f\mathscr{F}(c,0)[h](\varphi)=c\,\partial_{\varphi}h(\varphi)+\frac{1}{\sin(\theta_0)}\partial_{\varphi}\Big(\partial_{\theta}\Psi_{\textnormal{\tiny{FC}}}(\theta_0)h(\varphi)+\big(d_{f}\Psi_{p}^{[1]}\{0\}[h]\big)(\theta_0,\varphi)\Big).$$
		From \eqref{dttPsi tt0}, we deduce
		\begin{equation}\label{dphidttPsiFC}
			\frac{1}{\sin(\theta_0)}\partial_{\varphi}\Big(\partial_{\theta}\Psi_{\textnormal{\tiny{FC}}}(\theta_0)h\Big)=\left(\tfrac{\omega_{N}-\omega_{S}}{2}-\widetilde{\gamma}\right)\partial_{\varphi}h=-\left(\tfrac{\omega_{N}-\omega_{S}}{2}-\widetilde{\gamma}\right)\sum_{n=1}^{\infty}\mathbf{m}nh_n\sin(\mathbf{m}n\varphi).
		\end{equation}
		After straightforward simplifications and using Lemma \ref{lem int}, we get
		\begin{align*}
			\big(d_{f}\Psi_{p}^{[1]}\{0\}[h]\big)(\theta_0,\varphi)&=\frac{\omega_{N}-\omega_{S}}{4\pi}\sin(\theta_0)\int_{0}^{2\pi}h(\varphi')\log\Big(1-\cos^2(\theta_0)-\sin^2(\theta_0)\cos(\varphi-\varphi')\Big)d\varphi'\\
			&=\frac{\omega_{N}-\omega_{S}}{4\pi}\sin(\theta_0)\sum_{n=1}^{\infty}h_n\int_{0}^{2\pi}\log\Big(1-\cos^2(\theta_0)-\sin^2(\theta_0)\cos(\varphi')\Big)\cos\big(\mathbf{m}n(\varphi-\varphi')\big)d\varphi'\\
			&=\frac{\omega_{N}-\omega_{S}}{2}\sin(\theta_0)\sum_{n=1}^{\infty}h_n I_{\mathbf{m}n}(\theta_0,\theta_0)\cos(\mathbf{m}n\varphi)\\
			&=-\frac{\omega_{N}-\omega_{S}}{2}\sin(\theta_0)\sum_{n=1}^{\infty}\frac{h_n} {\mathbf{m}n}\cos(\mathbf{m}n\varphi)
			.
		\end{align*}
		Therefore,
		\begin{equation}\label{dPsip0}
			\frac{1}{\sin(\theta_0)}\partial_{\varphi}\Big(\big(d_{f}\Psi_{p}^{[1]}\{0\}[h]\big)(\theta_0,\varphi)\Big)=\frac{\omega_{N}-\omega_{S}}{2}\sum_{n=1}^{\infty}h_n\sin(\mathbf{m}n\varphi).
		\end{equation}
		Introducing the classical $2\pi$-periodic Hilbert transform $\mathcal{H}$ defined by
		$$\mathcal{H}h(\varphi)\triangleq\frac{1}{2\pi}\int_{0}^{2\pi}h(\varphi')\cot\left(\tfrac{\varphi-\varphi'}{2}\right)d\varphi',$$
		which acts on the cosine basis as
		$$\forall n\in\mathbb{N}^*,\quad\mathcal{H}\cos(n\varphi)=\sin(n\varphi),$$
		we have
		\begin{equation}\label{dPsip02}
			\frac{1}{\sin(\theta_0)}\partial_{\varphi}\Big(\big(d_{f}\Psi_{p}^{[1]}\{0\}[h]\big)(\theta_0,\varphi)\Big)=\frac{\omega_{N}-\omega_{S}}{2}\mathcal{H}h(\varphi).
		\end{equation}
		Combining \eqref{dphidttPsiFC}, \eqref{dPsip0} and \eqref{dPsip02}, we obtain the Fourier representation \eqref{split dfF2} or equivalently, the following structure for the linearized operator
		\begin{equation}\label{split dfF}
			d_{f}\mathscr{F}(c,0)=\Big(c+\frac{\omega_{N}-\omega_{S}}{2}-\widetilde{\gamma}\Big)\partial_{\varphi}+\frac{\omega_{N}-\omega_{S}}{2}\mathcal{H}.
		\end{equation}
		Clearly, if $c\neq\widetilde{\gamma}-\frac{\omega_{N}-\omega_{S}}{2},$ the operator $\Big(c+\frac{\omega_{N}-\omega_{S}}{2}-\widetilde{\gamma}\Big)\partial_{\varphi}:X_{\mathbf{m}}^{1+\alpha}\rightarrow Y_{\mathbf{m}}^{\alpha}$ is an isomorphism. We shall prove the compactness of the Hilbert transform in the Hölder spaces. For that, we come back to the integral expression which can be rewritten as
		$$\mathcal{H}h(\varphi)=\frac{1}{2\pi}\int_{0}^{2\pi}K(\varphi,\varphi')\partial_{\varphi'}h(\varphi')d\varphi',\qquad K(\varphi,\varphi')\triangleq \log\Big(\Big|\sin\big(\tfrac{\varphi-\varphi'}{2}\big)\Big|\Big).$$
		For any $\delta\in(0,1),$ we have
		$$|K(\varphi,\varphi')|\lesssim\left|\sin\big(\tfrac{\varphi-\varphi'}{2}\big)\right|^{-\delta},\qquad|\partial_{\varphi}K(\varphi,\varphi')|\lesssim\left|\sin\big(\tfrac{\varphi-\varphi'}{2}\big)\right|^{-(1+\delta)}\cdot$$
		Thus, we can apply Lemma \ref{prop-potentialtheory} (with $\delta=1-\beta$) and get
		$$\forall\beta\in(\alpha,1),\quad\|\mathcal{H}h\|_{C^{\beta}(\mathbb{T})}\lesssim\|\partial_{\varphi}h\|_{L^{\infty}(\mathbb{T})}\lesssim\|h\|_{C^{1+\alpha}(\mathbb{T})}.$$
		Since, for $\beta\in(\alpha,1),$ the injection $C^{\beta}(\mathbb{T})\hookrightarrow C^{\alpha}(\mathbb{T})$ is compact, we deduce that the operator $\mathcal{H}:C^{1+\alpha}(\mathbb{T})\rightarrow C^{\alpha}(\mathbb{T})$ is compact. Thus, \eqref{split dfF} together with \cite[Cor. 5.9]{CR21} implies, for $c\neq\widetilde{\gamma}-\frac{\omega_{N}-\omega_{S}}{2},$ the desired Fredholmness property. This proves Proposition \ref{propreg1}.
	\end{proof}
	According to \eqref{split dfF2}, the candidates for bifurcation points define the following singular set
	\begin{equation}\label{def cm}
		\mathcal{S}_{c}\triangleq\Big\{c_{\mathbf{m}}(\widetilde{\gamma})\triangleq\widetilde{\gamma}-(\omega_{N}-\omega_{S})\tfrac{\mathbf{m}-1}{2\mathbf{m}},\quad\mathbf{m}\in\mathbb{N}^*\Big\}.
	\end{equation}
	Finally, in the following proposition, we gather all the remaining conditions required to apply the Crandall-Rabinowitz Theorem. Then, Theorem \ref{thm bif vcap1} follows immediately from this proposition.
	\begin{prop}\label{prop hyp CR1}
		Let $\alpha\in(0,1)$, $\widetilde{\gamma}\in\mathbb{R}$ and $\mathbf{m}\in\mathbb{N}^*.$
		\begin{enumerate}[label=(\roman*)]
			\item The linear operator $d_{f}\mathscr{F}\big(c_{\mathbf{m}}(\widetilde{\gamma}),0\big):X_{\mathbf{m}}^{1+\alpha}\rightarrow Y_{\mathbf{m}}^{\alpha}$ is of Fredholm type with index zero.
			\item The kernel of $d_{f}\mathscr{F}\big(c_{\mathbf{m}}(\widetilde{\gamma}),0\big)$ is one dimensional. More precisely,
			\begin{equation}\label{ker1}
				\ker\Big(d_{f}\mathscr{F}\big(c_{\mathbf{m}}(\widetilde{\gamma}),0\big)\Big)=\mathtt{span}\big(\varphi\mapsto\cos(\mathbf{m}\varphi)\big).
			\end{equation}
			\item The transversality condition is satisfied, namely
			\begin{equation}\label{trans}
				\partial_{c}d_{f}\mathscr{F}\big(c_{\mathbf{m}}(\widetilde{\gamma}),0\big)[\varphi\mapsto\cos(\mathbf{m}\varphi)]\not\in\textnormal{Im}\Big(d_{f}\mathscr{F}\big(c_{\mathbf{m}}(\widetilde{\gamma}),\widetilde{\gamma},0\big)\Big).
			\end{equation}
		\end{enumerate}
	\end{prop}
	\begin{proof}
		\textbf{(i)} By construction \eqref{def cm}, our bifurcation points satisfy
		$$c_{\mathbf{m}}(\widetilde{\gamma})=\widetilde{\gamma}-(\omega_{N}-\omega_{S})\frac{\mathbf{m}-1}{2\mathbf{m}}\neq\widetilde{\gamma}-\frac{\omega_{N}-\omega_{S}}{2}\cdot$$
		Hence, Proposition \ref{propreg1}-(iii) gives the desired Fredholmness property.

		\medskip \noindent
		\textbf{(ii)} The sequence $\big(c_{\mathbf{m}n}(\widetilde{\gamma})\big)_{n\in\mathbb{N}^*}$ being strictly monotone, then \eqref{split dfF2} and \eqref{def cm} give that the kernel is one dimensional and generated by $\varphi\mapsto\cos(\mathbf{m}\varphi).$

		\medskip \noindent
		\textbf{(iii)} To prove the transversality condition, we first need to describe the range. For this aim, we introduce on $Y_{\mathbf{m}}^{\alpha}$ the scalar product
		$$\left(\sum_{n=1}^{\infty}a_n\sin(\mathbf{m}n\varphi)\Big|\sum_{n=1}^{\infty}b_n\sin(\mathbf{m}n\varphi)\right)\triangleq\sum_{n=1}^{\infty}a_nb_n.$$
		Now we claim that
		\begin{equation}\label{range}
			\textnormal{Im}\Big(d_{f}\mathscr{F}\big(c_{\mathbf{m}}(\widetilde{\gamma}),0\big)\Big)=\mathtt{span}^{\perp_{(\cdot|\cdot)}}\big(\varphi\mapsto\sin(\mathbf{m}\varphi)\big).
		\end{equation}
		Indeed, the first inclusion is obvious from \eqref{split dfF2}-\eqref{def cm}. The converse inclusion is obtained because the range is closed and of codimension 1, which results from the Fredholmness property with zero index and the one dimensional kernel condition. Now, it remains to check the transversality condition. In view of \eqref{dcF}, we infer
		\begin{equation}\label{calc dcdf}
			\partial_{c}d_{f}\mathscr{F}\big(c_{\mathbf{m}}(\widetilde{\gamma}),0\big)[\cos(\mathbf{m}\varphi)]=-\mathbf{m}\sin(\mathbf{m}\varphi)\in\mathtt{span}\big(\varphi\mapsto\sin(\mathbf{m}\varphi)\big).
		\end{equation}
		Combining \eqref{range} and \eqref{calc dcdf}, the condition \eqref{trans} follows. This achieves the proof of Proposition \ref{prop hyp CR1}.
	\end{proof}

	\section{The two--interfaces case: vorticity bands}\label{sec-two-interf}
	This section is devoted to the proof of Theorem \ref{thm bif vcap2} dealing with the case of two interfaces ($M=3$). As before, we shall reformulate the problem with a suitable functional and implement bifurcation techniques. The computations are more involved due to the interactions between the boundaries which leads to a vectorial analysis. 
	\subsection{Equations of interest}
	We start again by some remarks on the flat solution.
	\begin{lem}\label{lem flat cap2}
		Let $0<\theta_1<\theta_2<\pi.$ For any $\omega_{N},\omega_{C},\omega_{S}\in\mathbb{R}$ such that
		\begin{equation}\label{constraint omega NCS}
			\omega_{N}+\omega_{S}=(\omega_{N}-\omega_{C})\cos(\theta_1)+(\omega_{C}-\omega_{S})\cos(\theta_2),
		\end{equation}
		the following function describing the flat cap (FC2) $$\overline{\Omega}_{\textnormal{\tiny{FC2}}}(\theta)\triangleq\omega_{N}\mathbf{1}_{0<\theta<\theta_1}+\omega_{C}\mathbf{1}_{\theta_1\leqslant\theta<\theta_2}+\omega_{S}\mathbf{1}_{\theta_2\leqslant\theta<\pi}$$
		is a stationary solution to Euler equations.
		In addition, 
		\begin{equation}\label{dttPsi tt1 tt2}
			\partial_{\theta}\Psi_{\textnormal{\tiny{FC2}}}(\theta_1)=\omega_{N}\tan\big(\tfrac{\theta_1}{2}\big)-\widetilde{\gamma}\sin(\theta_1),\qquad\Psi_{\textnormal{\tiny{FC2}}}(\theta_2)=-\omega_{S}\cot\big(\tfrac{\theta_2}{2}\big)-\widetilde{\gamma}\sin(\theta_2).
		\end{equation}
	\end{lem}
	\begin{proof}
		$\blacktriangleright$ Observe that $$\forall\alpha\in\mathbb{R},\quad\forall\xi\in\mathbb{S}^2,\quad\overline{\Omega}_{\textnormal{\tiny{FC2}}}\big(\mathcal{R}(\alpha)\xi\big)=\overline{\Omega}_{\textnormal{\tiny{FC2}}}(\xi).$$
		Hence, Lemma \ref{lem sat vort} applies and proves that this is a stationary solution.\\
		$\blacktriangleright$ The constraint \eqref{constraint omega NCS} follows again from \eqref{null avrg vort} and \eqref{conv int S2}, namely
		\begin{align*}
			0=\int_{\mathbb{S}^2}\Omega_{\textnormal{\tiny{FC2}}}(t,\xi)d\sigma(\xi)&=\int_{0}^{2\pi}\int_{0}^{\pi}\Omega_{\textnormal{\tiny{FC2}}}(t,\theta,\varphi)\sin(\theta)d\theta d\varphi\\
			&=2\pi\left(\omega_{N}\int_{0}^{\theta_1}\sin(\theta)d\theta+\omega_{C}\int_{\theta_1}^{\theta_2}\sin(\theta)d\theta+\omega_{S}\int_{\theta_2}^{\pi}\sin(\theta)d\theta\right)\\
			&=2\pi\Big[\omega_{N}\big(1-\cos(\theta_1)\big)+\omega_{C}\big(\cos(\theta_1)-\cos(\theta_2)\big)+\omega_{S}\big(1+\cos(\theta_2)\big)\Big].
		\end{align*}
		$\blacktriangleright$ The potential velocity solves the elliptic equation
		$$\Delta\Psi_{\textnormal{\tiny{FC2}}}=\Omega_{\textnormal{\tiny{FC2}}},\qquad\textnormal{i.e.}\qquad\partial_{\theta}\big[\sin(\theta)\partial_{\theta}\Psi_{\textnormal{\tiny{FC2}}}(\theta)\big]=\sin(\theta)\Big(\omega_{N}\mathbf{1}_{0<\theta<\theta_1}+\omega_{C}\mathbf{1}_{\theta_1\leqslant\theta<\theta_2}+\mathbf{1}_{\theta_2\leqslant\theta<\pi}\Big)+\widetilde{\gamma}\sin(2\theta).$$
		Integrating the previous relation and chosing the constant of integration as in Lemma \ref{lem flat cap} gives
		$$\partial_{\theta}\Psi_{\textnormal{\tiny{FC2}}}(\theta)=\begin{cases}
			\frac{\omega_{N}}{\sin(\theta)}\big(1-\cos(\theta)\big)-\widetilde{\gamma}\sin(\theta), & \textnormal{if }\theta\in(0,\theta_1),\\
			\frac{\omega_{N}}{\sin(\theta)}\big(1-\cos(\theta_1)\big)+\frac{\omega_C}{\sin(\theta)}\big(\cos(\theta_1)-\cos(\theta)\big)-\widetilde{\gamma}\sin(\theta), & \textnormal{if }\theta\in[\theta_1,\theta_2),\\
			\frac{\omega_{N}}{\sin(\theta)}\big(1-\cos(\theta_1)\big)+\frac{\omega_C}{\sin(\theta)}\big(\cos(\theta_1)-\cos(\theta_2)\big)+\frac{\omega_{S}}{\sin(\theta)}\big(\cos(\theta_2)-\cos(\theta)\big)-\widetilde{\gamma}\sin(\theta), & \textnormal{if }\theta\in[\theta_2,\pi).
		\end{cases}$$
		Finally, using \eqref{constraint omega NCS}, we can write
		$$\partial_{\theta}\Psi_{\textnormal{\tiny{FC2}}}(\theta)=\begin{cases}
			\frac{\omega_{N}}{\sin(\theta)}\big(1-\cos(\theta)\big)-\widetilde{\gamma}\sin(\theta), & \textnormal{if }\theta\in(0,\theta_1),\\
			\frac{\omega_{C}}{\sin(\theta)}\big(\cos(\theta_2)-\cos(\theta)\big)-\frac{\omega_{S}}{\sin(\theta)}\big(1+\cos(\theta_2)\big)-\widetilde{\gamma}\sin(\theta), & \textnormal{if }\theta\in[\theta_1,\theta_2),\\
			-\frac{\omega_S}{\sin(\theta)}\big(1+\cos(\theta)\big)-\widetilde{\gamma}\sin(\theta), & \textnormal{if }\theta\in[\theta_2,\pi).
		\end{cases}$$
		At $\theta=\theta_1,$ we find
		\begin{align*}
			\partial_{\theta}\Psi_{\textnormal{\tiny{FC2}}}(\theta_1)
			&=\frac{\omega_{N}}{\sin(\theta_1)}\big(1-\cos(\theta_1)\big)-\widetilde{\gamma}\sin(\theta_1)\\
			&=\omega_{N}\tan\big(\tfrac{\theta_1}{2}\big)-\widetilde{\gamma}\sin(\theta_1).
		\end{align*}
		At $\theta=\theta_2,$ we find
		\begin{align*}
			\partial_{\theta}\Psi_{\textnormal{\tiny{FC2}}}(\theta_2)
			&=-\frac{\omega_{S}}{\sin(\theta_2)}\big(1+\cos(\theta_2)\big)-\widetilde{\gamma}\sin(\theta_2)\\
			&=-\omega_{S}\cot\big(\tfrac{\theta_2}{2}\big)-\widetilde{\gamma}\sin(\theta_2).
		\end{align*}
	\end{proof}
	
	\begin{figure}[!h]
		\begin{center}
			\tdplotsetmaincoords{80}{100}
			\begin{tikzpicture}[tdplot_main_coords,scale=3]
				\draw[thick,->] (0,0,0) -- (3,0,0) node[anchor=north east]{$x$};
				\draw[thick,->] (0,0,0) -- (0,1.2,0) node[anchor=north west]{$y$};
				\draw[thick,->] (0,0,0) -- (0,0,1.2) node[anchor=south]{$z$};
				
				\def\h{1.1}
				\coordinate (O) at (0,0,0);
				\coordinate (z) at (0,0,\h);
				\coordinate (x) at (\h,0,0);
				\coordinate (y) at (0,\h,0);
				\coordinate (a) at (0.6,0.6,0.53);
				\coordinate (b) at (0.6,0.6,0);
				
				\shade[tdplot_screen_coords, ball color = gray!40, opacity=0.4] (0,0) circle (1);
				\draw[gray!100,dashed] (0,0,0.6) ellipse(0.785cm and 0.15cm);
				\draw[gray!100,dashed] (0,0,0)--(0.55,0.45,0.55);
				\draw[gray!100,dashed] (0,0,-0.35) ellipse(0.94cm and 0.15cm);
				\draw[gray!100,dashed] (0,0,0)--(1,1,-0.2);
				
				\begin{pgfonlayer}{foreground}
					\draw[->] (z)+(-40:0.25) arc (-40:170:0.25);
				\end{pgfonlayer}
				\node at (0,0.3,1.15) {$\widetilde{\gamma}$};
				\node at (0,0.2,0.9) {$\omega_N(f_1)$};
				\node at (1,0.4,-0.5) {$\omega_S$};
				\node at (1,-0.3,0.3) {$\omega_C$};
				\node at (0,0.1,0.3) {$\textcolor{gray!100}{\theta_1}$};
				\node at (0.1,0.24,0.1) {$\textcolor{gray!100}{\theta_2}$};
				
				
				
				\draw[->,>=stealth,gray!100] (0,0,0.2) arc [start angle=0,end angle=45,x radius=-0.3,y radius=0.22];
				
				\draw[->,>=stealth,gray!100] (0,0,0.1) arc [start angle=0,end angle=45,x radius=-3.5,y radius=0.5];
				
				\tikzmath{function absci(\t) {return sin(54+5*cos(6*\t r))*cos(\t r);};}
				\tikzmath{function ord(\t) {return sin(54+5*cos(6*\t r))*sin(\t r);};}
				\tikzmath{function cote(\t) {return cos(54+5*cos(6*\t r));};}
				
				\draw[red,thick] plot[domain=0:6.3,smooth,variable=\t,samples=500] ({absci(\t)},{ord(\t)},{cote(\t)});
				
				\tikzmath{function absci2(\t) {return sin(110+5*cos(6*\t r))*cos(\t r);};}
				\tikzmath{function ord2(\t) {return sin(110+5*cos(6*\t r))*sin(\t r);};}
				\tikzmath{function cote2(\t) {return cos(110+5*cos(6*\t r));};}
				
				\draw[blue,thick] plot[domain=0:6.3,smooth,variable=\t,samples=500] ({absci2(\t)},{ord2(\t)},{cote2(\t)});
			\end{tikzpicture}
			\caption{Representation of two interfaces (in red and blue) vortex cap solutions with $6$-fold symmetry.}
		\end{center}
	\end{figure}
	
	From now on, we fix 
	\begin{equation}\label{choice th1-th2}
		0<\theta_1<\theta_2<\pi,
	\end{equation}
	$\omega_N,\omega_C,\omega_S\in\mathbb{R}$ satisfying \eqref{constraint omega NCS} and consider a vortex cap solution close to $\overline{\Omega}_{\textnormal{\tiny{FC2}}}$ in the form
	$$\overline{\Omega}(t,\theta,\varphi)=\omega_N(f_1)\mathbf{1}_{0<\theta<\theta_1+f_1(t,\varphi)}+\omega_C\mathbf{1}_{\theta_1+f_1(t,\varphi)\leqslant\theta<\theta_2+f_2(t,\varphi)}+\omega_{S}\mathbf{1}_{\theta_2+f_2(t,\varphi)\leqslant\theta<\pi},$$
	with
	$$\forall k\in\{1,2\},\quad|f_k(t,\varphi)|\ll1.$$
	The quantity $\omega_N(f_1)$ is defined implicitely so that the Gauss constraint is satisfied, namely
	\begin{align*}
		0&=\int_{\mathbb{S}^2}\overline{\Omega}(t,\xi)d\sigma(\xi)\\
		&=\int_{0}^{2\pi}\left(\omega_N(f_1)\int_{0}^{\theta_1+f_1(t,\varphi)}\sin(\theta)d\theta+\omega_C\int_{\theta_2+f_2(t,\varphi)}^{\theta_1+f_1(t,\varphi)}\sin(\theta)d\theta+\omega_S\int_{\theta_2+f_2(t,\varphi)}^{\pi}\sin(\theta)d\theta\right)d\varphi\\
		&=\int_{0}^{2\pi}\Big[\omega_N(f_1)\big(1-\cos(\theta_1+f_1(t,\varphi)\big)+\omega_C\Big(\cos\big(\theta_2+f_2(t,\varphi)\big)-\cos\big(\theta_1+f_1(t,\varphi)\big)\Big)\\
		&\quad\qquad\qquad\qquad\qquad\qquad\qquad\qquad\qquad+\omega_S\big(1+\cos(\theta_2+f_2(t,\varphi)\big)\Big]d\varphi.
	\end{align*}
	As in Section \ref{sec eq1}, by means of the implicit function theorem, we can rigorously justify the definition of $\omega_N(f_1)$ and obtain
	\begin{equation}\label{omgN0 and domgN0 2}
		\omega_N(0)=\omega_N,\qquad d_{f_1}\omega_N(0)=0.
	\end{equation}
	For $k\in\{1,2\}$, the interface oscillating around $\theta=\theta_k$ can be parametrized by
	$$z_{k}(t,\varphi)\triangleq\begin{pmatrix}
		\sin\big(\theta_k+f_{k}(t,\varphi)\big)\cos(\varphi)\\
		\sin\big(\theta_k+f_{k}(t,\varphi)\big)\sin(\varphi)\\
		\cos\big(\theta_k+f_{k}(t,\varphi)\big)
	\end{pmatrix}.$$
	In view of \eqref{vorticity cap equations}, the parametrizations $z_1$ and $z_2$ must satisfy the following equations
	$$\forall k\in\{1,2\},\quad\partial_{t}z_{k}(t,\varphi)\cdot\big(J\partial_{\varphi}z_{k}(t,\varphi)\big)=\partial_{\varphi}\Big(\Psi\big(t,z_{k}(t,\varphi)\big)\Big).$$
	Proceeding as in Section \ref{sec eq1}, we obtain
	$$\partial_{t}z_{k}(t,\varphi)\cdot\big(J\partial_{\varphi}z_{k}(t,\varphi)\big)=\sin\big(\theta_k+f_{k}(t,\varphi)\big)\partial_{t}f_{k}(t,\varphi).$$
	Consequently, the unknowns $f_1$ and $f_2$ have to solve the following (coupled) system
	$$\forall k\in\{1,2\},\quad\partial_{t}f_{k}(t,\varphi)=\frac{\partial_{\varphi}\Big(\Psi\big(t,z_{k}(t,\theta)\big)\Big)}{\sin\big(\theta_0+f_{k}(t,\varphi)\big)}\cdot$$
	Now, the stream function writes
	\begin{align*}
		\Psi\big(t,z_{k}(t,\varphi)\big)&=\frac{\omega_{N}(f_1)}{4\pi}\int_{0}^{2\pi}\int_{0}^{\theta_1+f_{1}(t,\varphi')}\log\Big(D\big(\theta_k+f_k(t,\varphi),\theta',\varphi,\varphi'\big)\Big)\sin(\theta')d\theta'd\varphi'\\
		&\quad+\frac{\omega_{C}}{4\pi}\int_{0}^{2\pi}\int_{\theta_1+f_{1}(t,\varphi')}^{\theta_2+f_{2}(t,\varphi')}\log\Big(D\big(\theta_k+f_k(t,\varphi),\theta',\varphi,\varphi'\big)\Big)\sin(\theta')d\theta'd\varphi'\\
		&\quad+\frac{\omega_{S}}{4\pi}\int_{0}^{2\pi}\int_{\theta_2+f_2(t,\varphi')}^{\pi}\log\Big(D\big(\theta_k+f_k(t,\varphi),\theta',\varphi,\varphi'\big)\Big)\sin(\theta')d\theta'd\varphi'\\
		&\quad+\frac{\widetilde{\gamma}}{4\pi}\int_{0}^{2\pi}\int_{0}^{\pi}\log\Big(D\big(\theta_k+f_k(t,\varphi),\theta',\varphi,\varphi'\big)\Big)\sin(2\theta')d\theta'd\varphi'.
	\end{align*}
	Remark that the unperturbed stream function can be written as follows
	\begin{align*}
		\Psi_{\textnormal{\tiny{FC2}}}(\theta)&=\frac{\omega_{N}}{4\pi}\int_{0}^{2\pi}\int_{0}^{\theta_1}\log\Big(D(\theta,\theta',0,\varphi')\Big)\sin(\theta')d\theta'd\varphi'\\
		&\quad+\frac{\omega_{C}}{4\pi}\int_{\theta_1}^{\theta_2}\int_{\theta_0}^{\pi}\log\Big(D(\theta,\theta',0,\varphi')\Big)\sin(\theta')d\theta'd\varphi'\\
		&\quad+\frac{\omega_{S}}{4\pi}\int_{\theta_2}^{\pi}\int_{\theta_0}^{\pi}\log\Big(D(\theta,\theta',0,\varphi')\Big)\sin(\theta')d\theta'd\varphi'\\
		&\quad+\frac{\widetilde{\gamma}}{4\pi}\int_{0}^{2\pi}\int_{0}^{\pi}\log\Big(D(\theta,\theta',0,\varphi')\Big)\sin(2\theta')d\theta' d\varphi'.
	\end{align*}
	Making appeal to Chasles' relation, we can write
	\begin{align*}
		\Psi\big(t,z_{k}(t,\varphi)\big)&=\Psi_{\textnormal{\tiny{FC2}}}\big(\theta_k+f_k(t,\varphi)\big)+\Psi_{p,2}^{[1]}\{f_1,f_2\}\big(\theta_k+f_k(t,\varphi),\varphi\big)+\Psi_{p,2}^{[2]}\{f_1\}\big(\theta_k+f_k(t,\varphi),\varphi\big)\\
		&\triangleq\Psi\{f_1,f_2\}\big(\theta_k+f_k(t,\varphi),\varphi\big),\\
		\Psi_{p,2}^{[1]}\{f_1,f_2\}(\theta,\varphi)&\triangleq\frac{\omega_{N}-\omega_{C}}{4\pi}\int_{0}^{2\pi}\int_{\theta_1}^{\theta_1+f_{1}(t,\varphi')}\log\Big(D(\theta,\theta',\varphi,\varphi')\Big)\sin(\theta')d\theta'd\varphi'\\
		&\quad+\frac{\omega_{C}-\omega_{S}}{4\pi}\int_{0}^{2\pi}\int_{\theta_2}^{\theta_2+f_{2}(t,\varphi')}\log\Big(D(\theta,\theta',\varphi,\varphi')\Big)\sin(\theta')d\theta'd\varphi',\\
		\Psi_{p,2}^{[2]}\{f_1\}\big(\theta,\varphi\big)&\triangleq\frac{\omega_{N}(f_1)-\omega_N}{4\pi}\int_{0}^{2\pi}\int_{0}^{\theta_1+f_1(t,\varphi')}\log\big(D(\theta,\theta',\varphi,\varphi')\big)\sin(\theta')d\theta'd\varphi'.
	\end{align*}
	Therefore, the vortex cap equation \eqref{vorticity cap equations} becomes
	\begin{equation}\label{vorticity cap eq general ansatz2}
		\forall k\in\{1,2\},\quad\partial_{t}f_k(t,\varphi)=\frac{\partial_{\varphi}\Big(\Psi\{f_1,f_2\}\big(\theta_k+f_k(t,\varphi),\varphi\big)\Big)}{\sin\big(\theta_k+f_k(t,\varphi)\big)}\cdot
	\end{equation}
	We look for traveling solutions at speed $c\in\mathbb{R}$
	$$\forall k\in\{1,2\},\qquad f_{k}(t,\varphi)=f_{k}(\varphi-ct).$$
	Thus, we shall solve
	$$\mathscr{G}(c,f_1,f_2)=0,\qquad\mathscr{G}\triangleq(\mathscr{G}_1,\mathscr{G}_2),$$
	where
	$$\mathscr{G}_{k}(c,f_1,f_2)(\varphi)\triangleq c\,\partial_{\varphi}f_{k}(\varphi)+\frac{\partial_{\varphi}\Big(\Psi\{f_1,f_2\}\big(\theta_k+f_k(\varphi),\varphi\big)\Big)}{\sin\big(\theta_k+f_k(\varphi)\big)}\cdot$$
	Observe that
	$$\forall c\in\mathbb{R},\quad\mathscr{G}(c,0,0)=0.$$
	This leads again to implement bifurcation theory.
	\subsection{Spectral properties}
	We check here the hypothesis of Crandall-Rabinowitz Theorem.
	\begin{prop}\label{propreg2}
		Let $\alpha\in(0,1)$ and $\mathbf{m}\in\mathbb{N}^*.$ There exists $r>0$ such that
		\begin{enumerate}[label=(\roman*)]
			\item The function $\mathscr{G}:\mathbb{R}\times B_{r,\mathbf{m}}^{1+\alpha}\times B_{r,\mathbf{m}}^{1+\alpha}\rightarrow Y_{\mathbf{m}}^{\alpha}\times Y_{\mathbf{m}}^{\alpha}$ is well-defined and of class $C^1.$
			\item The partial derivative $\partial_{c}d_{(f_1,f_2)}\mathscr{G}:\mathbb{R}\times B_{r,\mathbf{m}}^{1+\alpha}\times B_{r,\mathbf{m}}^{1+\alpha}\rightarrow\mathcal{L}(X_{\mathbf{m}}^{1+\alpha}\times X_{\mathbf{m}}^{1+\alpha},Y_{\mathbf{m}}^{\alpha}\times Y_{\mathbf{m}}^{\alpha})$ exists and is continuous.
			\item At the equilibrium $(f_1,f_2)=(0,0),$ the linearized operator admits the following Fourier representation 
			\begin{align}
				&d_{(f_1,f_2)}\mathscr{G}(c,0,0)\left[\sum_{n=1}^{\infty}h_{n}^{(1)}\cos(\mathbf{m}n\varphi),\sum_{n=1}^{\infty}h_{n}^{(2)}\cos(\mathbf{m}n\varphi)\right]\nonumber\\
				&=\sum_{n=1}^{\infty}\mathbf{m}nM_{\mathbf{m}n}(c,\theta_1,\theta_2)\begin{pmatrix}
					h_{n}^{(1)}\\
					h_{n}^{(2)}
				\end{pmatrix}\sin(\mathbf{m}n\varphi),\label{split dfG2}
			\end{align}
			with 
			\begin{equation}\label{def matrix}
				M_{n}(c,\theta_1,\theta_2)\triangleq\begin{pmatrix}
					-c+\frac{\omega_{N}-\omega_{C}}{2n}-\frac{\omega_{N}}{2\cos^2\left(\frac{\theta_1}{2}\right)}+\widetilde{\gamma} & \frac{\omega_C-\omega_S}{2n}\frac{\sin(\theta_2)}{\sin(\theta_1)}\tan^n\left(\frac{\theta_1}{2}\right)\cot^n\left(\frac{\theta_2}{2}\right)\vspace{0.2cm}\\
					\frac{\omega_N-\omega_C}{2n}\frac{\sin(\theta_1)}{\sin(\theta_2)}\tan^n\left(\frac{\theta_1}{2}\right)\cot^n\left(\frac{\theta_2}{2}\right)& -c+\frac{\omega_{C}-\omega_{S}}{2n}+\frac{\omega_{S}}{2\sin^{2}\left(\frac{\theta_2}{2}\right)}+\widetilde{\gamma}
				\end{pmatrix}.
			\end{equation}
			In addition, if $c\not\in\left\lbrace\widetilde{\gamma}-\tfrac{\omega_N}{2\cos^2\left(\frac{\theta_1}{2}\right)},\widetilde{\gamma}+\tfrac{\omega_S}{2\sin^2\left(\frac{\theta_2}{2}\right)}\right\rbrace,$ then the operator $d_{(f_1,f_2)}\mathscr{G}(c,0,0):X_{\mathbf{m}}^{1+\alpha}\times X_{\mathbf{m}}^{1+\alpha}\rightarrow Y_{\mathbf{m}}^{\alpha}\times Y_{\mathbf{m}}^{\alpha}$ is of Fredholm type with index zero.
		\end{enumerate}
	\end{prop}
	\begin{proof}
		\textbf{(i)} The proof is very close to Proposition \ref{propreg1}-(i). Indeed, the functional involves terms corresponding to the self-interaction of each boundary with itself, which correspond to the one-interface analysis. It also involves new terms corresponding to the interaction between both boundaries, but which are non-singular (smooth kernels). Therefore, we omit the proof and just give the expression of the linearized operator. For $k\in\{1,2\},$
		\begin{align*}
			&d_{f_k}\mathscr{G}_{k}(c,f_1,f_2)[h_k](\varphi)\\
			&=c\,\partial_{\varphi}h_k(\varphi)+h_k(\varphi)\frac{\cos\big(\theta_k+f_k(\varphi)\big)}{\sin^{2}\big(\theta_k+f_k(\varphi)\big)}\partial_{\varphi}\Big(\Psi_{\textnormal{\tiny{FC}}}\big(\theta_k+f_k(\varphi)\big)+\big(\Psi_{p,2}^{[1]}\{f_1,f_2\}+\Psi_{p,2}^{[2]}\{f_1\}\big)\big(\theta_k+f_k(\varphi),\varphi\big)\Big)\\
			&\quad+\frac{1}{\sin\big(\theta_k+f_k(\varphi)\big)}\partial_{\varphi}\Big(h_k(\varphi)\Big[\partial_{\theta}\Psi_{\textnormal{\tiny{FC}}}\big(\theta_k+f_k(\varphi)\big)+\big(\partial_{\theta}\Psi_{p,2}^{[1]}\{f_1,f_2\}+\partial_{\theta}\Psi_{p,2}^{[2]}\{f_1\}\big)\big(\theta_k+f_k(\varphi),\varphi\big)\Big]\Big)\\
			&\quad+\frac{1}{\sin\big(\theta_k+f_k(\varphi)\big)}\partial_{\varphi}\Big(\big(d_{f_k}\Psi_{p,2}^{[1]}\{f_1,f_2\}[h_k]\big)\big(\theta_k+f_k(\varphi),\varphi\big)+\delta_{k,1}d_{f_k}\big(\Psi_{p,2}^{[2]}\{f_1\}[h_k]\big)\big(\theta_k+f_k(\varphi),\varphi\big)\Big),
		\end{align*}
		where $\delta_{k,1}$ is the kronecker symbol and
		\begin{align*}
			d_{f_{3-k}}\mathscr{G}_{k}(c,f_1,f_2)[h_{3-k}](\varphi)&=\frac{1}{\sin\big(\theta_k+f_k(\varphi)\big)}\partial_{\varphi}\Big(\big(d_{f_{3-k}}\Psi_{p,2}^{[1]}\{f_1,f_2\}[h_{3-k}]\big)\big(\theta_k+f_k(\varphi),\varphi\big)\Big)\\
			&\quad+\frac{\delta_{3-k,1}}{\sin\big(\theta_k+f_k(\varphi)\big)}\partial_{\varphi}\Big(\big(d_{f_{3-k}}\Psi_{p,2}^{[2]}\{f_1\}[h]\big)\big(\theta_k+f_k(\varphi),\varphi\big)\Big).
		\end{align*}
		If we denote $(\omega_{N},\omega_{C},\omega_{S})=(\omega_{1},\omega_2,\omega_3)$, then for $k\in\{1,2\},$ we have,
		\begin{align*}
			\big(d_{f_k}\Psi_{p,2}^{[1]}\{f_1,f_2\}[h_k]\big)(\theta,\varphi)=\frac{\omega_{k}-\omega_{k+1}}{4\pi}\int_{0}^{2\pi}h_k(\varphi')\log\Big(D\big(\theta,\theta_k+f_k(\varphi'),\varphi,\varphi'\big)\Big)\sin\big(\theta_k+f_k(\varphi')\big)d\varphi'.	
		\end{align*}
		Moreover,
		\begin{align*}
			\big(d_{f_1}\Psi_{p,2}^{[2]}\{f_1\}[h_1]\big)(\theta,\varphi)&=\frac{d_{f_1}\omega_N(f_1)[h_1]}{4\pi}\int_{0}^{2\pi}\int_{0}^{\theta_1+f_1(\varphi')}\log\big(D(\theta,\theta',\varphi,\varphi')\big)\sin(\theta')d\theta'd\varphi'\\
			&\quad+\frac{\omega_N(f_1)-\omega_N}{4\pi}\int_{0}^{2\pi}h_{1}(\varphi')\log\Big(D\big(\theta,\theta_1+f_1(\varphi'),\varphi,\varphi'\big)\Big)\sin\big(\theta_1+f_1(\varphi')\big)d\varphi'.
		\end{align*}
		\textbf{(ii)} Immediate since 
		\begin{equation}\label{dcF2}
			\partial_{c}f_{(f_1,f_2)}\mathscr{G}(c,f_1,f_2)[h_1,h_2]=\begin{pmatrix}
				\partial_{\varphi}h_1 & 0\\
				0 & \partial_{\varphi}h_2
			\end{pmatrix}.
		\end{equation}
		\textbf{(iii)} We assume now that $(f_1,f_2)=(0,0).$ Using \eqref{omgN0 and domgN0 2} and $\Psi_{p,2}^{[1]}\{0,0\}=0,$ we have
		\begin{align*}
			d_{f_k}\mathscr{G}_k(c,0,0)[h_k](\varphi)&=c\,\partial_{\varphi}h_k(\varphi)+\frac{1}{\sin(\theta_k)}\partial_{\varphi}\Big(\partial_{\theta}\Psi_{\textnormal{\tiny{FC}}}(\theta_k)h_k(\varphi)+\big(d_{f_k}\Psi_{p,2}^{[1]}\{0,0\}[h_k]\big)(\theta_k,\varphi)\Big),\\
			d_{f_{3-k}}\mathscr{G}_{k}(c,0,0)[h_{3-k}](\varphi)&=\frac{1}{\sin(\theta_{k})}\partial_{\varphi}\Big(\big(d_{f_{3-k}}\Psi_{p,2}^{[1]}\{0,0\}[h_{3-k}]\big)(\theta_k,\varphi)\Big).
		\end{align*}
		From \eqref{dttPsi tt1 tt2}, we deduce
		\begin{align}
			\frac{1}{\sin(\theta_1)}\partial_{\varphi}\Big(\partial_{\theta}\Psi_{\textnormal{\tiny{FC}}}(\theta_1)h_1\Big)&=\left(\frac{\omega_{N}\tan\left(\frac{\theta_1}{2}\right)}{\sin(\theta_1)}-\widetilde{\gamma}\right)\partial_{\varphi}h_1\nonumber\\
			&=\left(\frac{\omega_{N}}{2\cos^2\left(\frac{\theta_1}{2}\right)}-\widetilde{\gamma}\right)\partial_{\varphi}h_1\nonumber\\
			&=\left(\widetilde{\gamma}-\frac{\omega_{N}}{2\cos^2\left(\frac{\theta_1}{2}\right)}\right)\sum_{n=1}^{\infty}\mathbf{m}nh_n^{(1)}\sin(\mathbf{m}n\varphi)\label{dphidttPsiFC2-1}
		\end{align}
		and
		\begin{align}
			\frac{1}{\sin(\theta_2)}\partial_{\varphi}\Big(\partial_{\theta}\Psi_{\textnormal{\tiny{FC}}}(\theta_2)h_2\Big)&=-\left(\frac{\omega_{S}\cot\left(\frac{\theta_2}{2}\right)}{\sin(\theta_2)}+\widetilde{\gamma}\right)\partial_{\varphi}h_2\nonumber\\
			&=-\left(\frac{\omega_{S}}{2\sin^2\left(\frac{\theta_2}{2}\right)}+\widetilde{\gamma}\right)\partial_{\varphi}h_2\nonumber\\
			&=\left(\frac{\omega_{S}}{2\sin^2\left(\frac{\theta_2}{2}\right)}+\widetilde{\gamma}\right)\sum_{n=1}^{\infty}\mathbf{m}nh_n^{(2)}\sin(\mathbf{m}n\varphi).\label{dphidttPsiFC2-2}
		\end{align}
		After straightforward simplifications and using Lemma \ref{lem int}, we get for $k,\ell\in\{1,2\},$
		\begin{align*}
			\big(d_{f_\ell}\Psi_{p,2}^{[1]}\{0,0\}[h_\ell]\big)(\theta_k,\varphi)&=\frac{\omega_{\ell}-\omega_{\ell+1}}{4\pi}\sin(\theta_\ell)\int_{0}^{2\pi}h_\ell(\varphi')\log\Big(D(\theta_k,\theta_\ell,\varphi,\varphi')\Big)d\varphi'\\
			&=\frac{\omega_{\ell}-\omega_{\ell+1}}{2}\sin(\theta_\ell)\sum_{n=1}^{\infty}h_n^{(\ell)}I_{\mathbf{m}n}(\theta_k,\theta_\ell)\cos(\mathbf{m}n\varphi)\\
			&=-\frac{\omega_{\ell}-\omega_{\ell+1}}{2}\sin(\theta_\ell)\sum_{n=1}^{\infty}\frac{h_n^{(\ell)}}{\mathbf{m}n}\tan^{\mathbf{m}n}\left(\tfrac{\min(\theta_k,\theta_\ell)}{2}\right)\cot^{\mathbf{m}n}\left(\tfrac{\max(\theta_k,\theta_\ell)}{2}\right)\cos(\mathbf{m}n\varphi).
		\end{align*}
		Therefore,
		$$\frac{\partial_{\varphi}\Big(\big(d_{f_\ell}\Psi_{p,2}^{[1]}\{0,0\}[h_\ell]\big)(\theta_k,\varphi)\Big)}{\sin(\theta_k)}=\frac{\omega_{\ell}-\omega_{\ell+1}}{2}\frac{\sin(\theta_\ell)}{\sin(\theta_k)}\sum_{n=1}^{\infty}h_n^{(\ell)}\tan^{\mathbf{m}n}\left(\tfrac{\min(\theta_k,\theta_\ell)}{2}\right)\cot^{\mathbf{m}n}\left(\tfrac{\max(\theta_k,\theta_\ell)}{2}\right)\sin(\mathbf{m}n\varphi).$$
		Putting together the foregoing calculations, we get \eqref{split dfG2}-\eqref{def matrix}. Now, denoting 
		$$\mathcal{Q}(\varphi)\triangleq\log\big(1-\cos(\theta_1)\cos(\theta_2)-\sin(\theta_1)\sin(\theta_2)\cos(\varphi)\big),$$
		we have
		\begin{align}
			&d_{(f_1,f_2)}\mathscr{G}(c,0,0)=I+K,\nonumber\\
			&I\triangleq\begin{pmatrix}
				\left(c+\frac{\omega_{N}}{2\cos^2\left(\frac{\theta_1}{2}\right)}-\widetilde{\gamma}\right)\partial_{\varphi} & 0\\
				0& \left(c-\frac{\omega_{S}}{2\sin^2\left(\frac{\theta_2}{2}\right)}-\widetilde{\gamma}\right)\partial_{\varphi}
			\end{pmatrix},\nonumber\\
			&K\triangleq\begin{pmatrix}
				\frac{\omega_{N}-\omega_C}{2}\mathcal{H} & \frac{\omega_C-\omega_S}{2}\frac{\sin(\theta_2)}{\sin(\theta_1)}\partial_{\varphi}\mathcal{Q}\ast\cdot\\
				\frac{\omega_N-\omega_C}{2}\frac{\sin(\theta_1)}{\sin(\theta_2)}\partial_{\varphi}\mathcal{Q}\ast\cdot & \frac{\omega_{C}-\omega_S}{2}\mathcal{H}
			\end{pmatrix}.\label{split dfG}\end{align}
		If $c\not\in\left\lbrace\widetilde{\gamma}-\tfrac{\omega_N}{2\cos^2\left(\frac{\theta_1}{2}\right)},\widetilde{\gamma}+\tfrac{\omega_S}{2\sin^2\left(\frac{\theta_2}{2}\right)}\right\rbrace,$ then $I:X_{\mathbf{m}}^{1+\alpha}\times X_{\mathbf{m}}^{1+\alpha}\rightarrow Y_{\mathbf{m}}^{\alpha}\times Y_{\mathbf{m}}^{\alpha}$ is an isomorphism. We have already studied the compact property of the Hilbert transform in the proof of Proposition \ref{prop hyp CR1}, so we are left with the anti-diagonal terms. Actually, the corresponding symbol decays exponentially fast in $n$, which implies that $\partial_{\varphi}\mathcal{Q}\ast\cdot$ is smoothing at every order (so a fortiori compact in the considered functional framework). Thus, the operator $K:X_{\mathbf{m}}^{1+\alpha}\times X_{\mathbf{m}}^{1+\alpha}\rightarrow Y_{\mathbf{m}}^{\alpha}\times Y_{\mathbf{m}}^{\alpha}$ is compact. We deduce the desired Fredholmness property.
	\end{proof}
	We shall now study the spectrum.
	\begin{lem}\label{lem:spectrum}
		Let $\widetilde{\gamma}\in\mathbb{R}.$ There exists $N(\theta_1,\theta_2)\triangleq N(\theta_1,\theta_2,\omega_N,\omega_S,\omega_C)\in\mathbb{N}^*$ such that for any $n\in\mathbb{N}^*$ with $n\geqslant N(\theta_1,\theta_2),$ there exist two velocities
		\begin{align}\label{def cm2}
			&c_{n}^{\pm}(\widetilde{\gamma},\theta_1,\theta_2)\triangleq\widetilde{\gamma}+\frac{\omega_{S}}{4\sin^{2}\left(\frac{\theta_2}{2}\right)}-\frac{\omega_{N}}{4\cos^2\left(\frac{\theta_1}{2}\right)}+\frac{\omega_{N}-\omega_{S}}{4n}\\
			&\quad\pm\frac{1}{4}\sqrt{\left(\frac{\omega_{S}}{\sin^{2}\left(\frac{\theta_2}{2}\right)}+\frac{\omega_{N}}{\cos^{2}\left(\frac{\theta_1}{2}\right)}-\frac{\omega_{N}+\omega_S-2\omega_C}{n}\right)^2+\frac{1}{n^2}(\omega_{N}-\omega_C)(\omega_C-\omega_S)\tan^{2n}\left(\tfrac{\theta_1}{2}\right)\cot^{2n}\left(\tfrac{\theta_2}{2}\right)}\nonumber
		\end{align}
		for which the matrix $M_{n}\big(c_{n}^{\pm}(\widetilde{\gamma},\theta_1,\theta_2),\theta_1,\theta_2\big)$ is singular. The sequences $\big(c_{n}^{\pm}(\widetilde{\gamma},\theta_1,\theta_2)\big)_{n\geqslant N(\theta_1,\theta_2)}$ are strictly monotone and
		\begin{equation}\label{def limit set}
			\mathbb{L}\triangleq\left\lbrace\lim_{n\to\infty}c_{n}^{+}(\widetilde{\gamma},\theta_1,\theta_2)\,,\,\lim_{n\to\infty}c_{n}^{-}(\widetilde{\gamma},\theta_1,\theta_2)\right\rbrace=\left\lbrace\widetilde{\gamma}-\frac{\omega_{N}}{2\cos^2\left(\frac{\theta_1}{2}\right)}\,,\,\widetilde{\gamma}+\frac{\omega_{S}}{2\sin^2\left(\frac{\theta_2}{2}\right)}\right\rbrace.
		\end{equation}
		Moreover,
		\begin{enumerate}
			\item If
			\begin{equation}\label{additional constraint}
				\omega_{S}\cos^2\left(\tfrac{\theta_1}{2}\right)+\omega_{N}\sin^{2}\left(\tfrac{\theta_2}{2}\right)\neq 0.
			\end{equation}
			then $|\mathbb{L}|=2$ and the following equations have no solution
			\begin{equation}\label{no spec coll}
				c_{p}^{+}(\widetilde{\gamma},\theta_1,\theta_2)=c_{q}^{-}(\widetilde{\gamma},\theta_1,\theta_2),\qquad p,q\geqslant N(\theta_1,\theta_2).
			\end{equation}
			\item If \begin{equation}\label{additional constraint2}
				\omega_{S}\cos^2\left(\tfrac{\theta_1}{2}\right)+\omega_{N}\sin^{2}\left(\tfrac{\theta_2}{2}\right)=0.
			\end{equation}
			then
			$$|\mathbb{L}|=1,\qquad\omega_N+\omega_S=\omega_C\neq 0,\qquad\omega_N\omega_S<0.$$
			In particular, \eqref{def cm2} simplifies into
			\begin{equation}\label{def cn+- alter}
				c_{n}^{\pm}(\widetilde{\gamma},\theta_1,\theta_2)=\widetilde{\gamma}+\frac{\omega_N-\omega_S}{4n}-\frac{\omega_N}{2\cos^2\left(\frac{\theta_1}{2}\right)}\pm\frac{1}{4n}\sqrt{\omega_C^2-\omega_N\omega_S\tan^{2n}\left(\tfrac{\theta_1}{2}\right)\cot^{2n}\left(\tfrac{\theta_2}{2}\right)}.
			\end{equation}
			In addition, for any $\mathbf{m}\in\mathbb{N}$ with $\mathbf{m}\geqslant N(\theta_1,\theta_2),$ under one of the additional constraints
			\begin{itemize}
				\item [$\mathbf{(H1+)}$] $\omega_C>0,\qquad\omega_N>0,\qquad\omega_S<0,$
				\item [$\mathbf{(H2+)}$] $\omega_C>0,\qquad\omega_N<0,\qquad\omega_S>0\qquad\textnormal{and}\qquad2\cos^2\left(\tfrac{\theta_1}{2}\right)>\sin^2\left(\tfrac{\theta_2}{2}\right),$
				\item [$\mathbf{(H3+)}$] $\omega_C<0,\qquad\omega_N>0,\qquad\omega_S<0,$
				\item [$\mathbf{(H4+)}$] $\omega_C<0,\qquad\omega_N<0,\qquad\omega_S>0\qquad\textnormal{and}\qquad2\sin^2\left(\tfrac{\theta_2}{2}\right)>\cos^2\left(\tfrac{\theta_1}{2}\right),$
			\end{itemize}
			the following equations have no solution
			$$c_{\mathbf{m}}^{+}(\widetilde{\gamma},\theta_1,\theta_2)=c_{k\mathbf{m}}^{-}(\widetilde{\gamma},\theta_1,\theta_2),\qquad
			k\in\mathbb{N}^*.$$
			And under one of the additional constraints
			\begin{itemize}
				\item [$\mathbf{(H1-)}$] $\omega_C>0,\qquad\omega_N<0,\qquad\omega_S>0,$
				\item [$\mathbf{(H2-)}$] $\omega_C>0,\qquad\omega_N>0,\qquad\omega_S<0\qquad\textnormal{and}\qquad2\sin^2\left(\tfrac{\theta_2}{2}\right)>\cos^2\left(\tfrac{\theta_1}{2}\right),$
				\item [$\mathbf{(H3-)}$] $\omega_C<0,\qquad\omega_N<0,\qquad\omega_S>0,$
				\item [$\mathbf{(H4-)}$] $\omega_C<0,\qquad\omega_N>0,\qquad\omega_S<0\qquad\textnormal{and}\qquad2\cos^2\left(\tfrac{\theta_1}{2}\right)>\sin^2\left(\tfrac{\theta_2}{2}\right),$
			\end{itemize}
			the following equations have no solution
			$$c_{k\mathbf{m}}^{+}(\widetilde{\gamma},\theta_1,\theta_2)=c_{\mathbf{m}}^{-}(\widetilde{\gamma},\theta_1,\theta_2),\qquad
			k\in\mathbb{N}^*.$$
			This means that there is no $\mathbf{m}$-fold spectral collision, i.e. the $\mathbf{m}$-fold spectrum is simple.
		\end{enumerate}
	\end{lem}
	
	\begin{figure}[!h]
		\null\hfill
		\subfigure[$2\sin^2\left(\tfrac{\theta_2}{2}\right)>\cos^2\left(\tfrac{\theta_1}{2}\right).$\label{fig1a}]{%
			\includegraphics[width=8cm]{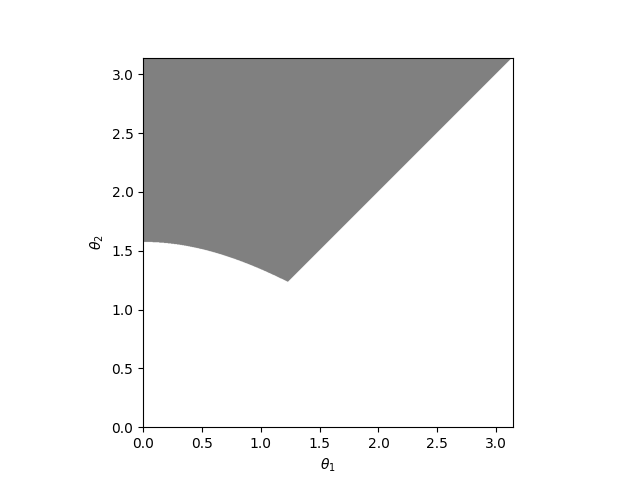}}
		\hfill
		\subfigure[$2\cos^2\left(\tfrac{\theta_1}{2}\right)>\sin^2\left(\tfrac{\theta_2}{2}\right).$\label{fig1b}]{%
			\includegraphics[width=8cm]{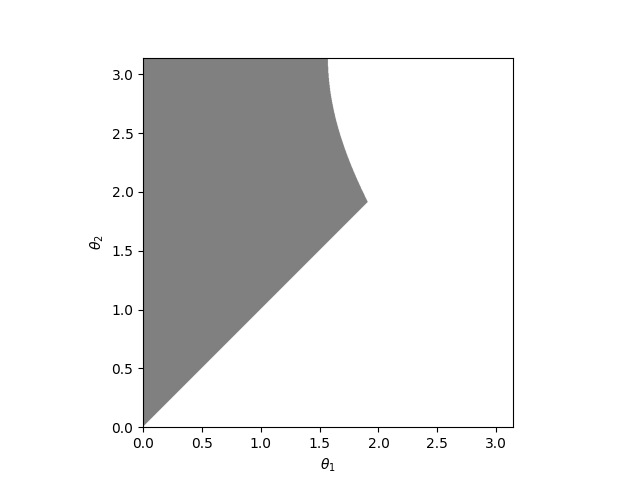}}
		\hfill\null
		\caption{The shaded regions represent the admissible couples $(\theta_1,\theta_2)\in(0,\pi)^2$ with $\theta_1<\theta_2.$}\label{fig1}
	\end{figure}

	\begin{proof}
		From \eqref{def matrix}, we have that the determinant of $M_{n}(c,\theta_1,\theta_2)$ is
		\begin{align}
			&\det\big(M_{n}(c,\theta_1,\theta_2)\big)\triangleq c^2-\beta_{n}(\widetilde{\gamma},\theta_1,\theta_2)c+\gamma_{n}(\widetilde{\gamma},\theta_1,\theta_2)\in\mathbb{R}_2[c],\nonumber\\
			&\beta_{n}(\widetilde{\gamma},\theta_1,\theta_2)\triangleq\frac{\omega_{N}-\omega_{S}}{2n}-\frac{\omega_{N}}{2\cos^2\left(\frac{\theta_1}{2}\right)}+\frac{\omega_{S}}{2\sin^{2}\left(\frac{\theta_2}{2}\right)}+2\widetilde{\gamma},\label{btn}\\
			&\gamma_{n}(\widetilde{\gamma},\theta_1,\theta_2)\triangleq \left(\frac{\omega_{C}-\omega_{S}}{2n}+\frac{\omega_{S}}{2\sin^{2}\left(\frac{\theta_2}{2}\right)}+\widetilde{\gamma}\right)\left(\frac{\omega_{N}-\omega_{C}}{2n}-\frac{\omega_{N}}{2\cos^2\left(\frac{\theta_1}{2}\right)}+\widetilde{\gamma}\right)\nonumber\\
			&\qquad\qquad\qquad-\frac{1}{4n^2}(\omega_{N}-\omega_C)(\omega_C-\omega_S)\tan^{2n}\left(\tfrac{\theta_1}{2}\right)\cot^{2n}\left(\tfrac{\theta_2}{2}\right).\label{gamn}
		\end{align}
		The discriminant of the previous polynomial is
		\begin{align*}
			\Delta_{n}(\theta_1,\theta_2)&\triangleq \beta_{n}^2(\widetilde{\gamma},\theta_1,\theta_2)-4\gamma_{n}(\widetilde{\gamma},\theta_1,\theta_2)\\
			&=\frac{1}{4}\left[\left(\frac{\omega_{S}}{\sin^{2}\left(\frac{\theta_2}{2}\right)}+\frac{\omega_{N}}{\cos^{2}\left(\frac{\theta_1}{2}\right)}-\frac{\omega_{N}+\omega_S-2\omega_C}{n}\right)^2+\frac{1}{n^2}(\omega_{N}-\omega_C)(\omega_C-\omega_S)\tan^{2n}\left(\tfrac{\theta_1}{2}\right)\cot^{2n}\left(\tfrac{\theta_2}{2}\right)\right].
		\end{align*}
		Notice that $\Delta_{n}(\theta_1,\theta_2)$ is independent of $\widetilde{\gamma}.$ 
		We shall now prove that
		\begin{equation}\label{posi discri}
			\exists\, N(\theta_1,\theta_2)\in\mathbb{N}^*,\quad\forall n\in\mathbb{N},\quad n\geqslant N(\theta_1,\theta_2)\quad\Rightarrow\quad\Delta_{n}(\theta_1,\theta_2)>0.
		\end{equation}
		Assuming that \eqref{posi discri} is true, then we conclude that, for $n$ large enough, we have two distinct real roots
		\begin{align*}
			c_{n}^{\pm}(\widetilde{\gamma},\theta_1,\theta_2)&\triangleq\frac{1}{2}\beta_{n}(\widetilde{\gamma},\theta_1,\theta_2)\pm\frac{1}{2}\sqrt{\Delta_{n}(\theta_1,\theta_2)}\\
			&=\widetilde{\gamma}+\frac{\omega_{S}}{4\sin^{2}\left(\frac{\theta_2}{2}\right)}-\frac{\omega_{N}}{4\cos^2\left(\frac{\theta_1}{2}\right)}+\frac{\omega_{N}-\omega_{S}}{4n}\\
			&\pm\frac{1}{4}\sqrt{\left(\frac{\omega_{S}}{\sin^{2}\left(\frac{\theta_2}{2}\right)}+\frac{\omega_{N}}{\cos^{2}\left(\frac{\theta_1}{2}\right)}-\frac{\omega_{N}+\omega_S-2\omega_C}{n}\right)^2+\frac{1}{n^2}(\omega_{N}-\omega_C)(\omega_C-\omega_S)\tan^{2n}\left(\tfrac{\theta_1}{2}\right)\cot^{2n}\left(\tfrac{\theta_2}{2}\right)}\nonumber.
		\end{align*}
		\textbf{1.} First assume that \eqref{additional constraint} holds. 
		From the proof of Lemma \ref{lem int}, we know that $\tan\left(\tfrac{\theta_1}{2}\right)\cot\left(\tfrac{\theta_2}{2}\right)<1$, then
		\begin{equation}\label{decay geo}
			\forall k\in\mathbb{N},\quad\frac{1}{n^2}(\omega_{N}-\omega_C)(\omega_C-\omega_S)\tan^{2n}\left(\tfrac{\theta_1}{2}\right)\cot^{2n}\left(\tfrac{\theta_2}{2}\right)\underset{n\to\infty}{=}O_{\theta_1,\theta_2}\left(\frac{1}{n^{k}}\right).
		\end{equation}
		Then,
		\begin{equation}\label{lim discri}
			\Delta_{\infty}(\theta_1,\theta_2)=\frac{1}{4}\left(\frac{\omega_{S}}{\sin^{2}\left(\frac{\theta_2}{2}\right)}+\frac{\omega_{N}}{\cos^2\left(\frac{\theta_1}{2}\right)}\right)^2>0,
		\end{equation}
		and \eqref{posi discri} is true. Factorizing, we can write for any $n$ sufficiently large
		$$c_{n}^{\pm}(\widetilde{\gamma},\theta_1,\theta_2)=\widetilde{\gamma}+\frac{\omega_{S}}{4\sin^{2}\left(\frac{\theta_2}{2}\right)}-\frac{\omega_{N}}{4\cos^2\left(\frac{\theta_1}{2}\right)}+\frac{\omega_{N}-\omega_{S}}{4n}\pm\frac{1}{4}\left|\frac{\omega_{S}}{\sin^{2}\left(\frac{\theta_2}{2}\right)}+\frac{\omega_{N}}{\cos^{2}\left(\frac{\theta_1}{2}\right)}-\frac{\omega_{N}+\omega_S-2\omega_C}{n}\right|\pm\mathtt{r}_{n}(\theta_1,\theta_2),$$
		with
		$$\mathtt{r}_{n}(\theta_1,\theta_2)\triangleq\frac{1}{4}\left|\tfrac{\omega_{S}}{\sin^{2}\left(\frac{\theta_2}{2}\right)}+\tfrac{\omega_{N}}{\cos^{2}\left(\frac{\theta_1}{2}\right)}-\tfrac{\omega_{N}+\omega_S-2\omega_C}{n}\right|\left[\sqrt{1+\frac{(\omega_{N}-\omega_C)(\omega_C-\omega_S)\tan^{2n}\left(\tfrac{\theta_1}{2}\right)\cot^{2n}\left(\tfrac{\theta_2}{2}\right)}{\left(\left[\frac{\omega_{S}}{\sin^{2}\left(\frac{\theta_2}{2}\right)}+\frac{\omega_{N}}{\cos^{2}\left(\frac{\theta_1}{2}\right)}\right]n-(\omega_{N}+\omega_S-2\omega_C)\right)^2}}-1\right].$$
		Notice that \eqref{decay geo} implies
		\begin{equation}\label{decay ttrn}
			\forall k\in\mathbb{N},\quad \mathtt{r}_{n}(\theta_1,\theta_2)\underset{n\to\infty}{=}O_{\theta_1,\theta_2}\left(\frac{1}{n^k}\right).
		\end{equation}
		We have the following dichotomy.
		\begin{enumerate}[label=\textbullet]
			\item If $\omega_{S}\cos^2\left(\tfrac{\theta_1}{2}\right)+\omega_{N}\sin^{2}\left(\tfrac{\theta_2}{2}\right)>0,$ then for $n$ large enough we have
			$$\left|\frac{\omega_{S}}{\sin^{2}\left(\frac{\theta_2}{2}\right)}+\frac{\omega_{N}}{\cos^{2}\left(\frac{\theta_1}{2}\right)}-\frac{\omega_{N}+\omega_S-2\omega_C}{n}\right|=\frac{\omega_{S}}{\sin^{2}\left(\frac{\theta_2}{2}\right)}+\frac{\omega_{N}}{\cos^{2}\left(\frac{\theta_1}{2}\right)}-\frac{\omega_{N}+\omega_S-2\omega_C}{n},$$
			and therefore
			$$\begin{cases}
				c_{n}^{+}(\theta_1,\theta_2)=\widetilde{\gamma}+\frac{\omega_{S}}{2\sin^2\left(\frac{\theta_2}{2}\right)}+\frac{\omega_{C}-\omega_S}{2n}+\mathtt{r}_{n}(\theta_1,\theta_2),\\
				c_{n}^{-}(\theta_1,\theta_2)=\widetilde{\gamma}-\frac{\omega_{N}}{2\cos^2\left(\frac{\theta_1}{2}\right)}+\frac{\omega_{N}-\omega_C}{2n}-\mathtt{r}_{n}(\theta_1,\theta_2).
			\end{cases}$$
			As a consequence,
			$$\begin{cases}
				c_{n+1}^{+}(\widetilde{\gamma},\theta_1,\theta_2)-c_{n}^{+}(\widetilde{\gamma},\theta_1,\theta_2)\underset{n\to\infty}{=}\frac{\omega_S-\omega_C}{2n(n+1)}+O_{\theta_1,\theta_2}\left(\frac{1}{n^3}\right),\\
				c_{n+1}^{-}(\widetilde{\gamma},\theta_1,\theta_2)-c_{n}^{-}(\widetilde{\gamma},\theta_1,\theta_2)\underset{n\to\infty}{=}\frac{\omega_C-\omega_N}{2n(n+1)}+O_{\theta_1,\theta_2}\left(\frac{1}{n^3}\right).
			\end{cases}$$
			Since $\omega_N\neq\omega_C$ and $\omega_C\neq\omega_S$, then we conclude the asymptotic (strict) monotonicity of $n\mapsto c_{n}^{\pm}(\widetilde{\gamma},\theta_1,\theta_2).$
			In addition,
			$$\lim_{n\to\infty}c_{n}^{+}(\widetilde{\gamma},\theta_1,\theta_2)=\widetilde{\gamma}+\frac{\omega_{S}}{2\sin^2\left(\frac{\theta_2}{2}\right)},\qquad\textnormal{and}\qquad\lim_{n\to\infty}c_{n}^{-}(\widetilde{\gamma},\theta_1,\theta_2)=\widetilde{\gamma}-\frac{\omega_{N}}{2\cos^2\left(\frac{\theta_1}{2}\right)}\cdot$$
			\item If $\omega_{S}\cos^2\left(\tfrac{\theta_1}{2}\right)+\omega_{N}\sin^{2}\left(\tfrac{\theta_2}{2}\right)<0,$ then for $n$ large enough we have
			$$\left|\frac{\omega_{S}}{\sin^{2}\left(\frac{\theta_2}{2}\right)}+\frac{\omega_{N}}{\cos^{2}\left(\frac{\theta_1}{2}\right)}-\frac{\omega_{N}+\omega_S-2\omega_C}{n}\right|=\frac{\omega_{N}+\omega_S-2\omega_C}{n}-\frac{\omega_{S}}{\sin^{2}\left(\frac{\theta_2}{2}\right)}-\frac{\omega_{N}}{\cos^{2}\left(\frac{\theta_1}{2}\right)}$$
			and therefore
			$$\begin{cases}
				c_{n}^{+}(\theta_1,\theta_2)=\widetilde{\gamma}-\frac{\omega_{N}}{2\cos^2\left(\frac{\theta_1}{2}\right)}+\frac{\omega_{N}-\omega_C}{2n}-\mathtt{r}_{n}(\theta_1,\theta_2),\\
				c_{n}^{-}(\theta_1,\theta_2)=\widetilde{\gamma}+\frac{\omega_{S}}{2\sin^2\left(\frac{\theta_2}{2}\right)}+\frac{\omega_{C}-\omega_S}{2n}+\mathtt{r}_{n}(\theta_1,\theta_2).
			\end{cases}$$
			As a consequence,
			$$\begin{cases}
				c_{n+1}^{+}(\widetilde{\gamma},\theta_1,\theta_2)-c_{n}^{+}(\widetilde{\gamma},\theta_1,\theta_2)\underset{n\to\infty}{=}\frac{\omega_C-\omega_N}{2n(n+1)}+O_{\theta_1,\theta_2}\left(\frac{1}{n^3}\right),\\
				c_{n+1}^{-}(\widetilde{\gamma},\theta_1,\theta_2)-c_{n}^{-}(\widetilde{\gamma},\theta_1,\theta_2)\underset{n\to\infty}{=}\frac{\omega_S-\omega_C}{2n(n+1)}+O_{\theta_1,\theta_2}\left(\frac{1}{n^3}\right).
			\end{cases}$$
			The monotonicity conclusion is still valid. In addition, in this case, the limits are exchanged
			$$\lim_{n\to\infty}c_{n}^{+}(\widetilde{\gamma},\theta_1,\theta_2)=\widetilde{\gamma}-\frac{\omega_{N}}{2\cos^2\left(\frac{\theta_1}{2}\right)}\qquad\textnormal{and}\qquad\lim_{n\to\infty}c_{n}^{-}(\widetilde{\gamma},\theta_1,\theta_2)=\widetilde{\gamma}+\frac{\omega_{S}}{2\sin^2\left(\frac{\theta_2}{2}\right)}\cdot$$
		\end{enumerate}
		The condition \eqref{additional constraint} implies that the above limits are well-separated. Together with the strict monotonicity property, we conclude \eqref{no spec coll}.

		\medskip \noindent
		\textbf{2.} Conversely, we assume that \eqref{additional constraint2} holds.
		This condition can also be written
		\begin{equation}\label{omgS func omgN}
			\omega_S=-\omega_N\frac{1-\cos(\theta_2)}{1+\cos(\theta_1)}\cdot
		\end{equation}
		Notice that the Gauss constraint \eqref{constraint omega NCS} can be written as follows
		\begin{align}
			\omega_C&=\frac{\omega_N\big(1-\cos(\theta_1)\big)+\omega_S\big(1+\cos(\theta_2)\big)}{\cos(\theta_2)-\cos(\theta_1)}\cdot\label{exp omgc 2}
		\end{align}
		Plugging \eqref{omgS func omgN} into \eqref{exp omgc 2} gives
		\begin{align}
			\omega_C&=\omega_N\frac{\big(1-\cos(\theta_1)\big)\big(1+\cos(\theta_1)\big)-\big(1-\cos(\theta_2)\big)\big(1+\cos(\theta_2)\big)}{\big(\cos(\theta_2)-\cos(\theta_1)\big)\big(1+\cos(\theta_1)\big)}\nonumber\\
			&=\omega_N\frac{\cos^2(\theta_2)-\cos^2(\theta_1)}{\big(\cos(\theta_2)-\cos(\theta_1)\big)\big(1+\cos(\theta_1)\big)}\nonumber\\
			&=\omega_N\frac{\cos(\theta_1)+\cos(\theta_2)}{1+\cos(\theta_1)}\cdot\label{omgC func omgN}
		\end{align}
		From \eqref{omgS func omgN} and \eqref{omgC func omgN}, we get
		\begin{align*}
			\omega_N+\omega_S=\frac{\omega_N}{1+\cos(\theta_1)}\Big[1+\cos(\theta_1)-\big(1-\cos(\theta_2)\big)\Big]=\omega_N\frac{\cos(\theta_1)+\cos(\theta_2)}{1+\cos(\theta_1)}=\omega_C.
		\end{align*}
		This last expression implies that $\omega_N\neq0$ (resp. $\omega_S\neq0$) otherwise $\omega_S=\omega_C$ (resp. $\omega_N=\omega_C$) which is excluded by construction. We also deduce
		$$\omega_N-\omega_C=-\omega_S,\qquad\omega_C-\omega_S=\omega_N.$$
		Now, using \eqref{omgS func omgN}, we infer
		\begin{align}\label{sign omgNS}
			\omega_N\omega_S&=-\omega_N^2\frac{1-\cos(\theta_2)}{1+\cos(\theta_1)}<0.
		\end{align}
		In particular $\omega_N\neq\omega_S$ and have opposite sign. Now, assume for the sake of contradiction that $\omega_C=0$, i.e. $\omega_N=-\omega_S.$ Combined with \eqref{additional constraint2} and the fact that $\tfrac{\theta_1}{2},\tfrac{\theta_2}{2}\in(0,\tfrac{\pi}{2}),$ we deduce
		$$\cos^2\left(\tfrac{\theta_1}{2}\right)=\sin^{2}\left(\tfrac{\theta_2}{2}\right),\qquad\textnormal{i.e.}\qquad\theta_1=\theta_2.$$
		This enters in contradiction with \eqref{choice th1-th2}. Thus,
		\begin{equation}\label{link omg NSC}
			\omega_C\neq0\qquad\textnormal{and}\qquad\omega_N+\omega_S-2\omega_C=-\omega_C\neq 0.
		\end{equation}
		In this case, the discriminant becomes
		$$\forall n\in\mathbb{N}^*,\quad\Delta_{n}(\theta_1,\theta_2)=\frac{1}{4n^2}\left[\omega_C^2-\omega_N\omega_S\tan^{2n}\left(\tfrac{\theta_1}{2}\right)\cot^{2n}\left(\tfrac{\theta_2}{2}\right)\right]>0.$$
		This implies in particular \eqref{posi discri}. Factorizing, we can write
		\begin{align*}
			c_{n}^{\pm}(\widetilde{\gamma},\theta_1,\theta_2)&=\widetilde{\gamma}-\frac{\omega_{N}}{2\cos^2\left(\frac{\theta_1}{2}\right)}+\frac{\omega_{N}-\omega_{S}}{4n}\pm\frac{|\omega_C|}{4n}\pm\mathtt{r}_{n}(\theta_1,\theta_2)\\
			&=\widetilde{\gamma}+\frac{\omega_{S}}{2\sin^2\left(\frac{\theta_2}{2}\right)}+\frac{\omega_{N}-\omega_{S}}{4n}\pm\frac{|\omega_C|}{4n}\pm\mathtt{r}_{n}(\theta_1,\theta_2),
		\end{align*}
		with
		$$\mathtt{r}_{n}(\theta_1,\theta_2)\triangleq\frac{|\omega_C|}{4n}\left[\sqrt{1-\frac{\omega_{N}\omega_S}{\omega_C^2}\tan^{2n}\left(\tfrac{\theta_1}{2}\right)\cot^{2n}\left(\tfrac{\theta_2}{2}\right)}-1\right].$$
		We have the following dichotomy.
		\begin{enumerate}[label=\textbullet]
			\item If $\omega_C>0,$ then $|\omega_C|=\omega_C=\omega_N+\omega_S$
			and therefore
			\begin{equation}\label{expand cn pm omcCpos}
				\begin{cases}
					c_{n}^{+}(\theta_1,\theta_2)=\widetilde{\gamma}-\frac{\omega_{N}}{2\cos^2\left(\frac{\theta_1}{2}\right)}+\frac{\omega_{N}}{2n}+\mathtt{r}_{n}(\theta_1,\theta_2)=\widetilde{\gamma}+\frac{\omega_{S}}{2\sin^2\left(\frac{\theta_2}{2}\right)}+\frac{\omega_{N}}{2n}+\mathtt{r}_{n}(\theta_1,\theta_2),\\
					c_{n}^{-}(\theta_1,\theta_2)=\widetilde{\gamma}-\frac{\omega_{N}}{2\cos^2\left(\frac{\theta_1}{2}\right)}-\frac{\omega_S}{2n}-\mathtt{r}_{n}(\theta_1,\theta_2)=\widetilde{\gamma}+\frac{\omega_{S}}{2\sin^2\left(\frac{\theta_2}{2}\right)}-\frac{\omega_S}{2n}-\mathtt{r}_{n}(\theta_1,\theta_2).
				\end{cases}
			\end{equation}
			As a consequence,
			$$\begin{cases}
				c_{n+1}^{+}(\widetilde{\gamma},\theta_1,\theta_2)-c_{n}^{+}(\widetilde{\gamma},\theta_1,\theta_2)\underset{n\to\infty}{=}-\frac{\omega_N}{2n(n+1)}+O_{\theta_1,\theta_2}\left(\frac{1}{n^3}\right),\\
				c_{n+1}^{-}(\widetilde{\gamma},\theta_1,\theta_2)-c_{n}^{-}(\widetilde{\gamma},\theta_1,\theta_2)\underset{n\to\infty}{=}\frac{\omega_S}{2n(n+1)}+O_{\theta_1,\theta_2}\left(\frac{1}{n^3}\right).
			\end{cases}$$
			This is sufficient to conclude the asymptotic strict monotonicity. In addition, since $\omega_N\omega_S<0,$ then both sequences have the same monotonicity asymptotically. Nevertheless, in this case, both part of the spectrum accumulate at the same point
			$$\lim_{n\to\infty}c_{n}^{+}(\widetilde{\gamma},\theta_1,\theta_2)=\widetilde{\gamma}-\frac{\omega_{N}}{2\cos^2\left(\frac{\theta_1}{2}\right)}=\widetilde{\gamma}+\frac{\omega_{S}}{2\sin^2\left(\frac{\theta_2}{2}\right)}=\lim_{n\to\infty}c_{n}^{-}(\widetilde{\gamma},\theta_1,\theta_2).$$
			Therefore, one must avoid the spectral collisions by a more careful analysis.\\
			$\blacktriangleright$ Let us first study the equation
			$$c_{\mathbf{m}}^{+}(\widetilde{\gamma},\theta_1,\theta_2)=c_{k\mathbf{m}}^{-}(\widetilde{\gamma},\theta_1,\theta_2),\qquad k\in\mathbb{N}^*.$$
			According to \eqref{expand cn pm omcCpos}, this equation is equivalent to
			\begin{equation}\label{eq11}
				-k\big(\omega_N+\widetilde{\mathtt{r}}_{\mathbf{m}}\big)=\omega_S+\widetilde{\mathtt{r}}_{k\mathbf{m}},\qquad\widetilde{\mathtt{r}}_{n}\triangleq2n\mathtt{r}_{n}(\theta_1,\theta_2)>0.
			\end{equation}
			Observe that $n\mapsto\widetilde{\mathtt{r}}_{n}$ is asymptotically decreasing and satisfies \eqref{decay ttrn}.\\
			\ding{226} If $\omega_N>0$, then the equation \eqref{eq11} can be written
			$$0=\omega_C+2(k-1)\omega_N+k\widetilde{\mathtt{r}}_{\mathbf{m}}+\widetilde{\mathtt{r}}_{k\mathbf{m}}.$$
			Each term in the right hand-side is non-negative and $\omega_{C}>0$, then this equation has no solution.\\
			\ding{226} Assume now that $\omega_N<0.$ By virtue of \eqref{sign omgNS}, we have $\omega_S>0.$ According to \eqref{decay ttrn} and \eqref{eq11}, we can select $\mathbf{m}$ large enough to ensure
			$$\widetilde{\mathtt{r}}_{\mathbf{m}}<|\omega_N|.$$
			Added to the asymptotic decay property of $n\mapsto\widetilde{\mathtt{r}}_{n}$ we get
			$$\forall k\in\mathbb{N}\setminus\{0,1\},\quad -k\big(\omega_N+\widetilde{\mathtt{r}}_{\mathbf{m}}\big)\geqslant 2\big(|\omega_{N}|-\widetilde{\mathtt{r}}_{\mathbf{m}}\big)\qquad\textnormal{and}\qquad\omega_S+\widetilde{\mathtt{r}}_{2\mathbf{m}}\geqslant\omega_S+\widetilde{\mathtt{r}}_{k\mathbf{m}}.$$
			Hence, it sufficies to impose
			$$2\big(|\omega_{N}|-\widetilde{\mathtt{r}}_{\mathbf{m}}\big)>\omega_S+\widetilde{\mathtt{r}}_{2\mathbf{m}},\qquad\textnormal{i.e.}\qquad2|\omega_N|>\omega_S+2\widetilde{\mathtt{r}}_{\mathbf{m}}+\widetilde{\mathtt{r}}_{2\mathbf{m}},$$
			so that the equations \eqref{eq11} admit no solution for any $k\in\mathbb{N}^*$ (recall that $c_{\mathbf{m}}^{+}\neq c_{\mathbf{m}}^{-}$). Using \eqref{decay ttrn}, we deduce that, up to taking $\mathbf{m}$ large enough, the following condition is sufficient
			\begin{equation}\label{cond omgSN11}
				2|\omega_N|>\omega_S.
			\end{equation}
			But, according to \eqref{additional constraint2}, the constraint \eqref{cond omgSN11} is equivalent to
			$$2\cos^2\left(\tfrac{\theta_1}{2}\right)>\sin^2\left(\tfrac{\theta_2}{2}\right).$$
			$\blacktriangleright$ Now, we turn to the study of the equation
			$$c_{k\mathbf{m}}^{+}(\widetilde{\gamma},\theta_1,\theta_2)=c_{\mathbf{m}}^{-}(\widetilde{\gamma},\theta_1,\theta_2),\qquad k\in\mathbb{N}^*.$$
			Using again \eqref{expand cn pm omcCpos}, this equation is equivalent to
			$$\omega_N+\widetilde{\mathtt{r}}_{k\mathbf{m}}=-k\big(\omega_S+\widetilde{\mathtt{r}}_{\mathbf{m}}\big).$$
			This is basically the same equation as \eqref{eq11} where $\omega_S$ and $\omega_N$ have been exchanged. So either $\omega_S>0$ and there is no solution, or $\omega_S<0$ and there is no solution provided that $\mathbf{m}$ is large enough and 
			$$2|\omega_S|>\omega_N,\qquad\textnormal{i.e.}\qquad2\sin^2\left(\tfrac{\theta_2}{2}\right)>\cos^2\left(\tfrac{\theta_1}{2}\right).$$
			\item If $\omega_C<0,$ then $|\omega_C|=-\omega_C=-\omega_N-\omega_S$ and therefore
			\begin{equation}\label{expand cn pm omcCneg}
				\begin{cases}
					c_{n}^{+}(\theta_1,\theta_2)=\widetilde{\gamma}-\frac{\omega_{N}}{2\cos^2\left(\frac{\theta_1}{2}\right)}-\frac{\omega_S}{2n}+\mathtt{r}_{n}(\theta_1,\theta_2)=\widetilde{\gamma}+\frac{\omega_{S}}{2\sin^2\left(\frac{\theta_2}{2}\right)}-\frac{\omega_S}{2n}+\mathtt{r}_{n}(\theta_1,\theta_2),\\
					c_{n}^{-}(\theta_1,\theta_2)=\widetilde{\gamma}-\frac{\omega_{N}}{2\cos^2\left(\frac{\theta_1}{2}\right)}+\frac{\omega_{N}}{2n}-\mathtt{r}_{n}(\theta_1,\theta_2)=\widetilde{\gamma}+\frac{\omega_{N}}{2\sin^2\left(\frac{\theta_2}{2}\right)}+\frac{\omega_N}{2n}-\mathtt{r}_{n}(\theta_1,\theta_2).
				\end{cases}
			\end{equation}
			As a consequence,
			$$\begin{cases}
				c_{n+1}^{+}(\widetilde{\gamma},\theta_1,\theta_2)-c_{n}^{+}(\widetilde{\gamma},\theta_1,\theta_2)\underset{n\to\infty}{=}\frac{\omega_S}{2n(n+1)}+O_{\theta_1,\theta_2}\left(\frac{1}{n^3}\right),\\
				c_{n+1}^{-}(\widetilde{\gamma},\theta_1,\theta_2)-c_{n}^{-}(\widetilde{\gamma},\theta_1,\theta_2)\underset{n\to\infty}{=}-\frac{\omega_N}{2n(n+1)}+O_{\theta_1,\theta_2}\left(\frac{1}{n^3}\right).
			\end{cases}$$
			As before, we can conclude the asymptotic strict monotonicity with the same limit.\\
			$\blacktriangleright$
			According to \eqref{expand cn pm omcCneg}, the equation
			$$c_{\mathbf{m}}^{+}(\widetilde{\gamma},\theta_1,\theta_2)=c_{k\mathbf{m}}^{-}(\widetilde{\gamma},\theta_1,\theta_2),\qquad k\in\mathbb{N}^*$$
			is equivalent to
			$$-k\big(-\omega_S+\widetilde{\mathtt{r}}_{\mathbf{m}}\big)=-\omega_N+\widetilde{\mathtt{r}}_{k\mathbf{m}}$$
			which is \eqref{eq11} with $(\omega_N,\omega_S)$ replaced by $(-\omega_S,-\omega_N).$ So there is no solution for either $\omega_S<0$ or $\omega_S>0$ and $2\sin^{2}\left(\tfrac{\theta_2}{2}\right)>\cos^2\left(\tfrac{\theta_1}{2}\right).$\\
			$\blacktriangleright$
			The equation
			$$c_{k\mathbf{m}}^{+}(\widetilde{\gamma},\theta_1,\theta_2)=c_{\mathbf{m}}^{-}(\widetilde{\gamma},\theta_1,\theta_2),\qquad k\in\mathbb{N}^*$$
			is equivalent to
			$$-\omega_S+\widetilde{\mathtt{r}}_{k\mathbf{m}}=-k\big(-\omega_N+\widetilde{\mathtt{r}}_{\mathbf{m}}\big)$$
			which is \eqref{eq11} with $(\omega_N,\omega_S)$ replaced by $(-\omega_N,-\omega_S).$ So there is no solution for either $\omega_N<0$ or $\omega_N>0$ and $2\cos^2\left(\tfrac{\theta_1}{2}\right)>\sin^{2}\left(\tfrac{\theta_2}{2}\right).$
		\end{enumerate}
	\end{proof}
	\subsection{Proof of the main result}
	In the following proposition, we gather all the remaining conditions required to apply the Crandall-Rabinowitz Theorem. 
	\begin{prop}\label{prop hyp CR2}
		Let $\alpha\in(0,1)$, $\kappa\in\{+,-\}$ and $\mathbf{m}\in\mathbb{N}^*$ with $\mathbf{m}\geqslant N(\theta_1,\theta_2).$ Assume that \eqref{additional constraint} holds or assume that \eqref{additional constraint2} with $\mathbf{(Hk\kappa)}$ for some $k\in\llbracket 1,4\rrbracket$ holds.
		\begin{enumerate}[label=(\roman*)]
			\item The linear operator $d_{(f_1,f_2)}\mathscr{G}\big(c_{\mathbf{m}}^{\kappa}(\widetilde{\gamma},\theta_1,\theta_2),0,0\big)$ is of Fredholm type with index zero.
			\item The kernel of $d_{(f_1,f_2)}\mathscr{G}\big(c_{\mathbf{m}}^{\kappa}(\widetilde{\gamma},\theta_1,\theta_2),0,0\big)$ is one dimensional. More precisely,
			\begin{equation}\label{ker2}
				\ker\Big(d_{(f_1,f_2)}\mathscr{G}\big(c_{\mathbf{m}}^{\kappa}(\widetilde{\gamma},\theta_1,\theta_2),0,0\big)\Big)=\mathtt{span}\big(u_0\big), 
			\end{equation}
			with
			$$u_0:\varphi\mapsto\begin{pmatrix}
				-c_{\mathbf{m}}^{\kappa}(\widetilde{\gamma},\theta_1,\theta_2)+\frac{\omega_C-\omega_S}{2\mathbf{m}}+\frac{\omega_S}{2\sin^{2}\left(\frac{\theta_2}{2}\right)}+\widetilde{\gamma}\vspace{0.2cm}\\
				\frac{\omega_C-\omega_N}{2\mathbf{m}}\frac{\sin(\theta_1)}{\sin(\theta_2)}\tan^{\mathbf{m}}\left(\frac{\theta_1}{2}\right)\cot^{\mathbf{m}}\left(\frac{\theta_2}{2}\right)
			\end{pmatrix}\cos(\mathbf{m}\varphi).$$
			\item The transversality condition is satisfied, namely
			\begin{equation}\label{trans2}
				\partial_{c}d_{(f_1,f_2)}\mathscr{G}\big(c_{\mathbf{m}}^{\kappa}(\widetilde{\gamma},\theta_1,\theta_2),0,0\big)[u_0]\not\in\textnormal{Im}\Big(d_{(f_1,f_2)}\mathscr{G}\big(c_{\mathbf{m}}^{\kappa}(\widetilde{\gamma},\theta_1,\theta_2),0,0\big)\Big).
			\end{equation}
		\end{enumerate}
	\end{prop}
	\begin{proof}
		\textbf{(i)} The strict monotonicity of $n\mapsto c_{n}^{\kappa}(\widetilde{\gamma},\theta_1,\theta_2)$ gives $c_{\mathbf{m}}^{\kappa}(\widetilde{\gamma},\theta_1,\theta_2)\not\in\mathbb{L}$, where $\mathbb{L}$ is defined in \eqref{def limit set}. Together with Proposition \ref{propreg2}-(iii) this implies the desired Fredholmness property.
		
		\medskip \noindent
		\textbf{(ii)} The non-degeneracy conditions imply that
		$$\forall n\in\mathbb{N}\setminus\{0,1\},\quad\det\Big(M_{\mathbf{m}n}\big(c_{\mathbf{m}}^{\kappa}(\widetilde{\gamma},\theta_1,\theta_2),\theta_1,\theta_2\big)\Big)\neq 0.$$
		Together with the fact that the matrix $M_{\mathbf{m}}\big(c_{\mathbf{m}}^{\kappa}(\widetilde{\gamma},\theta_1,\theta_2),\theta_1,\theta_2\big)$ is singular and non-zero, we obtain from \eqref{split dfG2} the desired result because
		$$\ker\Big(M_{\mathbf{m}}\big(c_{\mathbf{m}}^{\kappa}(\widetilde{\gamma},\theta_1,\theta_2),\theta_1,\theta_2\big)\Big)=\mathtt{span}\begin{pmatrix}
			-c_{\mathbf{m}}^{\kappa}(\widetilde{\gamma},\theta_1,\theta_2)+\frac{\omega_C-\omega_S}{2\mathbf{m}}+\frac{\omega_S}{2\sin^{2}\left(\frac{\theta_2}{2}\right)}+\widetilde{\gamma}\vspace{0.2cm}\\
			\frac{\omega_C-\omega_N}{2\mathbf{m}}\frac{\sin(\theta_1)}{\sin(\theta_2)}\tan^{\mathbf{m}}\left(\frac{\theta_1}{2}\right)\cot^{\mathbf{m}}\left(\frac{\theta_2}{2}\right)
		\end{pmatrix}.$$

		\medskip \noindent
		\textbf{(iii)} The next step is to describe the range. For this aim, we introduce on $Y_{\mathbf{m}}^{\alpha}\times Y_{\mathbf{m}}^{\alpha}$ the scalar product
		$$\left(\left(\sum_{n=1}^{\infty}a_n\sin(\mathbf{m}n\varphi),\sum_{n=1}^{\infty}c_n\sin(\mathbf{m}n\varphi)\right)\Big|\left(\sum_{n=1}^{\infty}b_n\sin(\mathbf{m}n\varphi),\sum_{n=1}^{\infty}d_n\sin(\mathbf{m}n\varphi)\right)\right)_2\triangleq\sum_{n=1}^{\infty}a_nb_n+c_nd_n.$$
		Now we claim that
		\begin{equation}\label{range2}
			\textnormal{Im}\Big(d_{(f_1,f_2)}\mathscr{G}\big(c_{\mathbf{m}}^{\kappa}(\widetilde{\gamma},\theta_1,\theta_2),0,0\big)\Big)=\mathtt{span}^{\perp_{(\cdot|\cdot)_2}}\left(g_0:\varphi\mapsto\begin{pmatrix}
				-c_{\mathbf{m}}^{\kappa}(\widetilde{\gamma},\theta_1,\theta_2)+\frac{\omega_S-\omega_C}{2\mathbf{m}}+\frac{\omega_S}{2\sin^{2}\left(\frac{\theta_2}{2}\right)}+\widetilde{\gamma}\vspace{0.2cm}\\
				\frac{\omega_S-\omega_C}{2\mathbf{m}}\frac{\sin(\theta_2)}{\sin(\theta_1)}\tan^{\mathbf{m}}\left(\frac{\theta_1}{2}\right)\cot^{\mathbf{m}}\left(\frac{\theta_2}{2}\right)
			\end{pmatrix}\sin(\mathbf{m}\varphi)\right).
		\end{equation}
		Indeed, as in the proof of Proposition \ref{prop hyp CR1}-(iii), we shall prove the first inclusion and the second one is obtained by the Fredholmness property and the previous point.
		First observe that
		$$v_{\mathbf{m}}^{\kappa}(\theta_1,\theta_2)\triangleq\begin{pmatrix}
			-c_{\mathbf{m}}^{\kappa}(\widetilde{\gamma},\theta_1,\theta_2)+\frac{\omega_C-\omega_S}{2\mathbf{m}}+\frac{\omega_S}{2\sin^{2}\left(\frac{\theta_2}{2}\right)}+\widetilde{\gamma}\vspace{0.2cm}\\
			\frac{\omega_S-\omega_C}{2\mathbf{m}}\frac{\sin(\theta_2)}{\sin(\theta_1)}\tan^{\mathbf{m}}\left(\frac{\theta_1}{2}\right)\cot^{\mathbf{m}}\left(\frac{\theta_2}{2}\right)
		\end{pmatrix}\in\ker\Big(M_{\mathbf{m}}^{\top}\big(c_{\mathbf{m}}^{\kappa}(\widetilde{\gamma},\theta_1,\theta_2),\theta_1,\theta_2\big)\Big).$$
		Now, consider
		$$g:\varphi\mapsto\sum_{n=1}^{\infty}\mathbf{m}nM_{\mathbf{m}n}\big(c_{\mathbf{m}}^{\kappa}(\widetilde{\gamma},\theta_1,\theta_2),\theta_1,\theta_2\big)\begin{pmatrix}
			h_{n}^{(1)}\\
			h_{n}^{(2)}
		\end{pmatrix}\sin(\mathbf{m}n\varphi)\in\textnormal{Im}\Big(d_{(f_1,f_2)}\mathscr{G}\big(c_{\mathbf{m}}^{\kappa}(\widetilde{\gamma},\theta_1,\theta_2),0,0\big)\Big).$$
		Then, denoting $\cdot$ the usual scalar product on $\mathbb{R}^2$, we have
		\begin{align*}
			\big(g\,|\,g_0\big)_2&=\left(\mathbf{m}M_{\mathbf{m}}\big(c_{\mathbf{m}}^{\kappa}(\widetilde{\gamma},\theta_1,\theta_2),\theta_1,\theta_2\big)\begin{pmatrix}
				h_{1}^{(1)}\\
				h_{1}^{(2)}
			\end{pmatrix}\right)\cdot v_{\mathbf{m}}^{\kappa}(\theta_1,\theta_2)\\
			&=\mathbf{m}\begin{pmatrix}
				h_{1}^{(1)}\\
				h_{1}^{(2)}
			\end{pmatrix}\cdot \Big(M_{\mathbf{m}}^{\top}\big(c_{\mathbf{m}}^{\kappa}(\widetilde{\gamma},\theta_1,\theta_2),\theta_1,\theta_2\big)v_{\mathbf{m}}^{\kappa}(\theta_1,\theta_2)\Big)\\
			&=0.
		\end{align*}
		This proves the claim. Now, we turn to the transversality condition. We shall prove that the following quantity does not vanish
		\begin{align*}
			&\Big(\partial_{c}d_{(f_1,f_2)}\mathscr{G}\big(c_{\mathbf{m}}^{\kappa}(\widetilde{\gamma},\theta_1,\theta_2),0,0\big)[u_0]\,|\,g_0\Big)_2\\
			&=\mathbf{m}\left[\left(-c_{\mathbf{m}}^{\kappa}(\widetilde{\gamma},\theta_1,\theta_2)+\frac{\omega_C-\omega_S}{2\mathbf{m}}+\frac{\omega_S}{2\sin^2\left(\frac{\theta_2}{2}\right)}+\widetilde{\gamma}\right)^2+\frac{1}{4\mathbf{m}^2}(\omega_N-\omega_C)(\omega_C-\omega_S)\tan^{2\mathbf{m}}\left(\tfrac{\theta_1}{2}\right)\cot^{2\mathbf{m}}\left(\tfrac{\theta_2}{2}\right)\right].
		\end{align*}
		Using the fact that $\det\Big(M_{\mathbf{m}}\big(c_{\mathbf{m}}^{\kappa}(\widetilde{\gamma},\theta_1,\theta_2),\theta_1,\theta_2\big)\Big)=0,$ we obtain
		\begin{align*}
			&\left(-c_{\mathbf{m}}^{\kappa}(\widetilde{\gamma},\theta_1,\theta_2)+\tfrac{\omega_C-\omega_S}{2\mathbf{m}}+\tfrac{\omega_S}{2\sin^2\left(\frac{\theta_2}{2}\right)}+\widetilde{\gamma}\right)^2+\tfrac{1}{4\mathbf{m}^2}(\omega_N-\omega_C)(\omega_C-\omega_S)\tan^{2\mathbf{m}}\left(\tfrac{\theta_1}{2}\right)\cot^{2\mathbf{m}}\left(\tfrac{\theta_2}{2}\right)\\
			&=\left(-c_{\mathbf{m}}^{\kappa}(\widetilde{\gamma},\theta_1,\theta_2)+\tfrac{\omega_C-\omega_S}{2\mathbf{m}}+\tfrac{\omega_S}{2\sin^2\left(\frac{\theta_2}{2}\right)}+\widetilde{\gamma}\right)\left(-2c_{\mathbf{m}}^{\kappa}(\widetilde{\gamma},\theta_1,\theta_2)+\tfrac{\omega_S}{2\sin^2\left(\frac{\theta_2}{2}\right)}-\tfrac{\omega_N}{2\cos^2\left(\frac{\theta_1}{2}\right)}+\tfrac{\omega_N-\omega_S}{2\mathbf{m}}+2\widetilde{\gamma}\right).
		\end{align*}
		According to the definition of $c_{\mathbf{m}}^{\kappa}$, the equation
		$$-2c_{\mathbf{m}}^{\kappa}(\widetilde{\gamma},\theta_1,\theta_2)+\frac{\omega_S}{2\sin^2\left(\frac{\theta_2}{2}\right)}-\frac{\omega_N}{2\cos^2\left(\frac{\theta_1}{2}\right)}+\frac{\omega_N-\omega_S}{2\mathbf{m}}+2\widetilde{\gamma}=0$$
		is equivalent to $\Delta_{\mathbf{m}}(\theta_1,\theta_2)=0$, which corresponds to a double eigenvalue. This is excluded by hypothesis.
		Besides, we can write
		$$-c_{\mathbf{m}}^{\kappa}(\widetilde{\gamma},\theta_1,\theta_2)+\frac{\omega_C-\omega_S}{2\mathbf{m}}+\frac{\omega_S}{2\sin^2\left(\frac{\theta_2}{2}\right)}+\widetilde{\gamma}=\frac{\omega_{N}}{4\cos^2\left(\frac{\theta_1}{2}\right)}+\frac{\omega_{S}}{4\sin^{2}\left(\frac{\theta_2}{2}\right)}-\frac{\omega_{N}+\omega_{S}-2\omega_C}{4\mathbf{m}}+\kappa\frac{1}{2}\sqrt{\Delta_{\mathbf{m}}(\theta_1,\theta_2)}.$$
		Assume for the sake of contradiction that 
		$$-c_{\mathbf{m}}^{\kappa}(\widetilde{\gamma},\theta_1,\theta_2)+\frac{\omega_C-\omega_S}{2\mathbf{m}}+\frac{\omega_S}{2\sin^2\left(\frac{\theta_2}{2}\right)}+\widetilde{\gamma}=0.$$
		This equation is equivalent to
		$$\frac{\omega_{N}}{2\cos^2\left(\frac{\theta_1}{2}\right)}+\frac{\omega_{S}}{2\sin^{2}\left(\frac{\theta_2}{2}\right)}-\frac{\omega_{N}+\omega_{S}-2\omega_C}{2\mathbf{m}}=-\kappa\sqrt{\Delta_{\mathbf{m}}(\theta_1,\theta_2)}.$$
		Taking the square, we end up with
		$$\frac{1}{\mathbf{m}^2}(\omega_{N}-\omega_C)(\omega_C-\omega_S)\tan^{2\mathbf{m}}\left(\tfrac{\theta_1}{2}\right)\cot^{2\mathbf{m}}\left(\tfrac{\theta_2}{2}\right)=0.$$
		But, by construction $\omega_N\neq\omega_C$ and $\omega_C\neq\omega_S.$ In addition $\tan\left(\tfrac{\theta_1}{2}\right)\cot\left(\tfrac{\theta_2}{2}\right)\in(0,1)$. Contradiction. Consequently,
		$$\Big(\partial_{c}d_{(f_1,f_2)}\mathscr{G}\big(c_{\mathbf{m}}^{\kappa}(\widetilde{\gamma},\theta_1,\theta_2),0,0\big)[u_0]\,|\,g_0\Big)_2\neq0.$$
		This ends the proof of Proposition \ref{prop hyp CR2}.
	\end{proof}
	Finally, from the previous analysis, we have proved the following theorem, which gives a refined version of Theorem \ref{thm bif vcap2}.	
	\begin{theo}\label{thm bif vcap3}
		Let $\widetilde{\gamma}\in\mathbb{R}$ and $0<\theta_1<\theta_2<\pi.$ Fix $\omega_N,\omega_C,\omega_S\in\mathbb{R}$ satisfying \eqref{Gauss thm2}.
		Then, there exists $N(\theta_1,\theta_2)\triangleq N(\theta_1,\theta_2,\omega_N,\omega_S,\omega_C)\in\mathbb{N}^*$  such that for any $\mathbf{m}\in\mathbb{N}^*$ with $\mathbf{m}\geqslant N(\theta_1,\theta_2)$, the following hold true.
		\begin{enumerate}
			\item If the condition \eqref{additional constraint} is satisfied then there exist two branches of $\mathbf{m}$-fold uniformly rotating vortex strips for \eqref{Euler eq on S2} bifurcating from the absolute vorticity distribution $\overline{\Omega}_{\textnormal{\tiny{FC2}}}$, given by \eqref{omegaFC2} at the velocities $c_{\mathbf{m}}^{\pm}(\widetilde{\gamma},\theta_1,\theta_2)$, defined in \eqref{def cn+- alter}.
			\item  Fix $\kappa\in\{+,-\}$ and assume that the condition \eqref{additional constraint2}, supplemented with $\mathbf{(Hk\kappa)}$ for some $k\in\llbracket 1,4\rrbracket$, holds.  Then, there exist one  branch of $\mathbf{m}$-fold uniformly rotating vortex strips for \eqref{Euler eq on S2} bifurcating from $\overline{\Omega}_{\textnormal{\tiny{FC2}}}$ at the velocity $c_{\mathbf{m}}^{\kappa}(\widetilde{\gamma},\theta_1,\theta_2)$.
		\end{enumerate}
		
	\end{theo}

	\appendix
	\section{Appendix}
	\subsection{An integral}\label{app-integral}
	In this appendix, we give an explicit value for an integral with parameters oftenly used in this work.
	\begin{lem}\label{lem int}
		Let $n\in\mathbb{N}^*$ and $a,b\in(0,\pi)$. We define
		$$I_{n}(a,b)\triangleq\frac{1}{2\pi}\int_{0}^{2\pi}\cos(nx)\log\big(1-\cos(a)\cos(b)-\sin(a)\sin(b)\cos(x)\big)dx.$$
		Then
		$$I_{n}(a,b)=I_{n}(b,a)=-\frac{1}{n}\tan^n\left(\frac{\min(a,b)}{2}\right)\cot^n\left(\frac{\max(a,b)}{2}\right).$$
	\end{lem}
	\begin{proof}
		$\blacktriangleright$ Let us first begin with the case $a=b.$ Observe that
		\begin{align*}
			\log\big(1-\cos^2(a)-\sin^2(a)\cos(x)\big)&=\log\big(1-\cos(x)\big)+\log\big(\sin^2(a)\big)\\
			&=\log\Big(\sin^2\big(\tfrac{x}{2}\big)\Big)+\log(2)+\log\big(\sin^2(a)\big).
		\end{align*}
		Hence, using \cite[Lem. A.3]{CCG16}, we get
		$$I_{n}(a,a)=\frac{1}{2\pi}\int_{0}^{2\pi}\log\Big(\sin^2\big(\tfrac{x}{2}\big)\Big)\cos(nx)dx=-\frac{1}{n}\cdot$$
		$\blacktriangleright$ Now, we assume $a\neq b$ with, without loss of generality, $a<b.$ We can write
		\begin{align*}
			I_{n}(a,b)&=\frac{1}{\pi}\int_{0}^{\pi}\cos(nx)\log\big(1-\mu_{a,b}\cos(x)\big)dx,\qquad\mu_{a,b}\triangleq\frac{\sin(a)\sin(b)}{1-\cos(a)\cos(b)}>0.
		\end{align*}
		Notice that 
		\begin{align*}
			a\neq b\Leftrightarrow\cos(a-b)<1&\Leftrightarrow\cos(a)\cos(b)+\sin(a)\sin(b)<1\\
			&\Leftrightarrow\sin(a)\sin(b)<1-\cos(a)\cos(b)\\
			&\Leftrightarrow\mu_{a,b}<1.
		\end{align*}
		Performing an integration by parts yields
		$$I_{n}(a,b)=-\frac{\mu_{a,b}}{n\pi}\int_{0}^{\pi}\frac{\sin(nx)\sin(x)}{1-\mu_{a,b}\cos(x)}dx.$$
		Now we shall use the following result which can be found in \cite[p. 391]{GR15}
		$$\frac{1}{1+\alpha^2}\int_{0}^{\pi}\frac{\sin(nx)\sin(x)}{1-\frac{2\alpha}{1+\alpha^2}\cos(x)}dx=\int_{0}^{\pi}\frac{\sin(nx)\sin(x)}{1-2\alpha\cos(x)+\alpha^2}dx=\begin{cases}
			\tfrac{\pi}{2}\alpha^{n-1}, & \textnormal{if }\alpha^2<1,\\
			\tfrac{\pi}{2\alpha^{n+1}}, & \textnormal{if }\alpha^2>1.
		\end{cases}$$
		We apply it with
		$$\frac{2\alpha}{1+\alpha^2}=\mu_{a,b},\qquad\textnormal{i.e.}\qquad\mu_{a,b}\alpha^2-2\alpha+\mu_{a,b}=0.$$
		The discriminant of the previous second order polynomial equation is $\Delta=4(1-\mu_{a,b}^2)>0,$ so we can take
		$$\alpha=\alpha_{a,b}\triangleq\frac{1-\sqrt{1-\mu_{a,b}^2}}{\mu_{a,b}}\in(0,1).$$
		Observe that
		\begin{align*}
			\alpha_{a,b}<1&\Leftrightarrow 1-\mu_{a,b}<\sqrt{1-\mu_{a,b}^2}\\
			&\Leftrightarrow(1-\mu_{a,b})^2<1-\mu_{a,b}^2\\
			&\Leftrightarrow\mu_{a,b}<1.
		\end{align*}
		Hence,
		$$I_{n}(a,b)=-\frac{\mu_{a,b}(1+\alpha_{a,b}^2)\alpha_{a,b}^{n-1}}{2n}=-\frac{\alpha_{a,b}^{n}}{n}\cdot$$
		We can also write
		\begin{align*}
			\alpha_{a,b}&=\frac{1}{\mu_{a,b}}-\sqrt{\frac{1}{\mu_{a,b}^2}-1}\\
			&=\frac{1-\cos(a)\cos(b)-\sqrt{\big(1-\cos(a)\cos(b)\big)^2-\sin^2(a)\sin^2(b)}}{\sin(a)\sin(b)}\cdot
		\end{align*}
		But
		\begin{align*}
			\big(1-\cos(a)\cos(b)\big)^2-\sin^2(a)\sin^2(b)&=1+\cos^2(a)\cos^2(b)-2\cos(a)\cos(b)-\big(1-\cos^2(a)\big)\big(1-\cos^2(b)\big)\\
			&=\cos^2(a)+\cos^2(b)-2\cos(a)\cos(b)\\
			&=\big(\cos(a)-\cos(b)\big)^2.
		\end{align*}
		Since $a,b\in(0,\pi)$ with $a<b,$ then $\cos(a)>\cos(b).$ Consequently,
		\begin{align*}
			\alpha_{a,b}&=\frac{1-\cos(a)\cos(b)-\big(\cos(a)-\cos(b)\big)}{\sin(a)\sin(b)}\\
			&=\frac{\big(1-\cos(a)\big)\big(1+\cos(b)\big)}{\sin(a)\sin(b)}\\
			&=\frac{2\sin^2\left(\frac{a}{2}\right)2\cos^2\left(\frac{b}{2}\right)}{2\sin\left(\frac{a}{2}\right)\cos\left(\frac{a}{2}\right)2\sin\left(\frac{b}{2}\right)\cos\left(\frac{b}{2}\right)}\\
			&=\tan\left(\tfrac{a}{2}\right)\cot\left(\tfrac{b}{2}\right).
		\end{align*}
		Finally, we have
		$$\forall n\in\mathbb{N}^*,\quad\forall\,0<a<b<\pi,\quad I_{n}(a,b)=-\frac{\tan^n\left(\frac{a}{2}\right)\cot^n\left(\frac{b}{2}\right)}{n}\cdot$$
	\end{proof}
	\subsection{Potential theory}\label{app-potential-theory}
	This section is devoted to some results on  the continuity of specific  operators with singular kernels. The proof of the next result can be found in \cite[Lem. 2.6]{HXX22}.
	\begin{prop}\label{prop-potentialtheory}
		Let $\alpha\in(0,1).$ Consider a kernel $K:\mathbb{T}\times\mathbb{T}\rightarrow \mathbb{R}$ smooth out of the diagonal and satisfying, for some $C_0>0$,
		\begin{align}
			\forall\varphi\neq\varphi'\in\mathbb{T},\quad|K(\varphi,\varphi')|&\leqslant C_0\left|\sin\left(\tfrac{\varphi-\varphi'}{2}\right)\right|^{-(1-\alpha)},\label{prop-potentialtheory-h0}\\ 
			\forall\varphi\neq\varphi'\in\mathbb{T},\quad|\partial_{\varphi} K(\varphi,\varphi')| &\leqslant C_0\left|\sin\left(\tfrac{\varphi-\varphi'}{2}\right)\right|^{-(2-\alpha)}.\label{prop-potentialtheory-h2}
		\end{align}
		Then, the integral operator $\mathcal{K}$ defined by
		$$\forall\varphi\in\mathbb{T},\quad\mathcal{K}(f)(\varphi)\triangleq \int_0^{2\pi} K(\varphi,\varphi')f(\varphi')d\varphi'$$
		is bounded from $L^\infty(\mathbb{T})$ into $C^\alpha(\mathbb{T}).$ More precisely, we have the following estimate
		$$\forall f\in L^{\infty}(\mathbb{T}),\quad\|\mathcal{K}(f)\|_{C^\alpha(\mathbb{T})}\leqslant C C_0 \|f\|_{L^\infty(\mathbb{T})},$$
		with $C>0$ an absolute constant. 
	\end{prop}

	In some cases, the above proposition cannot be applied directly because the kernel $K$ has a non differentiable term, and thus the condition \eqref{prop-potentialtheory-h2} does not make sense. In those cases, let us give the alternative result.
	\begin{prop}\label{prop-potentialtheory-2}
		Let $\alpha\in(0,1)$ and $g\in C^\alpha(\mathbb{T})$, consider a kernel $K:\mathbb{T}\times\mathbb{T}\rightarrow \mathbb{R}$ smooth out of the diagonal and satisfying
		\begin{align}
			\forall\varphi\neq\varphi'\in\mathbb{T},\quad|K(\varphi,\varphi')|&\leqslant C_0\left|\sin\left(\tfrac{\varphi-\varphi'}{2}\right)\right|^{-1},\label{prop-potentialtheory-h0-2}\\ 
			\forall\varphi\neq\varphi'\in\mathbb{T},\quad|\partial_{\varphi} K(\varphi,\varphi')| &\leqslant C_0\left|\sin\left(\tfrac{\varphi-\varphi'}{2}\right)\right|^{-2}.\label{prop-potentialtheory-h2-2}
		\end{align}
		Then, the integral operator $\widetilde{\mathcal{K}}_g$ defined by
		$$\forall\varphi\in\mathbb{T},\quad\widetilde{\mathcal{K}}_g(f)(\varphi)\triangleq \int_0^{2\pi} K(\varphi,\varphi')\big(g(\varphi)-g(\varphi')\big)f(\varphi')d\varphi'$$ is bounded from $L^\infty(\mathbb{T})$ into $C^\alpha(\mathbb{T}).$ More precisely, we have the following estimate
		$$\forall f\in L^{\infty}(\mathbb{T}),\quad\|\widetilde{\mathcal{K}}_g(f)\|_{C^\alpha(\mathbb{T})}\leqslant C C_0\|g\|_{C^\alpha(\mathbb{T})} \|f\|_{L^\infty(\mathbb{T})},$$
		with $C>0$ an absolute constant. 
	\end{prop}
	\begin{proof}
		The $L^\infty$ norm of $\widetilde{\mathcal{K}}_g(f)$ can be estimated as
		\begin{align*}
			\left|\widetilde{\mathcal{K}}_g(f)(\varphi)\right|&\leqslant C\|f\|_{L^\infty(\mathbb{T})}\|g\|_{C^\alpha(\mathbb{T})}\int_0^{2\pi} |K(\varphi,\varphi')||\varphi-\varphi'|^\alpha d\varphi'\\
			&\leqslant  C C_0\|f\|_{L^\infty(\mathbb{T})}\|g\|_{C^\alpha(\mathbb{T})}\ \sup_{\varphi\in\mathbb{T}}\int_0^{2\pi} \left|\sin\left(\tfrac{\varphi-\varphi'}{2}\right)\right|^{-1}|\varphi-\varphi'|^{\alpha}d\varphi'.
		\end{align*}
		Since we work on the torus, we can always assume that 
		$$\forall\varphi\neq\varphi'\in\mathbb{T},\quad0<|\varphi-\varphi'|\leqslant\pi.$$ 
		As a consequence, we have the following classical convexity estimate
		\begin{equation}\label{convex sin}
			\forall\varphi\neq\varphi'\in\mathbb{T},\quad\tfrac{2}{\pi}|\varphi-\varphi'|\leqslant2\left|\sin\left(\tfrac{\varphi-\varphi'}{2}\right)\right|\leqslant|\varphi-\varphi'|.
		\end{equation}
		Therefore, by using \eqref{convex sin}, a change of variable and the fact that $\alpha\in (0,1)$, we get
		$$\forall\varphi\in\mathbb{T},\quad\int_0^{2\pi} \left|\sin\left(\tfrac{\varphi-\varphi'}{2}\right)\right|^{-1}|\varphi-\varphi'|^{\alpha}d\varphi'\leqslant C\int_0^{2\pi} \left|\sin\left(\tfrac{\varphi'}{2}\right)\right|^{-(1-\alpha)}d\varphi'<\infty.$$
		Hence, we obtain
		$$
		\|\widetilde{\mathcal{K}}_g(f)\|_{L^\infty(\mathbb{T})}\leqslant C C_0\|f\|_{L^\infty(\mathbb{T})}\|g\|_{C^\alpha(\mathbb{T})}.
		$$
		For the H\"older regularity, take $\varphi_1\neq\varphi_2 \in\mathbb{T}$. 
		Define 
		$$d\triangleq2\left|\sin\left(\tfrac{\varphi_1-\varphi_2}{2}\right)\right|=\big|e^{\mathrm{i}\varphi_1}-e^{\mathrm{i}\varphi_2}\big|$$
		and for $\varphi\in\mathbb{T}$ and $r>0,$
		$$B_{\varphi}(r)\triangleq\left\{\varphi'\in\mathbb{T}\quad\textnormal{s.t.}\quad2\left|\sin\left(\tfrac{\varphi-\varphi'}{2}\right)\right|<r\right\},\qquad B^c_{\varphi}(r)\triangleq\mathbb{T}\setminus B_{\varphi}(r).$$
		Hence
		\begin{align*}
			\widetilde{\mathcal{K}}_g(f)(\varphi_1)-\widetilde{\mathcal{K}}_g(f)(\varphi_2)
			=&\int_0^{2\pi}K(\varphi_1,\varphi')\big(g(\varphi_1)-g(\varphi')\big)f(\varphi')d\varphi'-\int_0^{2\pi}K(\varphi_2,\varphi')\big(g(\varphi_2)-g(\varphi')\big)f(\varphi')d\varphi'\\
			=&\int_{B_{\varphi_1}(3d)}K(\varphi_1,\varphi')\big(g(\varphi_1)-g(\varphi')\big)f(\varphi')d\varphi'\\
			&-\int_{B_{\varphi_1}(3d)}K(\varphi_2,\varphi')\big(g(\varphi_2)-g(\varphi')\big)f(\varphi')d\varphi'\\
			&+\int_{B^c_{\varphi_1}(3d)}(K(\varphi_1,\varphi')-K(\varphi_2,\varphi'))\big(g(\varphi_1)-g(\varphi')\big)f(\varphi')d\varphi'\\
			&+\int_{B^c_{\varphi_1}(3d)}K(\varphi_2,\varphi')\big(g(\varphi_1)-g(\varphi_2)\big)f(\varphi')d\varphi'\\
			\triangleq& I_1+I_2+I_3+I_4.
		\end{align*}
		Using \eqref{prop-potentialtheory-h0}, \eqref{convex sin} and a change of variables, we arrive at
		\begin{align*}
			|I_1|&\leqslant CC_0\|f\|_{L^\infty(\mathbb{T})}\|g\|_{C^\alpha(\mathbb{T})}\int_{B_{\varphi_1}(3d)}\left|\sin\left(\tfrac{\varphi_1-\varphi'}{2}\right)\right|^{-1}|\varphi_1-\varphi'|^{\alpha}d\varphi'\\
			&\leqslant CC_0\|f\|_{L^\infty(\mathbb{T})}\|g\|_{C^\alpha(\mathbb{T})}\int_{B_{\varphi_1}(3d)}\left|\sin\left(\tfrac{\varphi_1-\varphi'}{2}\right)\right|^{-(1-\alpha)}d\varphi'\\
			&\leqslant CC_0\|f\|_{L^\infty(\mathbb{T})}\|g\|_{C^\alpha(\mathbb{T})}\int_{0}^{\frac{3}{2}d}\frac{dw}{|w|^{1-\alpha}\sqrt{1-w^2}}dw\\
			&\leqslant  CC_0\|f\|_{L^\infty(\mathbb{T})}\|g\|_{C^\alpha(\mathbb{T})} d^\alpha\\
			&=CC_0\|f\|_{L^\infty(\mathbb{T})} \|g\|_{C^\alpha(\mathbb{T})}|\varphi_1-\varphi_2|^\alpha.
		\end{align*}
		In order to work with $I_2$, note that $B_{\varphi_1}(3d)\subset B_{\varphi_2}(4d)$. Thus, proceding as before, we infer
		\begin{align*}
			|I_2|&\leqslant CC_0\|f\|_{L^\infty(\mathbb{T})}\|g\|_{C^\alpha(\mathbb{T})} \int_{B_{\varphi_1}(3d)}\left|\sin\left(\tfrac{\varphi_2-\varphi'}{2}\right)\right|^{-1}|\varphi_2-\varphi'|^{\alpha}d\varphi'\\
			&\leqslant CC_0\|f\|_{L^\infty(\mathbb{T})}\|g\|_{C^\alpha(\mathbb{T})} \int_{B_{\varphi_2}(4d)}\left|\sin\left(\tfrac{\varphi_2-\varphi'}{2}\right)\right|^{-1}|\varphi_2-\varphi'|^{\alpha}d\varphi'\\
			&\leqslant CC_0\|f\|_{L^\infty(\mathbb{T})}\|g\|_{C^\alpha(\mathbb{T})} \int_{B_{\varphi_2}(4d)}\left|\sin\left(\tfrac{\varphi_2-\varphi'}{2}\right)\right|^{-(1-\alpha)}d\varphi'\\
			&\leqslant CC_0\|f\|_{L^\infty(\mathbb{T})}\|g\|_{C^\alpha(\mathbb{T})} |\varphi_1-\varphi_2|^\alpha.
		\end{align*}
		For the third term $I_3$ we use the mean value theorem and \eqref{prop-potentialtheory-h2} achieving
		\begin{align*}
			|I_3|&\leqslant C \left|(\varphi_1-\varphi_2)\int_0^1\int_{B^c_{\varphi_1}(3d)}(\partial_{x}K)\big(\varphi_1+(1-s)(\varphi_2-\varphi_1),\varphi'\big)\big(g(\varphi_1)-g(\varphi')\big)f(\varphi')d\varphi' ds\right|\\
			&\leqslant C C_0 \|f\|_{L^\infty(\mathbb{T})}\|g\|_{C^\alpha(\mathbb{T})}|\varphi_1-\varphi_2| \int_0^1\int_{B^c_{\varphi_1}(3d)}\left|\sin\left(\tfrac{\varphi_1+(1-s)(\varphi_2-\varphi_1)-\varphi'}{2}\right)\right|^{-2}|\varphi_1-\varphi'|^\alpha d\varphi'ds.
		\end{align*}
		Note that if $\varphi'\in B^c_{\varphi_1}(3d)$ and $s\in[0,1]$, then
		\begin{align}
			2\left|\sin\left(\tfrac{\varphi_1+(1-s)(\varphi_2-\varphi_1)-\varphi'}{2}\right)\right|&=\left|e^{\mathrm{i}(\varphi_1-\varphi')}-e^{\mathrm{i}(1-s)(\varphi_1-\varphi_2)}\right|\nonumber\\
			&\geqslant\left|e^{\mathrm{i}(\varphi_1-\varphi')}-1\right|-\left|e^{\mathrm{i}(1-s)(\varphi_1-\varphi_2)}-1\right|\nonumber\\
			&\geqslant\left|e^{\mathrm{i}(\varphi_1-\varphi')}-1\right|-\left|e^{\mathrm{i}(\varphi_1-\varphi_2)}-1\right|=\left|e^{\mathrm{i}\varphi_1}-e^{\mathrm{i}\varphi'}\right|-\left|e^{\mathrm{i}\varphi_1}-e^{\mathrm{i}\varphi_2}\right|\nonumber\\
			&\geqslant\tfrac{2}{3}\left|e^{\mathrm{i}\varphi_1}-e^{\mathrm{i}\varphi'}\right|=\tfrac{4}{3}\left|\sin\left(\tfrac{\varphi_1-\varphi'}{2}\right)\right|\label{e-sin}
		\end{align}
		which implies, through \eqref{convex sin} and a change of variables,
		\begin{align*}
			|I_3|&\leqslant C C_0 \|f\|_{L^\infty(\mathbb{T})}\|g\|_{C^\alpha(\mathbb{T})} |\varphi_1-\varphi_2|\int_{B^c_{\varphi_1}(3d)}\left|\sin\left(\tfrac{\varphi_1-\varphi'}{2}\right)\right|^{-2}|\varphi_1-\varphi'|^{\alpha}d\varphi'\\
			&\leqslant C C_0 \|f\|_{L^\infty(\mathbb{T})}\|g\|_{C^\alpha(\mathbb{T})} |\varphi_1-\varphi_2|\int_{B^c_{\varphi_1}(3d)}\left|\sin\left(\tfrac{\varphi_1-\varphi'}{2}\right)\right|^{-(2-\alpha)}d\varphi'\\
			&\leqslant C C_0 \|f\|_{L^\infty(\mathbb{T})}\|g\|_{C^\alpha(\mathbb{T})} |\varphi_1-\varphi_2|\int_{d}^{1}\frac{dw}{|w|^{2-\alpha}\sqrt{1-w^2}}\\
			&\leqslant C C_0 \|f\|_{L^\infty(\mathbb{T})}\|g\|_{C^\alpha(\mathbb{T})} |\varphi_1-\varphi_2|\frac{1}{|\varphi_1-\varphi_2|^{1-\alpha}}\\
			&\leqslant C C_0 \|f\|_{L^\infty(\mathbb{T})}\|g\|_{C^\alpha(\mathbb{T})} |\varphi_1-\varphi_2|^\alpha.
		\end{align*}
		Let us finish with $I_4$, using that $g\in C^\alpha(\mathbb{T})$,
		\begin{align*}
			|I_4|\leqslant C C_0 \|f\|_{L^\infty(\mathbb{T}) }\|g\|_{C^\alpha(\mathbb{T})}\int_{B^c_{\varphi_1}(3d)}\left|\sin\left(\tfrac{\varphi_2-\varphi'}{2}\right)\right|^{-1}|\varphi_1-\varphi_2|^\alpha d\varphi'.
		\end{align*}
		Applying \eqref{e-sin} with $s=0$, we get
		$$2\left|\sin\left(\tfrac{\varphi_2-\varphi'}{2}\right)\right|\geqslant\tfrac{4}{3}\left|\sin\left(\tfrac{\varphi_1-\varphi'}{2}\right)\right|.$$
		Besides, for $\varphi'\in B_{\varphi_1}^c(3d),$ we have
		$$\left|\sin\left(\tfrac{\varphi_1-\varphi_2}{2}\right)\right|\leqslant\tfrac{1}{3}\left|\sin\left(\tfrac{\varphi_1-\varphi'}{2}\right)\right|.$$
		Combining the foregoing facts, we end up with
		\begin{align*}
			|I_4|&\leqslant C C_0 \|f\|_{L^\infty(\mathbb{T}) }\|g\|_{C^\alpha(\mathbb{T})}\int_{B^c_{\varphi_1}(3d)}\left|\sin\left(\tfrac{\varphi_1-\varphi'}{2}\right)\right|^{-(1-\alpha)}d\varphi'\\
			&\leqslant C C_0 \|f\|_{L^\infty(\mathbb{T}) }\|g\|_{C^\alpha(\mathbb{T})} |\varphi_1-\varphi_2|^\alpha.
		\end{align*}
		Putting together the preceding estimates yields
		$$
		\left|\widetilde{\mathcal{K}}_g(f)(\varphi_1)-\widetilde{\mathcal{K}}_g(f)(\varphi_2)\right| \leqslant C C_0 \|f\|_{L^\infty(\mathbb{T})}\|g\|_{C^\alpha(\mathbb{T})} |\varphi_1-\varphi_2|^\alpha,
		$$
		concluding the proof.
	\end{proof}

	\subsection{Crandall-Rabinowitz theorem}\label{sec-CR}
	In this last appendix, we recall the classical Crandall-Rabinowitz Theorem whose proof can be found  in \cite{CR71}.
	
	\begin{theo}[Crandall-Rabinowitz Theorem]\label{CR}
		Let $\lambda_0\in\mathbb{R},$ $X,Y$ be two Banach spaces, $V$ be a neighborhood of $0$ in $X$ and $\mathscr{F}:\mathbb{R}\times V\rightarrow Y$ be a function with the properties,
		\begin{enumerate}
			\item $\mathscr{F}(\lambda,0)=0$ for all $\lambda\in\mathbb{R}$.
			\item The partial derivatives  $\partial_\lambda \mathscr{F}$, $d_f\mathscr{F}$ and  $\partial_{\lambda}d_f\mathscr{F}$ exist and are continuous.
			\item The operator $d_f \mathscr{F}(\lambda_0,0)$ is Fredholm of zero index and $\ker\big(d_f\mathscr{F}(\lambda_0,0)\big)=\mathtt{span}(f_0)$ is one-dimensional. 
			\item  Transversality assumption: $\partial_{\lambda}d_f\mathscr{F}(\lambda_0,0)[f_0] \notin \textnormal{Im}\big(d_f\mathscr{F}(\lambda_0,0)\big)$.
		\end{enumerate}
		If $Z$ is any complement of  $\ker\big(d_f\mathscr{F}(\lambda_0,0)\big)$ in $X$, then there is a neighborhood $U$ of $(\lambda_0,0)$ in $\mathbb{R}\times X$, an interval  $(-a,a)$ with $a>0$, and two continuous functions $\Phi:(-a,a)\rightarrow\mathbb{R}$, $\beta:(-a,a)\rightarrow Z$ such that $\Phi(0)=\lambda_0$ and $\beta(0)=0$ and
		$$\mathscr{F}^{-1}(0)\cap U=\big\{\big(\Phi(s), s f_0+s\beta(s)\big) : |s|<a\big\}\cup\big\{(\lambda,0): (\lambda,0)\in U\big\}.$$
	\end{theo}	
	
	\subsection{Proof of Lemma \ref{lem sat vort}}\label{proof lemma}
	
	Fix $\alpha\in\mathbb{R}.$ Since $\mathcal{R}(\alpha)\in SO_3(\mathbb{R}),$ then it preserves the Euclidean norm $|\cdot|_{\mathbb{R}^3},$ i.e.
	$$\forall\xi\in\mathbb{R}^3,\quad\big|\mathcal{R}(\alpha)\xi\big|_{\mathbb{R}^3}=|\xi|_{\mathbb{R}^3}.$$
	As a consequence,
	\begin{align*}
		\forall(\xi,\xi')\in(\mathbb{S}^2)^2,\quad\forall\alpha\in\mathbb{R},\quad G\big(\mathcal{R}(\alpha)\xi,\mathcal{R}(\alpha)\xi'\big)&=\frac{1}{2\pi}\log\left(\frac{\big|\mathcal{R}(\alpha)\xi-\mathcal{R}(\alpha)\xi'\big|_{\mathbb{R}^3}}{2}\right)\\
		&=\frac{1}{2\pi}\log\left(\frac{\big|\mathcal{R}(\alpha)(\xi-\xi')\big|_{\mathbb{R}^3}}{2}\right)\\
		&=\frac{1}{2\pi}\log\left(\frac{|\xi-\xi'|_{\mathbb{R}^3}}{2}\right)\\
		&=G(\xi,\xi').
	\end{align*}
	Hence, using the change of variables $\xi'\mapsto\mathcal{R}(\alpha)\xi'\in SO(\mathbb{R}^3)$ (which preserves $\mathbb{S}^2$), we get for any $\xi\in\mathbb{S}^2,$
	\begin{align}\label{inv rot Psi}
		\Psi\big(\mathcal{R}(\alpha)\xi\big)&=\int_{\mathbb{S}^2}G\big(\mathcal{R}(\alpha)\xi,\xi'\big)\Omega(\xi')d\xi'\nonumber\\
		&=\int_{\mathbb{S}^2}G\big(\mathcal{R}(\alpha)\xi,\mathcal{R}(\alpha)\xi'\big)\Omega\big(\mathcal{R}(\alpha)\xi'\big)d\xi'\nonumber\\
		&=\int_{\mathbb{S}^2}G(\xi,\xi')\Omega(\xi')d\xi'\nonumber\\
		&=\Psi(\xi).
	\end{align}
	This achieves the proof of Lemma \ref{lem sat vort}.
	
	\bibliography{references}	

\renewcommand{\MR}[1]{}\def\cprime{$'$}
\begin{thebibliography}{10}

\bibitem{ADPM21}
Weiwei Ao, Juan D\'{a}vila, Manuel del Pino, Monica Musso, and Juncheng Wei.
\newblock Travelling and rotating solutions to the generalized inviscid surface
  quasi-geostrophic equation.
\newblock {\em Transactions of the American Mathematical Society},
  374(9):6665--6689, 2021.

\bibitem{AH12}
Kendall Atkinson and Weimin Han.
\newblock {\em Spherical harmonics and approximations on the unit sphere: an
  introduction}, volume 2044 of {\em Lecture Notes in Mathematics}.
\newblock Springer, Heidelberg, 2012.

\bibitem{BHM22}
Massimiliano Berti, Zineb Hassainia, and Nader Masmoudi.
\newblock Time quasi-periodic vortex patches of euler equation in the plane.
\newblock {\em Inventiones mathematicae}, 233(3):1279--1391, 2023.

\bibitem{BD15}
Stefanella Boatto and David~G. Dritschel.
\newblock The motion of point vortices on closed surfaces.
\newblock {\em Proceedings of the Royal Society A - Mathematical, Physical and
  Engineering Sciences}, 471(2176):20140890, 25, 2015.

\bibitem{B82}
Jacob Burbea.
\newblock Motions of vortex patches.
\newblock {\em Letters in Mathematical Physics}, 6(1):1--16, 1982.

\bibitem{CLZ21}
Daomin Cao, Shanfa Lai, and Weicheng Zhan.
\newblock Traveling vortex pairs for 2{D} incompressible {E}uler equations.
\newblock {\em Calculus of Variations and Partial Differential Equations},
  60(5):Paper No. 190, 16, 2021.

\bibitem{CLW14}
Daomin Cao, Zhongyuan Liu, and Juncheng Wei.
\newblock Regularization of point vortices pairs for the {E}uler equation in
  dimension two.
\newblock {\em Archive for Rational Mechanics and Analysis}, 212(1):179--217,
  2014.

\bibitem{CQZZ21}
Daomin Cao, Guolin Qin, Weicheng Zhan, and Changjun Zou.
\newblock Existence and regularity of co-rotating and traveling-wave vortex
  solutions for the generalized {SQG} equation.
\newblock {\em Journal of Differential Equations}, 299:429--462, 2021.

\bibitem{CQZZ22}
Daomin Cao, Guolin Qin, Weicheng Zhan, and Changjun Zou.
\newblock On the global classical solutions for the generalized {SQG} equation.
\newblock {\em Journal of Functional Analysis}, 283(2):Paper No. 109503, 37,
  2022.

\bibitem{CW22}
Daomin Cao and Jie Wan.
\newblock Multiscale steady vortex patches for 2{D} incompressible {E}uler
  equations.
\newblock {\em SIAM Journal on Mathematical Analysis}, 54(2):1488--1514, 2022.

\bibitem{CWWZ21}
Daomin Cao, Jie Wan, Guodong Wang, and Weicheng Zhan.
\newblock Rotating vortex patches for the planar {E}uler equations in a disk.
\newblock {\em Journal of Differential Equations}, 275:509--532, 2021.

\bibitem{CWZ23}
Daomin Cao, Guodong Wang, and Bijun Zuo.
\newblock Stability of degree-2 rossby-haurwitz waves.
\newblock {\em arXiv preprint arXiv:2305.03279}, 2023.

\bibitem{CM88}
Silvia Caprino and Carlo Marchioro.
\newblock On nonlinear stability of stationary {E}uler flows on a rotating
  sphere.
\newblock {\em Journal of Mathematical Analysis and Applications},
  129(1):24--36, 1988.

\bibitem{CCGS16}
Angel Castro, Diego C\'{o}rdoba, and Javier G\'{o}mez-Serrano.
\newblock Existence and regularity of rotating global solutions for the
  generalized surface quasi-geostrophic equations.
\newblock {\em Duke Mathematical Journal}, 165(5):935--984, 2016.

\bibitem{CCG16}
Angel Castro, Diego C\'{o}rdoba, and Javier G\'{o}mez-Serrano.
\newblock Uniformly rotating analytic global patch solutions for active
  scalars.
\newblock {\em Annals of PDE}, 2(1):Art. 1, 34, 2016.

\bibitem{CCGS19}
Angel Castro, Diego C\'{o}rdoba, and Javier G\'{o}mez-Serrano.
\newblock Uniformly rotating smooth solutions for the incompressible 2{D}
  {E}uler equations.
\newblock {\em Archive for Rational Mechanics and Analysis}, 231(2):719--785,
  2019.

\bibitem{CCGS20}
Angel Castro, Diego C\'{o}rdoba, and Javier G\'{o}mez-Serrano.
\newblock Global smooth solutions for the inviscid {SQG} equation.
\newblock {\em Memoirs of the American Mathematical Society}, 266(1292):v+89,
  2020.

\bibitem{CR21}
Christophe Cheverry and Nicolas Raymond.
\newblock {\em A guide to spectral theory Applications and exercises}.
\newblock Birkh\"{a}user/Springer, Cham, 2021.

\bibitem{CG22}
Adrian Constantin and Pierre Germain.
\newblock Stratospheric planetary flows from the perspective of the {E}uler
  equation on a rotating sphere.
\newblock {\em Archive for Rational Mechanics and Analysis}, 245(1):587--644,
  2022.

\bibitem{CR71}
Michael~G. Crandall and Paul~H. Rabinowitz.
\newblock Bifurcation from simple eigenvalues.
\newblock {\em Journal of Functional Analysis}, 8:321--340, 1971.

\bibitem{DPMW20}
Juan Davila, Manuel Del~Pino, Monica Musso, and Juncheng Wei.
\newblock Gluing methods for vortex dynamics in {E}uler flows.
\newblock {\em Archive for Rational Mechanics and Analysis}, 235(3):1467--1530,
  2020.

\bibitem{DPMW22}
Juan Davila, Manuel del Pino, Monica Musso, and Juncheng Wei.
\newblock Leapfrogging vortex rings for the 3-dimensional incompressible euler
  equations.
\newblock {\em Communications on Pure and Applied Mathematics}, 77:3843--3957,
  2024.

\bibitem{HHH16}
Francisco de~la Hoz, Zineb Hassainia, and Taoufik Hmidi.
\newblock Doubly connected {V}-states for the generalized surface
  quasi-geostrophic equations.
\newblock {\em Archive for Rational Mechanics and Analysis}, 220(3):1209--1281,
  2016.

\bibitem{HHHM15}
Francisco de~la Hoz, Zineb Hassainia, Taoufik Hmidi, and Joan Mateu.
\newblock An analytical and numerical study of steady patches in the disc.
\newblock {\em Analysis \& PDE}, 9(7):1609--1670, 2016.

\bibitem{HHMV16}
Francisco de~la Hoz, Taoufik Hmidi, Joan Mateu, and Joan Verdera.
\newblock Doubly connected {$V$}-states for the planar {E}uler equations.
\newblock {\em SIAM Journal on Mathematical Analysis}, 48(3):1892--1928, 2016.

\bibitem{DZ78}
Gary~S. Deem and Norman~J. Zabusky.
\newblock Vortex waves: Stationary ``${V}$ states,'' interactions, recurrence,
  and breaking.
\newblock {\em Physical Review Letters}, 40:859--862, Mar 1978.

\bibitem{DHR19}
David~G. Dritschel, Taoufik Hmidi, and Coralie Renault.
\newblock Imperfect bifurcation for the quasi-geostrophic shallow-water
  equations.
\newblock {\em Archive for Rational Mechanics and Analysis}, 231(3):1853--1915,
  2019.

\bibitem{DP92}
David~G. Dritschel and Lorenzo~M. Polvani.
\newblock The roll-up of vorticity strips on the surface of a sphere.
\newblock {\em Journal of Fluid Mechanics}, 234:47--69, 1992.

\bibitem{DP93}
David~G. Dritschel and Lorenzo~M. Polvani.
\newblock Wave and vortex dynamics on the surface of a sphere.
\newblock {\em Journal of Fluid Mechanics}, 255:35--64, 1993.

\bibitem{F00}
Ludwig~E. Fraenkel.
\newblock {\em An introduction to maximum principles and symmetry in elliptic
  problems}, volume 128 of {\em Cambridge Tracts in Mathematics}.
\newblock Cambridge University Press, Cambridge, 2000.

\bibitem{G20}
Claudia Garc\'ia.
\newblock K\'{a}rm\'{a}n vortex street in incompressible fluid models.
\newblock {\em Nonlinearity}, 33(4):1625--1676, 2020.

\bibitem{G21}
Claudia Garc\'ia.
\newblock Vortex patches choreography for active scalar equations.
\newblock {\em Journal of Nonlinear Science}, 31(5):Paper No. 75, 31, 2021.

\bibitem{GH22}
Claudia Garc\'ia and Susanna~V. Haziot.
\newblock Global bifurcation for corotating and counter-rotating vortex pairs.
\newblock {\em Communications in Mathematical Physics}, 402:1167--1204, 2023.

\bibitem{GHM22-1}
Claudia Garc\'ia, Taoufik Hmidi, and Joan Mateu.
\newblock Time {P}eriodic {S}olutions for 3{D} {Q}uasi-{G}eostrophic {M}odel.
\newblock {\em Communications in Mathematical Physics}, 390(2):617--756, 2022.

\bibitem{GHM22}
Claudia Garc\'ia, Taoufik Hmidi, and Joan Mateu.
\newblock Time periodic doubly connected solutions for the 3{D}
  quasi-geostrophic model.
\newblock {\em SIAM Journal on Mathematical Analysis}, 55(6):6133--6193, 2023.

\bibitem{GHM23}
Claudia Garc\'ia, Taoufik Hmidi, and Joan Mateu.
\newblock Time periodic solutions close to localized radial monotone profiles
  for the 2{D} euler equations,.
\newblock {\em Annals of PDE}, 10(1), 2024.

\bibitem{GHS20}
Claudia Garc\'ia, Taoufik Hmidi, and Juan Soler.
\newblock Non uniform rotating vortices and periodic orbits for the
  two-dimensional {E}uler equations.
\newblock {\em Archive for Rational Mechanics and Analysis}, 238(2):929--1085,
  2020.

\bibitem{GCGS20}
Ludovic Godard-Cadillac, Philippe Gravejat, and Didier Smets.
\newblock Co-rotating vortices with n fold symmetry for the inviscid surface
  quasi-geostrophic equation.
\newblock {\em Indiana University Mathematics Journal}, 2020.

\bibitem{GSIP23}
Javier G\'omez~Serrano, Alexandru~D. Ionescu, and Jaemin Park.
\newblock Quasiperiodic solutions of the generalized {SQG} equation.
\newblock {\em arXiv preprint arXiv:2303.03992.}, 2023.

\bibitem{GSPS:2021}
Javier G\'omez-Serrano, Jaemin Park, and Jia Shi.
\newblock Existence of non-trivial non-concentrated compactly supported
  stationary solutions of the 2d euler equation with finite energy.
\newblock {\em arXiv preprint arXiv:2112.03821}, 2021.

\bibitem{GSPSY-rigidity}
Javier G\'{o}mez-Serrano, Jaemin Park, Jia Shi, and Yao Yao.
\newblock Symmetry in stationary and uniformly rotating solutions of active
  scalar equations.
\newblock {\em Duke Mathematical Journal}, 170(13):2957--3038, 2021.

\bibitem{GR15}
I.~S. Gradshteyn and I.~M. Ryzhik.
\newblock {\em Table of integrals, series, and products}.
\newblock Elsevier/Academic Press, Amsterdam, eighth edition, 2015.
\newblock Translated from the Russian, Translation edited and with a preface by
  Daniel Zwillinger and Victor Moll, Revised from the seventh edition
  [MR2360010].

\bibitem{HH13}
Gregory~J. Hakim and James~R. Holton.
\newblock {\em An introduction to dynamic meteorology}.
\newblock International Geophysics Series. Elsevier Academic Press, Burlington,
  MA, 4 edition, 2004.

\bibitem{HH15}
Zineb Hassainia and Taoufik Hmidi.
\newblock On the {V}-states for the generalized quasi-geostrophic equations.
\newblock {\em Communications in Mathematical Physics}, 337(1):321--377, 2015.

\bibitem{HH21}
Zineb Hassainia and Taoufik Hmidi.
\newblock Steady asymmetric vortex pairs for {E}uler equations.
\newblock {\em Discrete and Continuous Dynamical Systems. Series A},
  41(4):1939--1969, 2021.

\bibitem{HHM21}
Zineb Hassainia, Taoufik Hmidi, and Nader Masmoudi.
\newblock Kam theory for active scalar equations.
\newblock {\em To appear in Memoirs of the American Mathematical Society,
  available at arXiv:2110.08615}, 2021.

\bibitem{HHR23}
Zineb Hassainia, Taoufik Hmidi, and Emeric Roulley.
\newblock Invariant kam tori around annular vortex patches for 2{D} {E}uler
  equations.
\newblock {\em Communications in Mathematical Physics}, 405(270):1--127, 2024.

\bibitem{HMW20}
Zineb Hassainia, Nader Masmoudi, and Miles~H. Wheeler.
\newblock Global bifurcation of rotating vortex patches.
\newblock {\em Communications on Pure and Applied Mathematics},
  73(9):1933--1980, 2020.

\bibitem{HR22}
Zineb Hassainia and Emeric Roulley.
\newblock Boundary effects on the emergence of quasi-periodic solutions for
  euler equations.
\newblock {\em arXiv preprint arXiv:2202.10053}, 2022.

\bibitem{HW22}
Zineb Hassainia and Miles~H. Wheeler.
\newblock Multipole vortex patch equilibria for active scalar equations.
\newblock {\em SIAM Journal on Mathematical Analysis}, 54(6):6054--6095, 2022.

\bibitem{H15}
Taoufik Hmidi.
\newblock On the trivial solutions for the rotating patch model.
\newblock {\em Journal of Evolution Equations}, 15(4):801--816, 2015.

\bibitem{HM16-1}
Taoufik Hmidi and Joan Mateu.
\newblock Bifurcation of rotating patches from {K}irchhoff vortices.
\newblock {\em Discrete and Continuous Dynamical Systems. Series A},
  36(10):5401--5422, 2016.

\bibitem{HM16}
Taoufik Hmidi and Joan Mateu.
\newblock Degenerate bifurcation of the rotating patches.
\newblock {\em Advances in Mathematics}, 302:799--850, 2016.

\bibitem{HM17}
Taoufik Hmidi and Joan Mateu.
\newblock Existence of corotating and counter-rotating vortex pairs for active
  scalar equations.
\newblock {\em Communications in Mathematical Physics}, 350(2):699--747, 2017.

\bibitem{HMV13}
Taoufik Hmidi, Joan Mateu, and Joan Verdera.
\newblock Boundary regularity of rotating vortex patches.
\newblock {\em Archive for Rational Mechanics and Analysis}, 209(1):171--208,
  2013.

\bibitem{HMV15}
Taoufik Hmidi, Joan Mateu, and Joan Verdera.
\newblock On rotating doubly connected vortices.
\newblock {\em Journal of Differential Equations}, 258(4):1395--1429, 2015.

\bibitem{HR21}
Taoufik Hmidi and Emeric Roulley.
\newblock Time quasi-periodic vortex patches for quasi-geostrophic
  shallow-water equations.
\newblock {\em To appear in M{\'e}moires de la Soci{\'e}t{\'e} Math{\'e}matique
  de France, availble at arXiv:2110.13751}, 2021.

\bibitem{HXX22}
Taoufik Hmidi, Liutang Xue, and Zhi Xue.
\newblock Emergence of time periodic solutions for the generalized surface
  quasi-geostrophic equation in the disc.
\newblock {\em arXiv preprint arXiv:2210.08760}, 2022.

\bibitem{K18}
Sun-Chul Kim.
\newblock A free-boundary problem for {E}uler flows with constant vorticity on
  the sphere.
\newblock {\em Journal of Mathematical Analysis and Applications},
  465(1):703--711, 2018.

\bibitem{KSS18}
Sun-Chul Kim, Takashi Sakajo, and Sung-Ik Sohn.
\newblock Stability of barotropic vortex strip on a rotating sphere.
\newblock {\em Proceedings A}, 474(2210):20170883, 25, 2018.

\bibitem{KS21}
Sun-Chul Kim and Sung-Ik Sohn.
\newblock Linear stability and nonlinear evolution of a polar vortex cap on a
  rotating sphere.
\newblock {\em European Journal of Mechanics. B. Fluids}, 85:102--109, 2021.

\bibitem{K74}
Gustav~R. Kirchhoff.
\newblock {\em Vorlesungen uber mathematische Physik. Mechanik.}
\newblock Teubner, Leipzig, 1876.

\bibitem{N22}
Marc Nualart.
\newblock On zonal steady solutions to the 2{D} {E}uler equations on the
  rotating unit sphere.
\newblock {\em Nonlinearity}, 36(9):4981--5006, 2023.

\bibitem{R21}
Emeric Roulley.
\newblock Vortex rigid motion in quasi-geostrophic shallow-water equations.
\newblock {\em Asymptotic Analysis}, pages 1--50, 2022.

\bibitem{R22}
Emeric Roulley.
\newblock Periodic and quasi-periodic {E}uler-$\alpha$ flows close to {R}ankine
  vortices.
\newblock {\em Dynamics of Partial Differential Equations}, 20(4):311--366,
  2023.

\bibitem{S17}
Yuri~N. Skiba.
\newblock {\em Mathematical problems of the dynamics of incompressible fluid on
  a rotating sphere}.
\newblock Springer, Cham, 2017.

\bibitem{T16}
Michael Taylor.
\newblock Euler equation on a rotating surface.
\newblock {\em Journal of Functional Analysis}, 270(10):3884--3945, 2016.

\bibitem{T85}
Bruce Turkington.
\newblock Corotating steady vortex flows with n-fold symmetry.
\newblock {\em Nonlinear Analysis. Theory, Methods \& Applications},
  9:351--369, 1985.

\bibitem{WXZ22}
Yuchen Wang, Xin Xu, and Maolin Zhou.
\newblock Degenerate bifurcation of two-fold doubly-connected vortex patches.
\newblock {\em arXiv preprint arXiv:2212.01869}, 2022.

\bibitem{Y63}
Victor~I Yudovich.
\newblock Non-stationary flow of an ideal incompressible liquid.
\newblock {\em USSR Computational Mathematics and Mathematical Physics},
  3(6):1407--1456, 1963.

\end{thebibliography}
	\bibliographystyle{plain}
	
	\noindent Departamento de Matem\'atica Aplicada \& Research Unit ``Modeling Nature'' (MNat), Facultad de Ciencias, Universidad de Granada, 18071 Granada, Spain.\\
	Email address : claudiagarcia@ugr.es.\\
	
	\noindent NYUAD Research Institute, New York University Abu Dhabi, PO Box 129188, Abu Dhabi, United Arab Emirates.\\
	Email address : zh14@nyu.edu.\\
	
	\noindent SISSA International School for Advanced Studies, Via Bonomea 265, 34136, Trieste, Italy.\\
	E-mail address : eroulley@sissa.it.
\end{document}